\documentclass[11pt]{article}

\usepackage[letterpaper, left=1truein, right=1truein, top = 1truein, bottom = 1truein]{geometry}

\RequirePackage{amsthm,amsmath,amsfonts,amssymb}
\RequirePackage[numbers,sort&compress]{natbib}
\RequirePackage[colorlinks,citecolor=blue,urlcolor=blue]{hyperref}
\RequirePackage{graphicx}

\usepackage{dutchcal}
\usepackage{subfigure}

\theoremstyle{plain}

\newtheorem{theorem}{Theorem}[section]
\newtheorem{lemma}[theorem]{Lemma}
\newtheorem{proposition}[theorem]{Proposition}
\theoremstyle{remark}
\newtheorem{definition}[theorem]{Definition}

\newtheorem{corollary}[theorem]{Corollary}
\newtheorem{remark}[theorem]{Remark}
\newtheorem{assumption}[theorem]{Assumption}

\def\R{\mathbb{R}}

\def\tg{\tilde{g}}
\def\C{\mathbb{C}}
\def\eps{\varepsilon}

\def\der{\mathrm{d}}

\def\G{\mathsf{G}}

\def\SO{\mathsf{SO}}

\def\E{\mathbb{E}}

\def\orbit{\mathcal{O}}

\def\der{\mathrm{d}}

\newcommand{\T}{\mathsf{T}}
\def\1{\mathbf{1}}
\def\E{\mathbb{E}}

\def\td{\tilde{d}}

\def\frakg{\mathfrak{g}}

\DeclareMathOperator{\Id}{Id}

\DeclareMathOperator{\diag}{diag}
\DeclareMathOperator{\rank}{rank}

\DeclareMathOperator*{\argmin}{arg\,min}
\DeclareMathOperator*{\argmax}{arg\,max}
\renewcommand{\i}{\mathbf{i}}
\renewcommand{\P}{\mathbb{P}}

\usepackage{caption}
\title{Misspecified Maximum Likelihood Estimation for Non-Uniform Group Orbit Recovery}

\author{Sheng Xu\thanks{Program in Applied and Computational Mathematics, Princeton University, \href{mailto:sxu21@princeton.edu}{sxu21@princeton.edu}},  Anderson Ye Zhang\thanks{Department of Statistics and Data Science, University of Pennsylvania, \href{mailto:ayz@wharton.upenn.edu}{ayz@wharton.upenn.edu}}, and Amit Singer\thanks{Program in Applied and Computational Mathematics and Department of Mathematics, Princeton University, \href{mailto:amits@math.princeton.edu}{amits@math.princeton.edu}}}

\date{}

\begin{document}
\maketitle

\let\thefootnote\relax

\begin{abstract}
We study maximum likelihood estimation (MLE) in the generalized group orbit recovery model, where each observation is generated by applying a random group action and a known, fixed linear operator to an unknown signal, followed by additive noise. This model is motivated by single-particle cryo-electron microscopy (cryo-EM) and can be viewed primarily as a structured continuous Gaussian mixture model. In practice, signal estimation is often performed by marginalizing over the group using a uniform distribution—an assumption that typically does not hold and renders the MLE misspecified. This raises a fundamental question: how does the misspecified MLE perform? We address this question from several angles. First, we show that in the absence of projection, the misspecified population log-likelihood has desired optimization landscape that leads to correct signal recovery. In contrast, when projections are present, the global optimizers of the misspecified likelihood deviate from the true signal, with the magnitude of the bias depending on the noise level. To address this issue, we propose a joint estimation approach tailored to the cryo-EM setting, which parameterizes the unknown distribution of the group elements and estimates both the signal and distribution parameters simultaneously.
\end{abstract}

\section{Introduction}\label{sec:introduction}
We study the problem of recovering a signal from noisy, randomly transformed and projected observations, where each transformation belongs to a known compact subgroup of the orthogonal group. Formally, let $\theta_*\in\R^d$ be an unknown signal of interest and $\G\subset \mathsf{O}(d)$ be a known compact subgroup of the orthogonal group of dimension $d$. Consider the \textit{generalized group orbit recovery} model as follows, where we observe $n$ i.i.d. samples:
\begin{equation}\label{eq:projectedmodel}
Y_i=P  g_i \theta_* +\sigma \eps_i \in \R^{\td}, \qquad i=1,\ldots,n.
\end{equation}
Here, $P \in\R^{\td\times d}$ is a known linear map from $\R^d$ to $\R^{\tilde{d}}$, which typically represents a projection. This model is therefore also known as \textit{projected group orbit recovery} \cite{abbe2018estimation, bandeira2023estimation, fan2023likelihood, fan2024maximum}. The matrices $g_1,\ldots,g_n \overset{iid}{\sim} \Lambda_*$ are unknown random orthogonal matrices in $\G$, and
$\eps_1,\ldots,\eps_n \overset{iid}{\sim} \mathrm{N}(0,\Id_{\td})$ are
Gaussian noise vectors in the observation space of dimension $\td$. The probability measure $\Lambda_*$, known as the \textit{group measure}, represents an unknown (possibly non-Haar) probability measure on $\G$. Typically, the group element $g\in\G$ corresponds to an irreducible matrix representation of some intrinsic group element $\frakg$, as will be illustrated in the examples below. Our goal is to estimate the underlying true orbit $\orbit_{\theta_*}:=\{g\theta_*:g\in \G\}$ given the observed samples $Y_1,\ldots,Y_n$ without knowledge of the latent group measure $\Lambda_*$.

A primary motivation for studying the generalized orbit recovery model arises from its close connection to single-particle cryo-electron microscopy (cryo-EM), a widely used imaging technique in structural biology \cite{dubochet1988cryo, henderson1990model, frank2006three}. Cryo-EM is used to reconstruct a three-dimensional molecular structure—typically represented as an electrostatic potential function defined on $\R^3$—from numerous two-dimensional projection images. Each image corresponds to a tomographic projection of the molecule taken at an unknown and randomly distributed orientation in 3-D. In this setting, the linear operator $P$ corresponds to the physical process of tomographic projection from 3-D to 2-D, typically modeled by the X-ray transform\footnote{More precisely, in cryo-EM each image is affected by a distinct projection operator \( P_i \), which incorporates not only the tomographic projection from 3-D to 2-D but also the effect of the microscope's contrast transfer function (CTF), which varies across images.}.

A standard approach to reconstructing the molecular structure is to perform maximum likelihood estimation (MLE), often combined with regularization and implemented through iterative algorithms such as expectation-maximization (EM) algorithm or stochastic gradient descent \cite{sigworth1998maximum, scheres2012relion, punjani2017cryosparc}. A key component of these methods is the need to specify a prior distribution over the latent orientations of the particles. Since the true orientation distribution is typically unknown and is generally not accessible as prior knowledge for reconstruction, it is common in applications to assume a uniform (Haar) distribution over the rotation group $\SO(3)$. While this modeling choice simplifies computation, it introduces a potential source of \textit{misspecification}, especially when the true distribution is non-uniform due to experimental factors such as preferred particle orientations or sample preparation artifacts \cite{taylor2008retrospective, liu2013deformed, tan2017addressing, glaeser2017opinion, noble2018routine, carragher2019current, lyumkis2019challenges, baldwin2020non}.

Despite its widespread use in practice, the statistical and computational consequences of using a misspecified group distribution, particularly the Haar measure, remain poorly understood. Fundamental questions—such as whether this misspecification leads to bias in the reconstructed structure, how such bias interacts with the noise level, and whether consistent recovery is still achievable—have not been systematically studied. This paper aims to address these questions by initiating a theoretical study of orbit recovery under unknown or misspecified group distributions. By analyzing the effects of replacing the true latent group measure with an incorrect or idealized alternative (such as the Haar measure), we seek to understand when and how such approximations affect estimation accuracy.

The main contributions of this paper are summarized as follows:

\begin{enumerate}

\item{\emph{Misspecified MLE Works Without Projection.}} In the absence of projection (i.e., when the observations are full-dimensional, $P=\Id$), we show that the misspecified negative population log-likelihood—specifically using the Haar measure—has a landscape that is invariant to the true underlying group measure, even when the true measure has local support. Remarkably, this misspecified objective still exhibits a favorable landscape in which the true signal is a global minimizer. This leads to a surprising and practically important conclusion: in the non-projected setting, the misspecified MLE achieves exact recovery without knowledge of the true group distribution.

\vspace{0.8mm}

\item{\emph{Bias under Projection.}} When projection is present, we show that the global minimizers of the misspecified negative population log-likelihood generally do not coincide with the true orbit. That is, the estimated orbit under a misspecified group distribution may be systematically biased away from the true one. We quantify this deviation and show that its magnitude depends on the noise level $\sigma$, vanishing as $\sigma$ goes to $0$. These findings provide theoretical justification for the use of misspecified MLE with an idealized Haar prior in the low-noise regime, where the induced bias becomes negligible and accurate recovery remains achievable. Notably, the precise threshold for this regime depends on the interaction between the noise level and the underlying group measure, and may in practice require extremely low-noise levels. However, in the high-noise regime—which is particularly relevant in cryo-EM—this bias can become substantial \cite{liu2025overcoming}. Our results suggest that relying on a uniform prior in such settings can lead to systematically distorted reconstructions. Moreover, standard validation techniques—such as the Fourier shell correlation on disjoint subsets of data \cite{harauz1986exact, saxton1982correlation}, which has become the universal resolution metric for assessing 3-D reconstruction quality in cryo-EM \cite{rosenthal2003optimal, scheres2012prevention}—may fail to detect this bias if it is shared across reconstructions, underscoring the importance of modeling the orientation distribution accurately in practice.

\vspace{0.8mm}

\item{\emph{Joint Estimation with Parameterized Group Measure.}} To address the bias introduced by misspecification and to handle the unknown nature of $\Lambda_*$, we propose a joint estimation approach that parameterizes the group measure. In the context of cryo-EM, where $\Lambda_*$ is a distribution on $\SO(3)$, we represent the group measure via a basis expansion over the entries of \textit{real} Wigner D-matrices with coefficients $\mathcal{B}_*$. This choice is natural and theoretically justified by Peter-Weyl Theorem, which ensures these entries form a complete orthogonal basis for square-integrable functions on $\SO(3)$. Intuitively, this is analogous to representing a smooth function on the sphere via spherical harmonics, or a function on the circle via Fourier modes. The resulting formulation leads to a joint likelihood involving both the signal $\theta$ and the coefficients $\mathcal{B}$, and we further establish consistency of the joint MLE. While other parameterizations are possible, the Wigner matrix-based expansion offers a general and flexible framework that may be applicable to other problems involving distributions on $\SO(3)$ and related groups.
\end{enumerate}

\subsection{Related Literature} A natural perspective on group orbit recovery is to view it as a structured Gaussian mixture model, where the components lie along a group orbit and the mixing distribution $\Lambda_*$ is supported on a (typically continuous) compact group. This viewpoint connects orbit recovery to a broad class of latent variable models, highlighting the role of symmetry and group structure in shaping both statistical limits and algorithmic behavior. While relevant to both discrete and continuous groups, our focus is on the continuous setting, as motivated by applications such as cryo-EM.

\paragraph*{Group Orbit Recovery} The group orbit recovery model has emerged as a unifying framework for a range of statistical estimation problems where signals are observed under latent group transformations, including multi-reference alignment (MRA) \cite{perry2019sample, abbe2018multireference, bandeira2020optimal, brunel2019learning, fan2023likelihood, romanov2021multi}, cryo-EM \cite{bandeira2023estimation, abbe2018estimation, bendory2019multi, sharon2020method, bandeira2020non}, and other problems with latent group structure \cite{pumir2021generalized}.

A central line of research has investigated the sample complexity of orbit recovery, showing that latent group actions can substantially increase the difficulty of estimation. For example, in discrete MRA, \cite{perry2019sample} established that the minimax estimation error scales as $\sigma^3$ in the high-noise regime, in contrast to the usual $\sigma$ scaling in standard models without hidden group structure. Subsequent work extended these results to continuous settings, establishing sample complexity bounds under broader classes of group actions \cite{abbe2018multireference, bandeira2020optimal, bandeira2023estimation}, as well as under additional structural assumptions such as sparsity \cite{bendory2023autocorrelation, bendory2024sample}.

Recent work has studied the nonconvex geometry of the log-likelihood in orbit recovery, revealing deep connections between statistical behavior, optimization landscape, and group invariants. For discrete groups without projection, \cite{fan2023likelihood} showed that the landscape is benign at low noise but may exhibit spurious local optima at high noise. \cite{fan2024maximum} extended this analysis to continuous groups with potential projection, relating critical points of the log-likelihood to a sequence of moment optimization problems. In parallel, \cite{katsevich2023likelihood} studied low-SNR Gaussian mixtures and showed that moment matching approximates likelihood optimization, revealing a shared structure across estimation methods.

Most prior theoretical work on orbit recovery assumes the latent group elements follow the Haar (uniform) measure, an assumption dating back to Kam’s seminal method-of-moments approach for cryo-EM \cite{kam1980reconstruction}. While simplifying analysis and computation, this assumption often fails in applications like cryo-EM, where orientation distributions can be highly non-uniform. Some prior efforts have explored non-uniformity from algorithmic or numerical perspectives \cite{abbe2018multireference, sharon2020method, bendory2022dihedral}, but the statistical consequences of such misspecification remain underexplored. This work initiates a systematic study of estimation under misspecified group distributions. We analyze how using an incorrect prior—such as assuming uniformity—affects signal recovery, especially in the presence of projection, and propose a joint estimation framework to address these challenges.

\paragraph*{Misspecification} The misspecified MLE—also known as the quasi-maximum likelihood estimator or pseudo-likelihood estimator—is known to minimize the Kullback-Leibler divergence between the true model and the misspecified model class. Beyond likelihood-based approaches, Bayesian methods under model misspecification have also been studied; see, for example, \cite{kleijn2012bernstein} and \cite{wang2019variational}, which analyzed posterior contraction and variational inference under misspecified models.

Much of the existing literature on misspecified MLE focuses on the asymptotic properties of the method, including consistency and limiting distributions as the sample size grows, as in \cite{van2000asymptotic,white1996estimation}.  In certain settings, such as community detection where the parameter of interest is discrete, the method can still recover the true parameter exactly despite misspecification \cite{amini2013pseudo}. However, in general, the estimator is biased. 

For discrete Gaussian mixture models, \cite{dwivedi2018theoretical} and \cite{dwivedi2020singularity} examined several simple settings—such as a three-component Gaussian mixture on $\mathbb{R}$ misspecified by a symmetric two-component mixture—where the bias of the misspecified MLE can be characterized explicitly. These works also analyzed the convergence behavior of the EM algorithm under model misspecification. Nevertheless, a general understanding of bias in misspecified Gaussian mixtures remains limited. Our work extends this line of research by investigating bias phenomena in a more general and structured setting, namely group orbit recovery model that corresponds to continuous Gaussian mixtures supported on group orbits. We uncover a surprising dichotomy: when there is no projection, the MLE exhibits no asymptotic bias, even under severely misspecified group measure; in contrast, in the presence of projection, the bias becomes nonzero and depends on the signal-to-noise ratio.

\subsection{Organization}
The paper is organized as follows. Section~\ref{sec:setup} presents the formal setup of the model, along with motivating examples, a discussion of identifiability, and the correctly specified MLE. Section~\ref{sec:landscape_without_projection} analyzes the optimization landscape of the misspecified MLE in the absence of projection. Section~\ref{sec:landscape_with_projection} investigates the setting with projection. In Section~\ref{sec:joint_MLE}, we introduce a parameterization of the rotation distribution and develop a joint MLE framework for cryo-EM applications. Section~\ref{sec:simulations} includes numerical experiments. Section~\ref{sec:discussion} is a discussion section. Due to space constraints, additional examples, all proofs (except that of Theorem~\ref{thm:withoutProj}), and further technical details on spherical harmonics and representations of rotations are provided in the appendix.

\subsection{Notations}
The symbol $:=$ denotes ``is defined as'' and the symbol $\cong$ denotes isomorphism. 
Given square matrices $A_1 \in \mathbb{R}^{d_1 \times d_1}, A_2 \in \mathbb{R}^{d_2 \times d_2}, \dots, A_n \in \mathbb{R}^{d_n \times d_n}$, their \emph{direct sum} is the block-diagonal matrix defined as
\[
\bigoplus_{i=1}^n A_i := 
\begin{pmatrix}
A_1 & 0 & \cdots & 0 \\
0 & A_2 & \cdots & 0 \\
\vdots & \vdots & \ddots & \vdots \\
0 & 0 & \cdots & A_n
\end{pmatrix}
\in \mathbb{R}^{d \times d},
\]
where $d = \sum_{i=1}^n d_i$. Each $A_i$ occupies a diagonal block, and all off-diagonal blocks are zero matrices of appropriate dimensions. We denote $\SO(d)$ as the special orthogonal group of dimension $d$, $\Id$ as the identity matrix, $\|\cdot\|$ as the Euclidean distance for vectors and the operator norm for matrices, $\mathcal{S}^1$ and $\mathcal{S}^2$ as  the unit circle and unit sphere, $\i =\sqrt{-1}$ as the imaginary unit. For two random vectors $X,Y$ of the same dimension, $X\overset{d}{=}Y$ means that they have the same distribution. 

\section{Model Setup and Preliminaries}\label{sec:setup}

This section lays the foundation for our theoretical analysis for formalizing the likelihoods functions and maximum likelihood estimation objectives. In Section \ref{sec:likelihoods}, we introduce the likelihood objectives under both correctly specified and misspecified group measures. Section \ref{sec:examples} illustrates how these expressions arise in concrete applications such as multi-reference alignment and cryo-EM. In Section \ref{sec:identification}, we discuss the identifiability of the model and clarify the notion of the underlying orbit we aim to recover. Finally, Section~\ref{sec:correct_MLE} analyzes the global minimizers of the correctly specified likelihood as a baseline, and sets the stage for Section~\ref{sec:theory}, which turns to the misspecified setting.

\subsection{Likelihood Formulations: Correctly Specified and Misspecified Models}\label{sec:likelihoods}
From (\ref{eq:projectedmodel}), the probability density or likelihood function for $Y=y$ in the generalized orbit recovery model takes the form:
\begin{equation}\label{eq:likelihood}
p_\theta(y;\Lambda_*,\sigma^2)=\int_{\G} \frac{1}{(2\pi\sigma^2)^{\td/2}}
\exp\left(-\frac{\|y-P  g  \theta\|^2}{2\sigma^2}\right)\der
\Lambda_*(g),
\end{equation}
which is the Gaussian mixture density for $Y$, obtained by marginalizing over the unknown rotation $g\sim \Lambda_*$. Here, $\Lambda_*$ and $\sigma^2$ are explicitly included to highlight their dependence. We will refer to this as the \textit{correctly specified} likelihood function in the sequel. In practice, the underlying group measure $\Lambda_*$ is rarely known. A common approach is to replace $\Lambda_*$ by the Haar measure on $\G$, denoted by $\Lambda_0$, yielding the \textit{misspecified} likelihood (see references in the introduction):
\begin{align*}
p_\theta(y;\Lambda_0,\sigma^2)=\int_{\G} \frac{1}{(2\pi\sigma^2)^{\td/2}}
\exp\left(-\frac{\|y-P  g  \theta\|^2}{2\sigma^2}\right)\der
\Lambda_0(g).
\end{align*}
Although one could in principle substitute $\Lambda_*$ with other group measures, the Haar measure is the most widely used surrogate in practice. Therefore, unless otherwise stated, we will use the term misspecified likelihood to specifically refer to the surrogate model based on the Haar measure.

To facilitate theoretical analysis of maximum likelihood methods under both correctly specified and misspecified settings, we define the corresponding negative \textit{population log-likelihood}:
\begin{align}\label{eq:population_likelihood}
\ell(\theta;\Lambda,\sigma^2) := -\E_{\theta_*, \Lambda_*} \log p_\theta(y; \Lambda,\sigma^2),~~~\text{for}~\Lambda=\Lambda_*~\text{and}~\Lambda_0.
\end{align}
Here, the expectations are taken under the true underlying data-generating process, which follows the model in \eqref{eq:projectedmodel}. While both functions depends implicitly on $\theta_*$ and $\Lambda_*$ via the expectation, we omit this dependence for notational simplicity here. In Section \ref{sec:theory}, we will make this dependence explicit where necessary for theoretical analysis.  

These population log-likelihoods represents the large sample limits of their empirical counterparts:
\begin{align}\label{eq:sample_likelihood}
\ell_n(\theta;\Lambda,\sigma^2):=-\frac{1}{n}\sum_{i=1}^n \log p_\theta(Y_i; \Lambda, \sigma^2),~~~\text{for}~\Lambda=\Lambda_*~\text{and}~\Lambda_0.
\end{align}
We are particularly interested in analyzing the landscape of these negative \textit{population log-likelihood} functions—both the correctly specified and the misspecified versions—and their key differences, especially focusing on their global minimizers. For notational and conceptual convenience, we will refer to the global minimizers of the negative \textit{population} log-likelihoods as the maximum likelihood estimators (MLE). While this constitutes a slight abuse of terminology—since the term MLE typically refers to the minimizers of the \textit{empirical} likelihood in \eqref{eq:sample_likelihood}—we adopt this convention throughout, given our focus on asymptotic behavior in the large-sample regime.

\subsection{Motivating Examples}\label{sec:examples}

We introduce several main examples within the framework of the generalized orbit recovery model. In the main text, we primarily discuss two representative cases: symmetric two-component Gaussian mixture and continuous MRA. We then present the likelihood functions for more involved examples—spherical registration, unprojected cryo-EM and cryo-electron tomography (cryo-ET), and projected cryo-EM—with full details deferred to Appendix~\ref{sec:additional_examples} due to space constraints. While the projected cryo-EM example serves as a primary motivation for this work, its complete formulation is provided in the appendix to streamline the exposition. For all these examples, we primarily follow the setup presented in \cite{fan2024maximum}; see also \cite{bandeira2023estimation}. Minor adjustments have been made to ensure a concise and self-contained presentation.

\subsubsection{Symmetric Two-Component Gaussian Mixture}\label{sec:model_2GM}
$P =\Id$ is the identity map. $\G$ is isomorphic to the intrinsic group $\mathbb{Z}_2=\{+1, -1\}$. The group action on $\theta$ is given by
\begin{align}\label{eq:2Gaussian}
g \theta = \begin{cases}
    \theta & \text{if } \frakg=+1,\\
    -\theta & \text{if } \frakg=-1.
\end{cases}
\end{align}
The likelihood function is exactly the marginal probability density for symmetric two-component Gaussian mixture
\begin{align}\label{eq:likelihood_2G}
p_{\theta}(y) =&\int_{\mathbb{Z}_2} \frac{1}{(2\pi\sigma^2)^{d/2}}
\exp\left(-\frac{\|y-\frakg\theta\|^2}{2\sigma^2}\right) \der \lambda_*(\frakg) \notag\\
=&\pi \cdot \frac{1}{(2\pi\sigma^2)^{d/2}}
\exp\left(-\frac{\|y-\theta\|^2}{2\sigma^2}\right)+(1-\pi) \cdot \frac{1}{(2\pi\sigma^2)^{d/2}}
\exp\left(-\frac{\|y+\theta\|^2}{2\sigma^2}\right),
\end{align}
where $\lambda_*$ is the Bernoulli distribution on $\mathbb{Z}_2=\{+1,-1\}$ with probability $\pi$ on $+1$. In other words, $\pi$ corresponds to the mixture weight for component $\theta$.

\subsubsection{Continuous MRA}\label{sec:model_MRA}
We consider estimating a periodic function on the unit circle $f:\mathcal{S}^1\to \R$, and identify $\mathcal{S}^1$ with $[0,2\pi)$. Assume $f\in L_2(\mathcal{S}^1, \R)$ and admits a bandlimited representation under the real Fourier basis (see \cite[Section 3]{fan2024maximum}):
\begin{align*}
f(t) = \theta^{(0)} + \sum_{l=1}^L \theta^{(l)}_1 \sqrt{2}\cos  lt + \sum_{l=1}^L \theta^{(l)}_2 \sqrt{2}\sin lt,  \quad\text{ for }t\in[0, 2\pi),
\end{align*}
where $L\geq 1$ is the bandlimit. We represent the rotation of $f$ by an element $\frakg\in \mathsf{SO}(2)\cong [0,2\pi)$ as $f_\frakg(t) = f(t+\frakg \text{ mod } 2\pi)$. For the case without projection, each observation is a realization of the rotated function $f_\frakg$ corrupted by white noise, $f_\frakg(t) \der t+\sigma \der W(t)$, where $\frakg\sim \lambda_*$ for some possibly non-Haar measure $\lambda_*$ on $\mathsf{SO}(2)$ and $\der W(t)$ denotes a standard Gaussian white noise process on $[0,2\pi)$. Writing 
\begin{align*}
\theta = (\theta^{(0)}, \theta^{(1)}_1, \theta^{(1)}_2, \ldots, \theta^{(L)}_1, \theta^{(L)}_2)\in\R^d,\quad d=2L+1,
\end{align*}
for the vector of Fourier coefficients, the rotation $\frakg$ acting on $f$ can be represented in the space of these coefficients by the matrix multiplication $\theta \mapsto g\theta$ where $g\in \G$ is the block-diagonal matrix 
\begin{align}
g = \mathcal{H}(\frakg)=\bigoplus_{l=0}^L \mathcal{H}^{(l)}(\frakg)\in \R^{d\times d}, \quad\text{ for }\frakg\in[0, 2\pi),\label{eq:MRA_rotation}
\end{align}
with 
\begin{align}\label{eq:MRA_rotation_detail}
\mathcal{H}^{(0)}(\frakg)=1 \text{ and } \mathcal{H}^{(l)}(\frakg)= \begin{pmatrix}
\cos l\frakg & -\sin l\frakg \\
\sin l\frakg & \cos l\frakg
\end{pmatrix}.
\end{align}
The observation model for the Fourier coefficients of $f$ is then a special case of \eqref{eq:projectedmodel} where $P =\Id$, $\G$ is isomorphic to $\mathsf{SO}(2)\cong [0,2\pi)$, and $\Lambda_*$ is the unique probability measure on $\G$ induced by the isomorphism between $\G$ and $[0,2\pi)$, with respect to $\lambda_*$. 

For the case with a two-fold projection, we consider the observations
\begin{align*}
(\Pi \circ f_\frakg)(t) \der t+\sigma \der W(t),
\end{align*}
where $(\Pi \circ f_\frakg)(t)= f_\frakg(t)+f_\frakg(1-t)$ representing the two-fold projection of $\mathcal{S}^1$ onto $(0,\pi)$ and $\der W(t)$ is a standard Gaussian white noise process on $(0,\pi)$. Expressing $\Pi \circ f$ in the space of Fourier coefficients, $\Pi$ corresponds to a linear map $P:\R^d\to\R^{\td}$ for $\td= L+1$ where $P  \theta = \sqrt{2} (\theta^{(0)},\theta^{(1)}_1,\ldots, \theta^{(L)}_1)$ (see \cite[Appendix C.3]{fan2024maximum}).

For both cases, the likelihood function can thus be written as 
\begin{align}\label{eq:likelihood_CMRA}
    p_{\theta}(y) = 
\int_0^{2\pi} \frac{1}{(2\pi\sigma^2)^{\td/2}}
\exp\left(-\frac{\|y-P \mathcal{H}(\frakg)\theta\|^2}{2\sigma^2}\right)\der
\lambda_*(\frakg),
\end{align}
where $P=\Id$ for the case without projection.

\subsubsection{Spherical Registration}\label{sec:model_SR}
We consider estimating a function on the unit sphere $\mathcal{S}^2$. The likelihood function can be written as 
\begin{align}\label{eq:likelihood_SR}
    p_{\theta}(y) = 
\int_{\mathsf{SO}(3)} \frac{1}{(2\pi\sigma^2)^{d/2}}
\exp\left(-\frac{\|y- \mathcal{D}(\frakg)\theta\|^2}{2\sigma^2}\right)\der
\lambda_*(\frakg).
\end{align}
See Appendix \ref{sec:model_SR_appendix} for details of the model and the definition of $\mathcal{D}$.

\subsubsection{Unprojected Cryo-EM and Cryo-ET}\label{sec:model_cryoET}

We consider estimating a function on $\R^3$ and the action of $\mathsf{SO}(3)$ on $\R^3$ is given by rotation about the origin. The likelihood function can be written as 
\begin{align}\label{eq:likelihood_cryoET}
    p_{\theta}(y) = 
\int_{\mathsf{SO}(3)} \frac{1}{(2\pi\sigma^2)^{d/2}}
\exp\left(-\frac{\|y- \check{\mathcal{D}}(\frakg)\theta\|^2}{2\sigma^2}\right)\der
\lambda_*(\frakg).
\end{align}
See Appendix \ref{sec:model_cryoET_appendix} for details of the model and the definition of $\check{\mathcal{D}}$.

\subsubsection{(Projected) Cryo-EM}\label{sec:model_cryoEM}

We extend the model from the previous section to include the tomographic projection that occurs in the practice of cryo-EM. As before, we aim to estimate a function on $\R^3$. In this projected model, the signal undergoes a rotation in $\mathsf{SO}(3)$ about the origin, followed by the tomographic projection. The observed samples are on the projection domain $\R^2$.
The likelihood function can be written as 
\begin{align}\label{eq:likelihood_cryoEM}
    p_{\theta}(y) = 
\int_{\mathsf{SO}(3)} \frac{1}{(2\pi\sigma^2)^{\td/2}}
\exp\left(-\frac{\|y- P\check{\mathcal{D}}(\frakg)\theta\|^2}{2\sigma^2}\right)\der
\lambda_*(\frakg).
\end{align}
See Appendix \ref{sec:model_cryoEM_appendix} for details of the model and the definitions of $P,\check{\mathcal{D}}$.

\subsection{Identifiability}\label{sec:identification}

According to the model~\eqref{eq:projectedmodel}, the parameter of interest $\theta_*$ is identifiable up to the distribution of the mixture centers $P g \theta_*$, where $g$ follows some possibly non-Haar probability measure $\Lambda_*$ over the compact group $\G$. 
To more precisely characterize the nature of this identifiability, 
we consider four settings that vary based on whether the group measure $\Lambda_*$ is known or unknown, whether it equals the Haar measure $\Lambda_0$, and whether $\theta_*$ is the sole parameter of interest. In all cases, $\theta_*$ is assumed to be unknown. 

For clarity, we first study the simplified setting where $P = \Id$. In this case, the statement that $\theta_*$ is identifiable up to the distribution of $g \theta_*$ (with $g \sim \Lambda_*$) can be further simplified in each setting:
\begin{itemize}
    \item[(1)] \emph{Estimating $\theta_*$ with known $\Lambda_*=\Lambda_0$.}
    The equality in distribution $g \theta \overset{d}{=} g \theta_*$ for $g \sim \Lambda_0$ holds if and only if $\orbit_{\theta} = \orbit_{\theta_*}$. In other words, $\theta_*$ is identifiable exactly up to its group orbit $\orbit_{\theta_*}$ under $\G$.

    \item[(2)] \emph{Estimating $\theta_*$ with  known  $\Lambda_*\neq \Lambda_0$. }
    The equality in distribution $g \theta \overset{d}{=} g \theta_*$ for $g \sim \Lambda_*$ holds if and only if $\theta=\theta_*$ for a generic choice of the pair $(\theta_*,\Lambda_*)$. In other words, under a generic measure $\Lambda_*$, a generic signal $\theta_*$ is point identified. 

    \item[(3)] \emph{Estimating the pair $(\theta_*, \Lambda_*)$ with $\Lambda_*$  unknown. }
    The equality in distribution $g' \theta \overset{d}{=} g \theta_*$ for $g \sim \Lambda_*,g'\sim \Lambda$ holds if and only if $\orbit_{(\theta,\Lambda)}=\orbit_{(\theta_*,\Lambda_*)}$ for a generic choice of the pair $(\theta_*,\Lambda_*)$. Here $\orbit_{(\theta,\Lambda)}:=\{(g\theta, \Lambda g^{-1}):g\in\G\}$   denotes the joint group orbit of the pair $(\theta,\Lambda)$ under the group action, where $\Lambda g^{-1}$ is the pushforward of the measure $\Lambda$ under right multiplication by $g^{-1}$, that is, for any measurable set $A \subseteq \G$, $(\Lambda g^{-1})(A) := \Lambda(A g)$. In other words, the pair $(\theta_*,\Lambda_*)$ is generically identifiable up to its joint group orbit under $\G$.

    \item[(4)] \emph{Estimating only $\theta_*$ with $\Lambda_*$  unknown.}
    For a generic pair $(\theta_*, \Lambda_*)$, since the group measure $\Lambda_*$ is not of direct inferential interest, the signal $\theta_*$ is  identifiable up to its group orbit $\orbit_{\theta_*}$ under $\G$.

\end{itemize}

More generally, when the projection $P \ne \Id$, the identifiability problem can become more intricate. The nature of identifiability—whether up to orbits or points—depends on the setting considered above, such as whether the group measure $\Lambda_*$ is known or needs to be estimated. However, once a nontrivial projection is introduced, additional non-identifiability may arise beyond this baseline. That is, in the orbit-identifiable regime, multiple distinct orbits may become indistinguishable after projection, and in the point-identifiable regime, multiple distinct signals may project to the same distribution. The extent of such ambiguity is determined by the interplay among the group $\G$, the projection operator $P$, and the group measure $\Lambda_*$ (see also~\cite{fan2024maximum}). For example, in the case of tomographic projection in cryo-EM (see Section~\ref{sec:model_cryoEM}) with the Haar measure, it is known that $Pg\theta = Pg\theta'$ even when $\theta' \neq \theta$, where $\theta'$ is the reflection of $\theta$ through a 2-D plane passing through the origin. As a result, $\theta_*$ is identifiable only up to chirality~\cite{bendory2020single}.

The identifiability claims made above are understood to hold for \textit{generic} instances of the signal $\theta_*$ and, when applicable, the group measure $\Lambda_*$. We briefly explain what this means. For the signal $\theta_* \in \mathbb{R}^d$, genericity is understood in the sense of real algebraic geometry:

\begin{definition}
A subset $H \subseteq \mathbb{R}^d$ is called \emph{generic} if its complement lies in the zero set of some non-zero real analytic function $f: \mathbb{R}^d \to \mathbb{R}^k$ for some $k \geq 1$. A statement is said to hold for \emph{generic} $\theta \in \mathbb{R}^d$ if it holds for all $\theta$ in some generic subset $H \subseteq \mathbb{R}^d$.
\end{definition}

This notion ensures that the exceptional set has Lebesgue measure zero~\cite[Prop.~0]{mityagin2015zero}. For the distribution $\Lambda_*$ over the group $\G$, genericity is typically imposed on its expansion coefficients with respect to a suitable basis adapted to the group structure. While we do not aim to establish the weakest possible generic conditions in full generality—especially since the precise form depends on the application at hand—concrete results are known in specific settings.
For example, in the discrete dihedral MRA model, $\Lambda_*$ admits a finite-dimensional representation in $\mathbb{R}^{2d}$, and genericity can be formulated in terms of this vector representation~\cite{bendory2022dihedral}. In the cryo-EM setting, $\Lambda_*$ can be expanded in terms of the complex Wigner D-matrix basis, and genericity is imposed on the corresponding expansion coefficients~\cite{sharon2020method}.

Since our focus lies in understanding the optimization landscape rather than establishing a complete identifiability theory, we refrain from formalizing genericity assumptions for the most general model. We refer interested readers to the aforementioned works for more detailed treatments in specific settings. The development of more broadly applicable conclusions is deferred to future work.

The four settings discussed above are not isolated, but rather closely related. Among them, Setting (1), which assumes the group measure $\Lambda_*=\Lambda_0$, is the most thoroughly understood from a theoretical standpoint (see e.g. \cite{fan2023likelihood, fan2024maximum, bandeira2023estimation}). However, in practice, the Haar assumption is often not satisfied, motivating interest in Settings (2) to (4), where the measure is non-Haar. Both Setting (1) and (2) are directly related to the correctly specified MLE analysis in Section \ref{sec:correct_MLE}. Setting (2) assumes $\Lambda_*$ is known; the more challenging and realistic scenarios are Settings (3) and (4), in which $\Lambda_*$ is unknown and possibly infinite-dimensional.

Between these two, Setting (4)—where the goal is to estimate only the signal $\theta_*$ without recovering $\Lambda_*$—is often of primary practical interest. This is particularly true in fields such as structural biology, where the distribution $\Lambda_*$ typically lacks direct scientific interpretation, and the main objective is to recover $\theta_*$, which represents the underlying biomolecular structure. This setting serves as the main motivation for our analysis, and Sections \ref{sec:landscape_without_projection} and \ref{sec:landscape_with_projection} focus on it from an optimization perspective.

Nonetheless, there is a direct connection between Settings (3) and (4): if one obtains a good estimate of the joint orbit of $(\theta_*, \Lambda_*)$ as in Setting (3), then projecting this estimate onto the signal domain naturally yields an estimator for the orbit of $\theta_*$ in the spirit of Setting (4). This connection underpins our study of the joint MLE, which we analyze in Section \ref{sec:joint_MLE}.

\subsection{Global Minimizers under Correct Specification}\label{sec:correct_MLE}

We present the following theorem, which characterizes the global minimizers of the correctly specified negative population log-likelihood $\ell(\theta;\Lambda_,\sigma^2)$. These minimizers comprise the set of parameters $\theta$ for which the distribution of $Pg\theta$ matches that of $Pg\theta_*$ under $g \sim \Lambda_*$, as naturally follows from the identification analysis.

\begin{theorem}\label{lem:GlobalMin_CorrectLikelihood}
The global minimizers of the negative population log-likelihood for the correctly specified model are given by $\{\theta\in\R^d: Pg\theta  \overset{d}{=} Pg\theta_*~\text{ for }~g\sim \Lambda_*\}$.
\end{theorem}

As an immediate consequence, when $\Lambda_* = \Lambda_0$, the set of global minimizers corresponds to the set of signals $\{\theta\in\R^d: Pg\theta  \overset{d}{=} Pg\theta_*~\text{ for }~g\sim \Lambda_0\}$. This set can be interpreted as an equivalence class of orbits $\orbit_\theta$ whose projections match that of the true signal $P(\orbit_{\theta})\equiv P(\orbit_{\theta_*})$, meaning that these sets are equal both as subsets of $\R^{\td}$ and in distribution: $Pg\theta \stackrel{d}{=} Pg\theta_*$ under $g \sim \Lambda_0$. In generic scenario, there exists only finitely many orbits $\orbit_\theta$ satisfying $P(\orbit_{\theta})\equiv P(\orbit_{\theta_*})$, yielding a form of generic \textit{list recovery} as characterized in \cite{bandeira2023estimation}. For more general, generic measures $\Lambda_*$, it typically identifies a set of equivalent points containing $\theta_*$, as discussed in the Section \ref{sec:identification}.

For the misspecified model, recall that the negative population log-likelihood for the misspecified model is defined in \eqref{eq:population_likelihood} as
\begin{align*}
\ell(\theta;\Lambda_0,\sigma^2) = -\E_{\theta_*, \Lambda_*} \log p_\theta(y; \Lambda_0,\sigma^2).
\end{align*}
Note that this objective is invariant under the group action: for any $g \in \G$, we have $p_{g\theta}(y; \Lambda_0,\sigma^2) = p_\theta(y; \Lambda_0,\sigma^2)$. As a result, the global minimizers of $\ell(\theta; \Lambda_0, \sigma^2)$ must form a union of orbits under the action of $\G$. However, there is generally no explicit or universal characterization of these global minimizers. Therefore, in the following sections, we analyze and discuss their structure under different scenarios, including cases with and without projection, as well as under high-noise and low-noise regimes.

\section{Theory}\label{sec:theory}
In this section, we present the main results of the paper. Our analysis is organized around two key settings. In Section~\ref{sec:landscape_without_projection}, we consider the \emph{unprojected case} where the projection operator is the identity ($P = \mathrm{Id}$). In this setting, we establish a strong robustness property of the MLE under group distribution misspecification. In Sections~\ref{sec:landscape_with_projection} and~\ref{sec:joint_MLE}, we turn to the more general and practically relevant \emph{projected case} ($P \neq \mathrm{Id}$), where misspecification induces bias. We characterize this bias and introduce a joint estimation framework to mitigate it when the true group distribution is unknown.

\subsection{Optimization Landscape Without Projection}\label{sec:landscape_without_projection}
In this section, we are going to show when there is no projection, the MLE still works when the group measure is misspecified as the Haar measure. 
Recall that the negative population log-likelihood for the misspecified model is defined in \eqref{eq:population_likelihood} as
\begin{align*}
\ell(\theta;\Lambda_0,\sigma^2) = -\E_{\theta_*, \Lambda_*} \log p_\theta(y; \Lambda_0,\sigma^2).
\end{align*}
To explicitly capture the dependence on $\Lambda_*$, which is embedded in the expectation as mentioned in Section \ref{sec:setup}, we introduce the more general notation: 
\begin{align}\label{eq:population_likelihood_general}
\ell(\theta;\Lambda_1,\Lambda_2,\sigma^2) := -\E_{\theta_*, \Lambda_2} \log p_\theta(y; \Lambda_1,\sigma^2),
\end{align}
where the expectation is taken under the model \eqref{eq:projectedmodel} with $\Lambda_*=\Lambda_2$. This formulation represents the negative log-likelihood when the data is generated according to the group measure $\Lambda_2$, while the model assumes the group measure $\Lambda_1$.

Under this notation, the negative log-likelihood for the correctly specified model is $\ell(\theta;\Lambda_*,\Lambda_*,\sigma^2)$, while that for the misspecified model is $\ell(\theta;\Lambda_0,\Lambda_*,\sigma^2)$. For simplicity, when the data is generated according to the true underlying group measure $\Lambda_*$, we always omit the $\Lambda_*$ in $\ell(\theta;\Lambda,\Lambda_*,\sigma^2)$ and use the shorthand notation $\ell(\theta;\Lambda,\sigma^2)$, which aligns with the definition \eqref{eq:population_likelihood}. According to Theorem \ref{lem:GlobalMin_CorrectLikelihood}, the global minimizers of the correctly specified log-likelihood $\ell(\theta;\Lambda_0,\Lambda_0,\sigma^2)$ are exactly those points whose projected orbits agree with the true projected orbit, i.e., $\{\theta\in\R^d: P(\orbit_\theta)\equiv P(\orbit_{\theta_*})\}$. We now focus on characterizing the optimization landscape and the global minimizers of the misspecified log-likelihood
$\ell(\theta;\Lambda_0,\Lambda,\sigma^2)$ under a general group measure $\Lambda$.

\subsubsection{Motivating Example: Symmetric Two-Component Gaussian Mixture} \label{sec:two_component}
For the purpose of illustration, we first consider a motivating example, where the data is generated from the symmetric two-component Gaussian mixture model on $\mathbb{R}$ as described in Section \ref{sec:model_2GM}. Recall that $\pi\in[0,1]$, $\theta_*\in\R$ are the true weight and center. The misspecified negative log-likelihood can be written as $\ell(\theta;1/2,\pi)$, where $\ell(\theta;\pi_1,\pi_2)$ is defined analogously to (\ref{eq:population_likelihood_general}):
\begin{align*}
\ell(\theta;\pi_1,\pi_2)= -\E_{\theta_*,\pi_2}\log\left[ \pi_1\frac{1}{\sqrt{2\pi\sigma^2}}
 \exp\left(-\frac{(y-\theta)^2}{2\sigma^2}\right)+(1-\pi_1)\frac{1}{\sqrt{2\pi\sigma^2}}
 \exp\left(-\frac{(y+\theta)^2}{2\sigma^2}\right)\right].
\end{align*}

Since $\theta$ is one-dimensional, we can visualize the landscape of  $\ell(\theta;1/2,\pi)$ with respect to $\pi$ and $\theta$ through a 3D surface plot (see the left panel of Figure \ref{fig:contour}), demonstrating that irrespective of the value of \(\pi\), the shape of the negative log-likelihood function with respect to \(\theta\) remains invariant. In addition, the minimum is consistently achieved at the ground truth $\theta_*$ and $-\theta_*$,  showing that the misspecified MLE works.

\begin{figure}
\centering
\begin{minipage}{.5\textwidth}
  \centering
  \includegraphics[width=\textwidth]{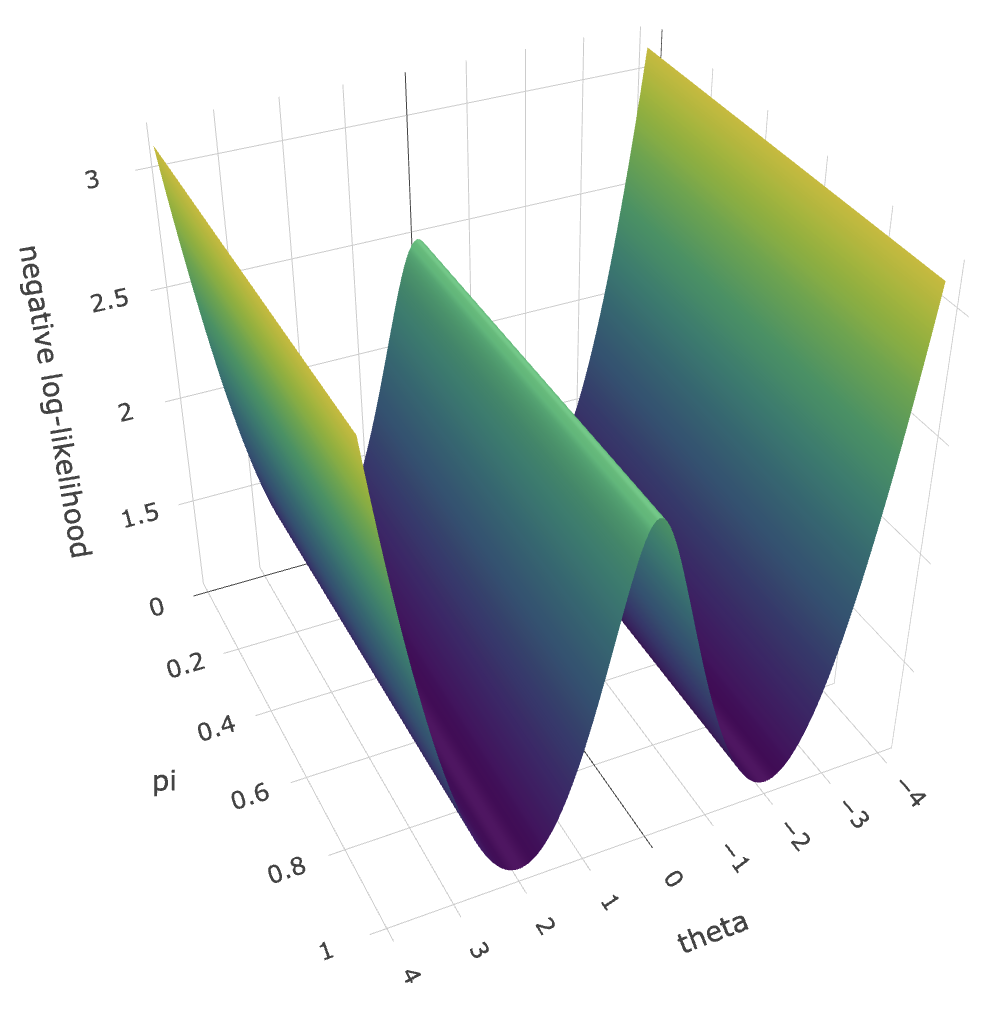}
\end{minipage}%
\begin{minipage}{.5\textwidth}
  \centering
  \includegraphics[width=\textwidth]{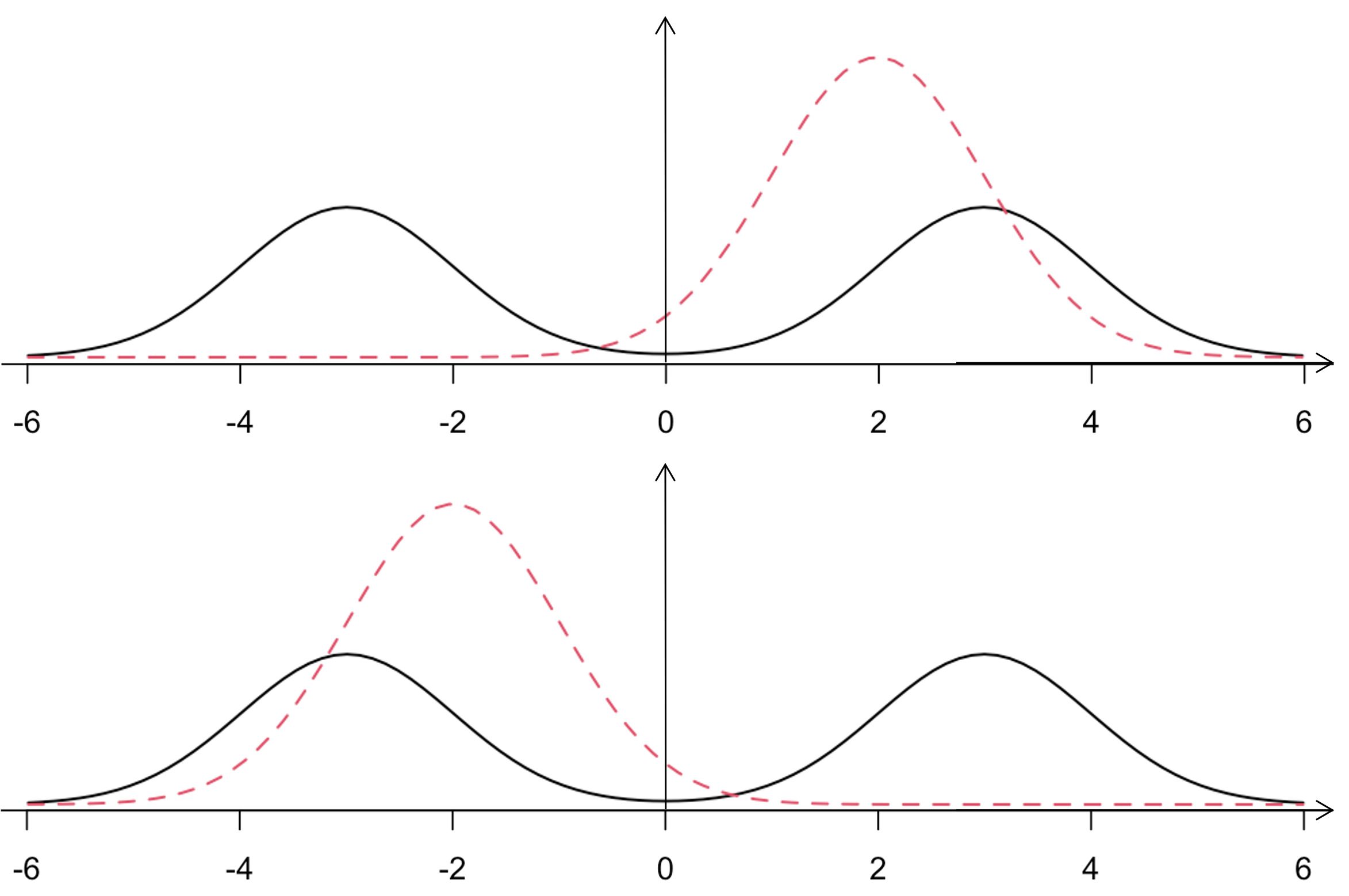}
\end{minipage}
\caption{Left: 3D surface plot of the negative log-likelihood function \(\ell(\theta;1/2,\pi)\) for the two-component Gaussian mixture with \(\theta_*=2\) with \(\theta\) and \(\pi\) varying. Right: Illustration of normal densities.}\label{fig:contour}
\end{figure}

The intuition behind the left panel of Figure \ref{fig:contour} is as follows. Note that the density function under the uniform prior, $  \frac{1}{2}\frac{1}{\sqrt{2\pi\sigma^2}}
\exp\left(-\frac{(y-\theta)^2}{2\sigma^2}\right)+\frac{1}{2}\frac{1}{\sqrt{2\pi\sigma^2}}
\exp\left(-\frac{(y+\theta)^2}{2\sigma^2}\right)$, is symmetric about 0 (illustrated as the black curves in the right panel of Figure \ref{fig:contour}). Together with the symmetry between $\mathrm{N}(\theta_*,\sigma^2)$ and $\mathrm{N}(-\theta_*,\sigma^2)$ (illustrated as the red dashed curves in the same panel), we have $\ell(\theta;1/2,0) =\ell(\theta;1/2,1)$. As a result, $\ell(\theta;1/2,\pi)$, which is equal to $\pi\ell(\theta;1/2,1)+(1-\pi)\ell(\theta;1/2,0)$,  is invariant of $\pi$.

\subsubsection{General Case}
Despite the specific instance shown in Figure \ref{fig:contour}, which is limited to one-dimensional, symmetric two-component Gaussian mixture models, the conclusion extends to more general settings. In fact, the following result shows that, in the absence of projection, the misspecified population log-likelihood is invariant to the choice of the underlying group distribution.

\begin{theorem}\label{thm:withoutProj}
Suppose there is no projection, i.e., $P=\Id$. Then for any group measure $\Lambda$ on $\G$ and noise level $\sigma^2$, we have 
\begin{align*}
\ell(\theta;\Lambda_0,\Lambda,\sigma^2)=\ell(\theta;\Lambda_0,\Lambda_0,\sigma^2).
\end{align*}
In particular, this holds for the true underlying group measure: 
\begin{align*}
\ell(\theta;\Lambda_0,\Lambda_*,\sigma^2)=\ell(\theta;\Lambda_0,\Lambda_0,\sigma^2)
\end{align*}
\end{theorem}

\begin{proof}
Recall that by Fubini's theorem, we can change the order of the integrals, yielding that
\begin{align}\label{eq:withoutProj}
\ell(\theta;\Lambda_0,\Lambda,\sigma^2)=&-\E_{\theta_*,\Lambda} \log p_\theta(y;\Lambda_0,\sigma^2)\notag\\
=&-\int_{\G}\Big[\int_{\R^d} \log p_\theta(y;\Lambda_0,\sigma^2) \cdot  \frac{1}{(2\pi\sigma^2)^{d/2}}
\exp\Big(-\frac{\|y- \tg\theta_*\|^2}{2\sigma^2}\Big)\der y\Big]\der \Lambda(\tg).
\end{align}
We now turn to show the quantity inside the square brackets does not depend on $\tg\in\G$. Note that for any measurable function $h:\R^d\to\R$ and $\tg\in\G$, $y\in\R^d$,
\begin{align*}
\int_{\G} h\Big(g^\T \tg y\Big)\der \Lambda_0(g)=\int_{\G} h\Big((g\tg^{-1})^\T y\Big)\der \Lambda_0(g) = \int_{\G} h\Big(g^\T y\Big)\der \Lambda_0(g),
\end{align*}
where the last equation follows from the fact that the Haar measure is invariant under the group action, i.e., if $g\sim \Lambda_0$ then $g\tg^{-1}\sim \Lambda_0$. Additionally, since $\tg\in\G$ is an orthogonal matrix, we further have $\int_{\R^d}h(\tg x)\der x = \int_{\R^d} h(x)\der x$ and $\|\tg x\|=\|x\|$ for any $x\in\R^d$. Consequently, for any $\tg\in\G$ and $\tilde\theta\in\R^d$,
\begin{align*}
& \int_{\R^d} \log p_\theta(y;\Lambda_0,\sigma^2)  \cdot 
\exp\Big(-\frac{\|y- \tg \tilde\theta\|^2}{2\sigma^2}\Big)\der y\\
=&\int_{\R^d} \Big[\log\int_{\G} \frac{1}{(2\pi\sigma^2)^{d/2}}
\exp\Big(-\frac{\|y- g\theta\|^2}{2\sigma^2}\Big)\der
\Lambda_0(g)\Big]\cdot 
\exp\Big(-\frac{\|y- \tg \tilde\theta\|^2}{2\sigma^2}\Big)\der y\\
=&\int_{\R^d} \Big[\log\int_{\G} \frac{1}{(2\pi\sigma^2)^{d/2}}
\exp\Big(-\frac{\|\tg y- g\theta\|^2}{2\sigma^2}\Big)\der
\Lambda_0(g)\Big]\cdot 
\exp\Big(-\frac{\|\tg y- \tg \tilde\theta\|^2}{2\sigma^2}\Big)\der y\\
=&\int_{\R^d} \Big[\log\int_{\G} \frac{1}{(2\pi\sigma^2)^{d/2}}
\exp\Big(-\frac{\|g^\T\tg y- \theta\|^2}{2\sigma^2}\Big)\der
\Lambda_0(g)\Big]\cdot 
\exp\Big(-\frac{\|y- \tilde\theta\|^2}{2\sigma^2}\Big)\der y\\
=&\int_{\R^d} \Big[\log\int_{\G} \frac{1}{(2\pi\sigma^2)^{d/2}}
\exp\Big(-\frac{\|y- g\theta\|^2}{2\sigma^2}\Big)\der
\Lambda_0(g)\Big]\cdot 
\exp\Big(-\frac{\|y- \tilde\theta\|^2}{2\sigma^2}\Big)\der y,
\end{align*}
which does not depend on $\tg$ as desired. Hence, we can replace $\Lambda$ by $\Lambda_0$ in \eqref{eq:withoutProj} and deduce that 
\begin{align*}
\ell(\theta;\Lambda_0,\Lambda,\sigma^2)=&-\int_{\G}\Big[\int_{\R^d} \log p_\theta(y;\Lambda_0,\sigma^2) \cdot  \frac{1}{(2\pi\sigma^2)^{d/2}}
\exp\Big(-\frac{\|y- \tg\theta_*\|^2}{2\sigma^2}\Big)\der y\Big]\der \Lambda(\tg)\\
=&-\int_{\G}\Big[\int_{\R^d} \log p_\theta(y;\Lambda_0,\sigma^2) \cdot  \frac{1}{(2\pi\sigma^2)^{d/2}}
\exp\Big(-\frac{\|y- \tg\theta_*\|^2}{2\sigma^2}\Big)\der y\Big]\der \Lambda_0(\tg)\\
=&\ell(\theta;\Lambda_0,\Lambda_0,\sigma^2),
\end{align*}
which completes the proof.
\end{proof}

\begin{remark}
The proof of Theorem~\ref{thm:withoutProj} actually extends beyond the Gaussian noise setting. The argument applies more generally to any isotropic noise distribution for which the conditional density of the observation $y$ given $g \in \G$ and parameter $\theta$ depends only on the Euclidean distance $\|y - g\theta\|$. In such cases, the invariance and change-of-variable arguments remain valid, and the conclusion of the theorem continues to hold.
\end{remark}

We immediately have the following corollary, characterizing the global minimizers of the misspecified log-likelihood.
\begin{corollary}\label{cor:withoutProj}
Suppose there is no projection, i.e., $P=\Id$. The negative log-likelihoods $\ell(\theta;\Lambda_0,\Lambda_*,\sigma^2)$ and $\ell(\theta;\Lambda_0,\Lambda_0,\sigma^2)$ have exactly the same optimization landscape. In particular, the orbit of $\theta_*$ contains all the global minimizers of misspecified log-likelihood $\ell(\theta;\Lambda_0,\Lambda_*,\sigma^2)$.
\end{corollary}

\begin{remark}
This corollary leads to a surprising yet intuitive conclusion: without projection, we do not even need to know the true data-generating group measure $\Lambda_*$. By optimizing the misspecified MLE objective, one can still recover the true orbit. Remarkably, this result holds for any group measure, including those with only local support.
\end{remark}

\begin{remark}
Corollary~\ref{cor:withoutProj} implies that the sample misspecified MLE, the global minimizer of $\ell_n(\theta;\Lambda_0,\sigma^2)$, is consistent. This follows from standard M-estimation theory  \citep{van2000asymptotic}: since $\ell_n(\theta;\Lambda_0,\sigma^2)$ is the finite-sample version of of the population objective $\ell(\theta;\Lambda_0,\Lambda_*,\sigma^2)$, under mild regularity conditions, its global minimizer converges to that of the population objective as the sample size $n \to \infty$.

Nevertheless, misspecification does come with certain costs. One is identifiability: $\theta_*$ is point identified in correctly specified MLE but is only identifiable up to orbit in misspecified MLE (see Section \ref{sec:identification}). Another important cost concerns efficiency. Although the global minimizers of the misspecified likelihood coincide with the true orbit corresponding to the global minimizer of the correctly specified likelihood, the two likelihood functions are generally different—that is, $\ell(\theta; \Lambda_0, \Lambda_*, \sigma^2)$ typically differs from $\ell(\theta; \Lambda_*, \Lambda_*, \sigma^2)$, with the former often exhibiting flatter curvature at the global minimizers. This implies that the misspecified MLE is asymptotically less efficient. Figure~\ref{fig:fisher} illustrates this phenomenon in the symmetric two-component Gaussian mixture example discussed in Section~\ref{sec:two_component}.
\end{remark}

\begin{figure}
\centering
\begin{minipage}{.5\textwidth}
  \centering
  \includegraphics[width=\textwidth]{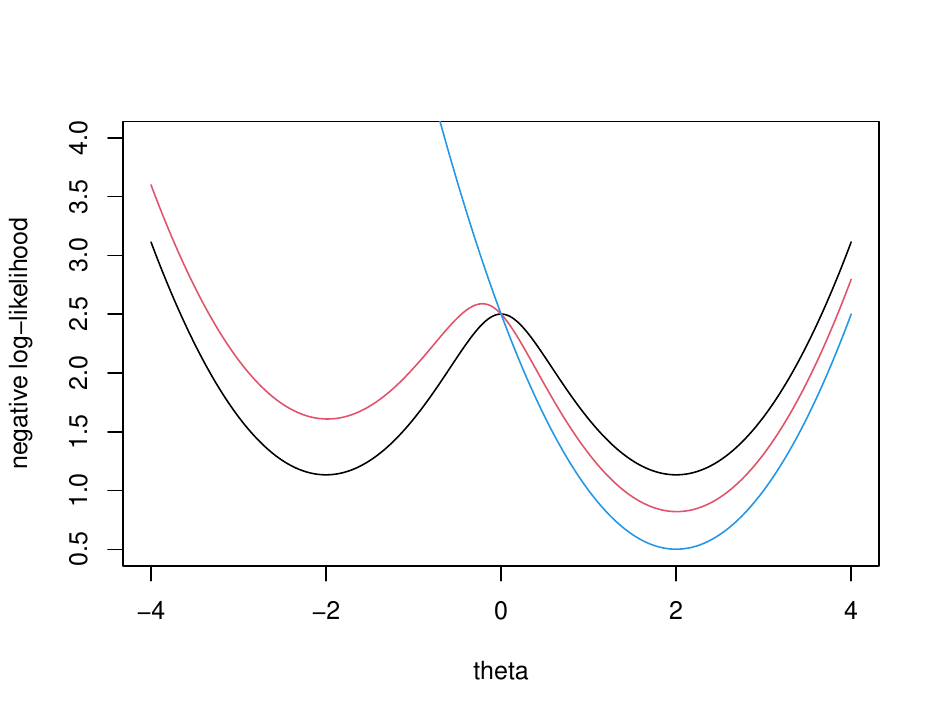}
\end{minipage}%
\begin{minipage}{.5\textwidth}
  \centering
  \includegraphics[width=\textwidth]{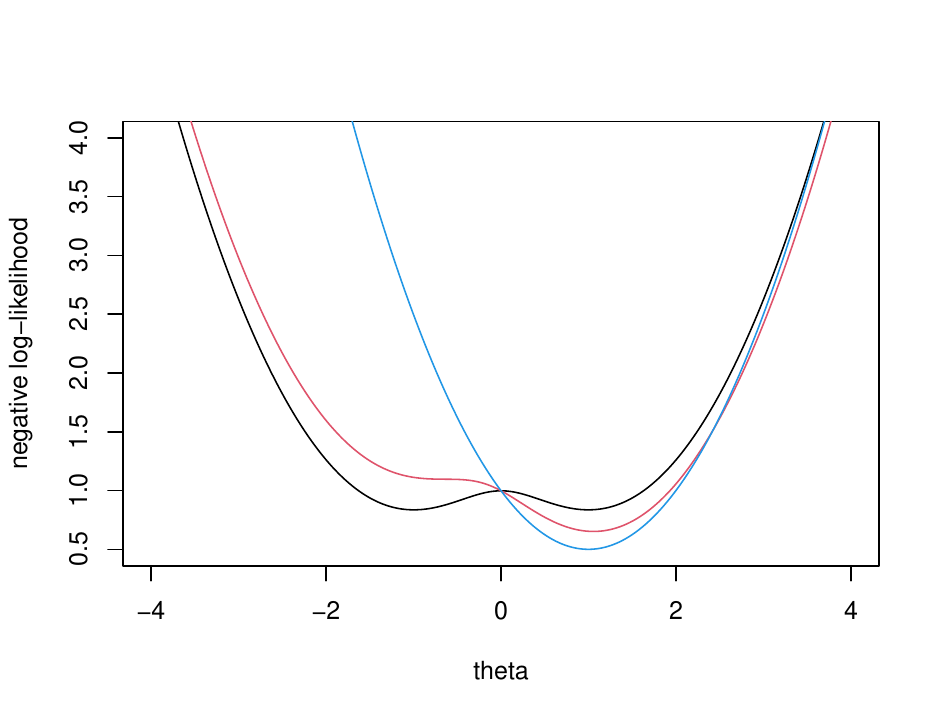}
\end{minipage}
\caption{Left: Plots of the negative log-likelihood function $\ell(\theta; \pi, \pi)$ for the symmetric two-component Gaussian mixture with $\theta_* = 2$ and varying $\pi$. Black: $\pi = 0.5$; Red: $\pi = 0.7$; Blue: $\pi = 1$. The global minimum occurs only at $\theta_*$ for the red and blue curves and at both $\pm\theta_*$ for the black curve. Since $\ell(\theta; \pi, \pi)$ corresponds to the correctly specified MLE, its curvature at $\theta_*$ reflects asymptotic statistical efficiency: the flatter the curve at $\theta_*$, the lower the efficiency. Among the three, the black curve is the flattest. The misspecified MLE $\ell(\theta; 0.5, \pi)$ shares the same landscape as $\ell(\theta; 0.5, 0.5)$, and thus has the same lower efficiency, making it less efficient than the correctly specified MLE $\ell(\theta; \pi, \pi)$.
Right: The difference in curvature at $\theta_*$ becomes more significant when  $\theta_*=1$.}\label{fig:fisher}
\end{figure}

\subsection{Global Minimizers With Projection}\label{sec:landscape_with_projection}
In this section, we analyze the properties of the misspecified population negative log-likelihood in the presence of projection, with a focus on the case where the group $\G$ is a continuous group. While some of our results also hold or have analogues in the discrete group setting, we defer discussion of those differences to the remarks following the main theorems. Recall that the negative population log-likelihood for the misspecified model with projection has the form 
\begin{align}\label{eq:loss_with_projection}
\ell(\theta;\Lambda_0,\sigma^2)=-\E_{\theta_*,\Lambda_*}\log\int_{\G} \frac{1}{(2\pi\sigma^2)^{\td/2}}
\exp\left(-\frac{\|y- P g\theta\|^2}{2\sigma^2}\right)\der
\Lambda_0(g).
\end{align}
Define the misspecified MLE as
\begin{align*}
\hat\theta(\Lambda_0,\sigma^2)=&\argmin_{\theta} \ell(\theta;\Lambda_0,\sigma^2)\\
=&\argmax_{\theta} \E_{\theta_*, \Lambda_*} \log p_\theta(y; \Lambda_0,\sigma^2).
\end{align*}
As discussed in Section~\ref{sec:correct_MLE}, the global minimizers of $\ell(\theta; \Lambda_0, \sigma^2)$ are not unique in general, but rather form a union of orbits under the group action of $\G$; that is, $\ell(\theta; \Lambda_0, \sigma^2)$ attains the same value for all points within the same orbit. Accordingly, $\hat\theta(\Lambda_0,\sigma^2)$ should be understood as an element in this orbit set of global minimizers.

Specifically, in Section~\ref{sec:existence_bias}, we show that for a generic noise level, there always exists a data-generating group measure under which the true signal orbit $\orbit_{\theta_*}$ fails to minimize $\ell(\theta;\Lambda_0,\sigma^2)$, thereby revealing the intrinsic bias of the misspecified MLE. This bias is further characterized in Section~\ref{sec:quantification_bias} via a one-sided Hausdorff distance between the true \textit{projected orbit} and the \textit{projected orbit} associated with the misspecified MLE, which quantifies how far the misspecified MLE deviates from the ground truth in the high signal-to-noise regime (see \eqref{eq:proj_orbit} for the definition of projected orbit). Finally, in Section~\ref{sec:bias_sigma_0}, we strengthen this result by showing that the full Hausdorff distance between the two projected orbits vanishes as $\sigma \to 0$, providing theoretical justification for the use of the misspecified MLE in low-noise settings.

\subsubsection{Existence of Bias}\label{sec:existence_bias}

We aim to establish the existence of estimation bias under a generic noise level in the projected model. Specifically, we show that the true orbit $\orbit_{\theta_*}$ does not lie within the set of global minimizers of the misspecified population log-likelihood. Following the discussion in Section \ref{sec:correct_MLE}, this is equivalent to showing that the true signal $\theta_*$ itself is not a global minimizer of $\ell(\theta; \Lambda_0, \sigma^2)$. To this end, it suffices to show that the gradient of the objective is non-zero at $\theta_*$.

Let $\G$ be a continuous and connected group. Our main technical tool is the identity theorem for real analytic functions. To apply this theorem, we need to introduce analytic structure into the likelihood landscape. This motivates a reparameterization of the group $\G$ via a real-analytic surjection from a subset of Euclidean space. Specifically, assume that there exist a surjective function $\eta:U\mapsto \G$ and a measure $\tilde\lambda_0$ on $U$ such that 
\begin{itemize}
    \item[(1)] $U$ is a connected bounded subset of $\R^m$ for some positive $m$;
    \item[(2)] $U$ has non-empty interior with respect to $\R^m$;
    \item[(3)] $\eta$ is measurable with respect to $\tilde\lambda_0$ and the corresponding pushforward measure is the Haar measure $\Lambda_0$ over $\G$.
\end{itemize}
The assumption of such a surjective parametrization is natural and is satisfied in practical settings, as will be evident in the examples of projected MRA and cryo-EM discussed later. With this setup, for $v\in U$, $\theta\in\R^d$, and $\tau^2$, we define the function
\begin{align*}
H(v;\theta,\tau^2)=\E_{y\sim \mathrm{N}(P\eta(v)\theta,\tau^2\Id_{\td})}  \frac{\int_{U}
\exp\Big(-\frac{\|y- P \eta(u)\theta\|^2}{2\tau^2}\Big) \cdot \eta(u)^\T P^\T \Big(P \eta(u)\theta-y\Big)\der
\tilde\lambda_0(u)}{\int_{U}
\exp\Big(-\frac{\|y- P \eta(u)\theta\|^2}{2\tau^2}\Big)\der
\tilde\lambda_0(u)}.
\end{align*}
This expression captures the core integrand structure underlying the population gradient of the misspecified log-likelihood, expressed in local coordinates, and its analytic properties plays a central role in our argument. Notably, the integrand of the expectation in the display above is independent of $\eta(v)$ and is real analytic in $(\theta,\tau^2)$. Therefore, if $\eta(v)$ is real analytic in some domain $\tilde{U}\subseteq U$, it follows that $H(v;\theta,\tau^2)$ is also real analytic in $\tilde{U}\times \R^d\times \R_+$.

We now present the following main theorem in this section, which demonstrates that in the presence of projection, the misspecified negative log-likelihood generally induces a bias. In other words, the estimator obtained by maximizing the misspecified likelihood—using the Haar measure in place of the true group distribution—is not guaranteed to recover the true signal.

\begin{theorem}\label{thm:bias_existence}
Suppose there exists an open connected subset $\tilde{U}\subseteq U$ such that $\eta$ is entrywise real analytic in $\tilde{U}$. If there exist some $u_0\in\tilde{U}$, $\theta_0\in\R^d$, and $\tau_0^2$ such that $H(u_0;\theta_0,\tau_0^2)\not=0$, then for generic $\theta_*\in\R^d\setminus\{0\}$ and $\sigma^2\in\R_+$, there exists a group measure $\Lambda_*$ on $\G$ such that $\nabla_{\theta}\ell(\theta;\Lambda_0,\sigma^2)|_{\theta=\theta_*}\not=0$. In other words, $\theta_*$ is not a global minimizer of the misspecified negative population log-likelihood.
\end{theorem}

\paragraph{Projected Continuous MRA} Consider the example of Section \ref{sec:model_MRA} with the group element $g$ given by \eqref{eq:MRA_rotation}. In this case, we have the $U=[0,2\pi)$ and $\eta(u)=\mathcal{H}(u)$, the measure $\tilde\lambda_0$ corresponds to the uniform measure on $[0,2\pi)$. Moreover, $\eta$ is real analytic on $\tilde{U}=U^\circ=(0,2\pi)$. The existence of $u_0$, $\theta_0$, and $\tau_0^2$ satisfying the conditions of Theorem~\ref{thm:bias_existence} can be verified numerically. When $L = 1$, we can take $u_0 = 0$, $\theta_0 = (0,2, 1)^\T$, and $\tau_0^2 = 1$, and numerically compute $H(u_0;\theta_0,\tau_0^2)\approx (-0.21,-0.10)^\T$, which is non-zero. For the case $L > 1$, we can use the same values for $u_0$ and $\tau_0^2$, and set $\theta_0 \in \mathbb{R}^{2L+1}$ such that its first three coordinates are $0,2,1$, and the remaining entries are zero. With this construction, the value of $H(u_0; \theta_0, \tau_0^2)$ remains exactly the same as in the $L = 1$ case, and thus is also non-zero.
Therefore, the conclusion of Theorem~\ref{thm:bias_existence} holds: $\theta_*$ is not a global minimizer of the misspecified negative population log-likelihood in the projected continuous MRA model. To further illustrate this, the right panel of Figure~\ref{fig:so2} displays a contour plot of the misspecified negative population log-likelihood, which confirms that $\theta_*$ is not its minimizer.

\begin{figure}[ht]
\centering
\includegraphics[width=0.48\textwidth]{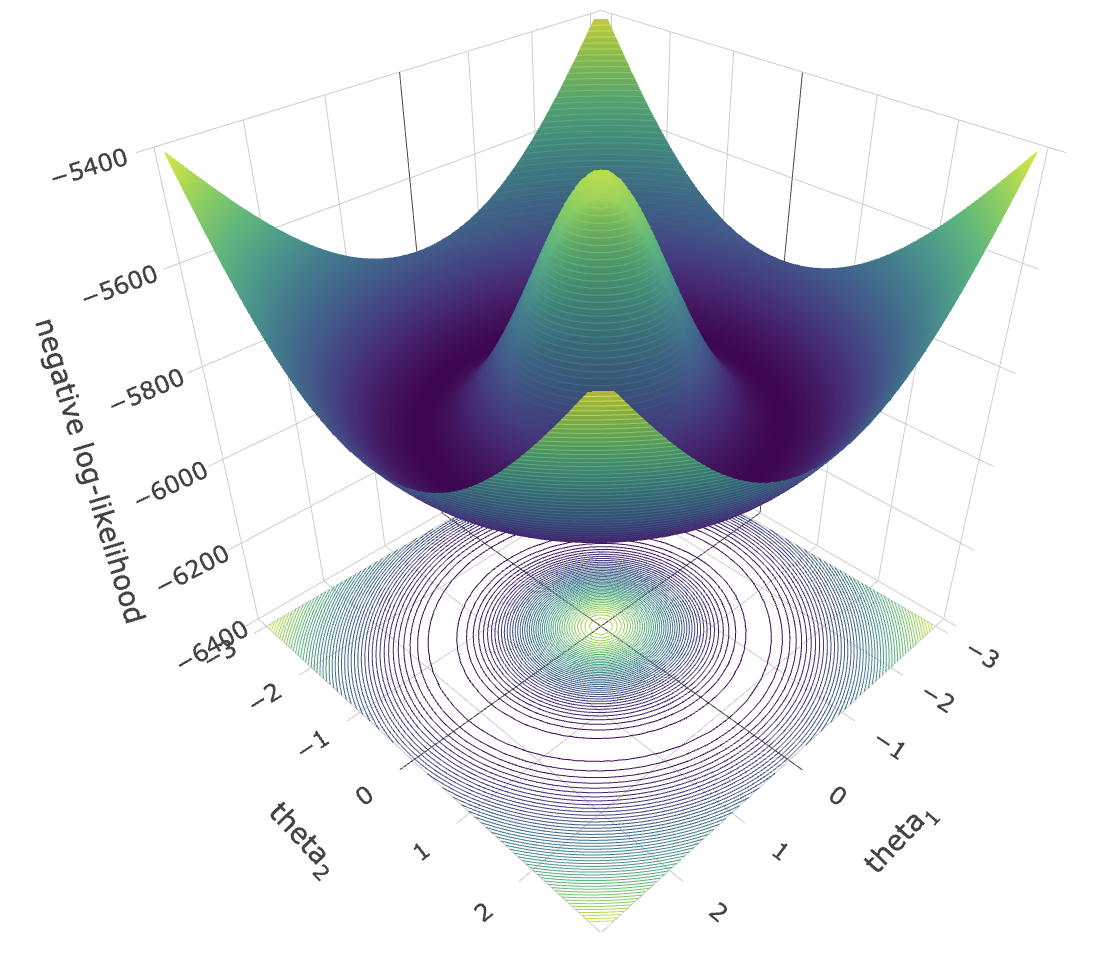}
\includegraphics[width=0.48\textwidth]{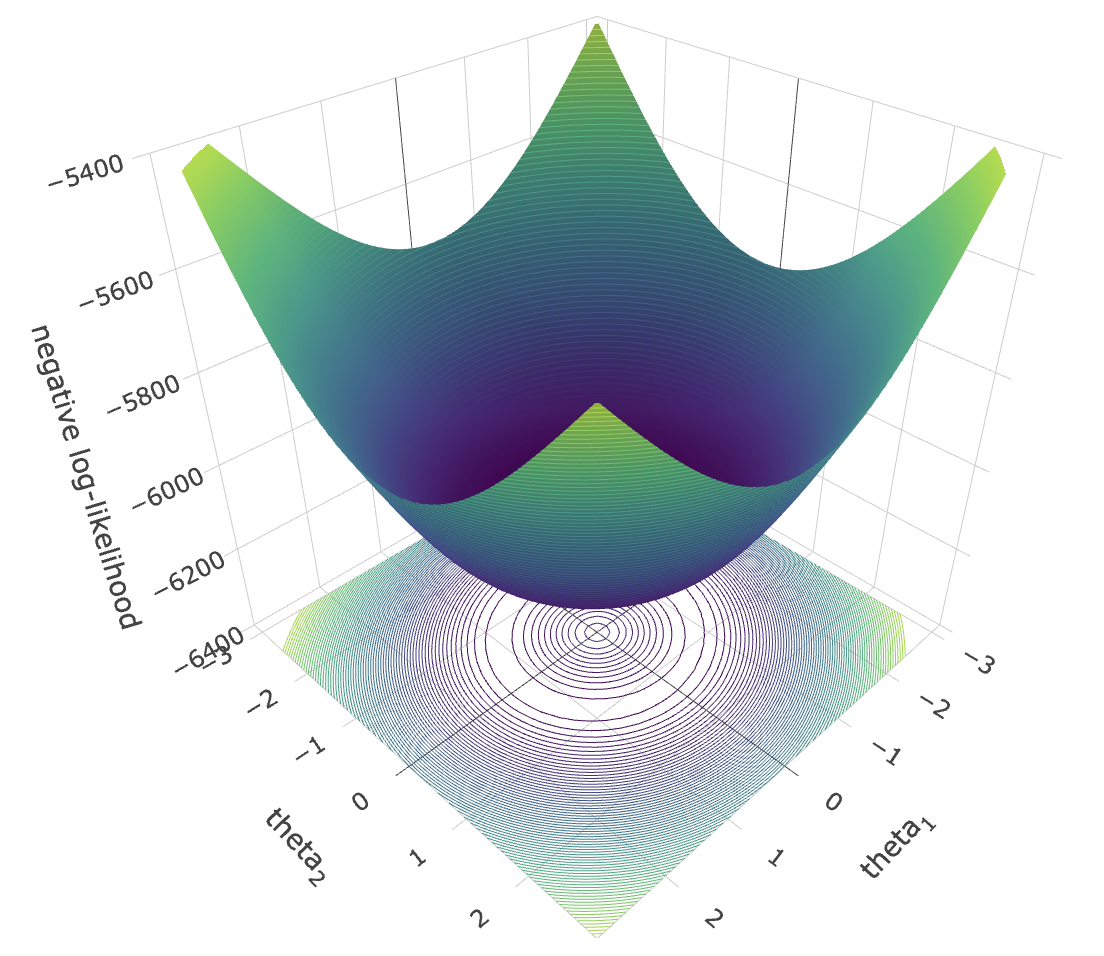}
\caption{Illustration of the optimization landscapes of the negative population log-likelihood for continuous MRA with $L=1$. The landscapes are functions of $(\theta_1^{(1)},\theta_2^{(1)})$, which appears as $(\theta_1,\theta_2)$ in the plots for convenience. The true parameter are $(2,0)$. Left: The correctly specified model where the minimizer is exactly $(2,0)$, recovering the true parameter. Right: A mispecified model where the minimizer is around $(1,0)$, indicating the true parameter is not its minimizer.}\label{fig:so2}
\end{figure}

\paragraph{(Projected) Cryo-EM} Consider the example of Section \ref{sec:model_cryoEM} with the group element $g$ given by \eqref{eq:cryoEM_rotation}. In this case, we can use the axis-angle representation for any $\frakg\in \SO(3)$. Specifically, let $\vartheta\in[0,2\pi)$ be the angle of rotation, and $(v_1,v_2,v_3)\in\mathcal{S}^2$ represent the unit vector along the axis of rotation. The vector $(v_1,v_2,v_3)$ can be parametrized by its latitude $\varphi_1\in[0,\pi]$ and longitude $\varphi_2\in[0,2\pi)$. Then the following $3\times 3$ matrix, which belongs to $\SO(3)$ and is denoted by $\frakg(\vartheta,\varphi_1,\varphi_2)$, corresponds to a rotation by angle $\vartheta$ around the axis $(v_1,v_2,v_3)$:
\begin{align*}
\begin{pmatrix}
v_1^2(1-\cos\vartheta)+\cos\vartheta \quad & \quad v_1 v_2(1-\cos\vartheta)-v_3\sin\vartheta \quad &  \quad v_1 v_3(1-\cos\vartheta)+v_2\sin\vartheta \\
v_1 v_2(1-\cos\vartheta)+v_3\sin\vartheta \quad & \quad v_2^2(1-\cos\vartheta)+\cos\vartheta \quad & \quad v_2 v_3(1-\cos\vartheta)-v_1\sin\vartheta \\
v_1v_3(1-\cos\vartheta)-v_2\sin\vartheta \quad & \quad v_2v_3(1-\cos\vartheta)+v_1\sin\vartheta \quad & \quad v_3^2(1-\cos\vartheta)+\cos\vartheta
\end{pmatrix}.
\end{align*}
Conversely, any element $\frakg\in\SO(3)$ can be represented in this form. Hence, we can set $\eta(u)=\check{\mathcal{D}}(\frakg(u))$ with $u=(\vartheta,\varphi_1,\varphi_2)$ and $U= [0,2\pi)\times [0,\pi]\times [0,2\pi)\subset \R^3$, such that $\eta:U\mapsto \G$ is a surjective function. The measure $\tilde\lambda_0$ corresponds to the product measure of the Haar measure on $\mathcal{S}^2$ and the uniform distribution on $[0,2\pi)$. Moreover, following \cite[Eqn. A16]{boyle2013angular} where $R_a=\cos(\vartheta/2)+\i \sin(\vartheta/2)v_3$ and $R_b=\sin(\vartheta/2)v_2+\i \sin(\vartheta/2)v_1$, we conclude that $\eta$ is real analytic in the domain $(0,2\pi)\times ((0,\pi/2)\cup (\pi/2,\pi))\times (0,2\pi)$. This ensures that $R_a\not=0$ and $R_b\not=0$ simultaneously. The existence of $u_0$, $\theta_0$, and $\tau_0^2$ satisfying $H(u_0; \theta_0, \tau_0^2)\neq 0$ in Theorem~\ref{thm:bias_existence} can be verified numerically, in a manner similar to the previous example. We omit the details here.

\subsubsection{Bias Quantification}\label{sec:quantification_bias}

Recall that the misspecified MLE is defined as
\begin{align*}
\hat\theta(\Lambda_0,\sigma^2)=\argmin_{\theta} \ell(\theta;\Lambda_0,\sigma^2).
\end{align*}
For any fixed $\theta\in\R^d$, define the projected orbit map 
\begin{align}\label{eq:Mtheta}
M_\theta: \G\to \R^{\td},~~~~~~ M_\theta (g) = P g \theta.
\end{align}
The \emph{projected orbit} of $\theta$ is then defined as the image of $M_\theta$, i.e.,
\begin{align}\label{eq:proj_orbit}
M_\theta(\G) := \{Pg\theta : g \in \G\} \subset \R^{\td}.
\end{align}
Based on the identifiability discussion in Section~\ref{sec:identification}, projected orbit recovery is often sufficient to ensure orbit list recovery. A full treatment of this implication, however, is beyond the scope of the present work.

In Theorems~\ref{thm:consistency_in_measure} and~\ref{thm:consistency_Hausdorff}, we establish upper bounds on the bias between the true signal $\theta_*$ and the misspecified estimator $\hat\theta(\Lambda_0, \sigma^2)$. More precisely, we quantify the deviation between their associated \textit{projected orbits}, $M_{\theta_*}(\G)$ and $M_{\hat\theta(\Lambda_0, \sigma^2)}(\G)$. Theorem~\ref{thm:consistency_in_measure} bounds the measure of points in $M_{\theta_*}(\G)$ that lie far from $M_{\hat\theta(\Lambda_0, \sigma^2)}(\G)$—a quantity that arises naturally in our analysis. This result is then translated into a more interpretable geometric distance: a one-sided Hausdorff distance between the two projected orbits, as formalized in Theorem~\ref{thm:consistency_Hausdorff}.

\begin{theorem}\label{thm:consistency_in_measure}
Assume that $\|P\|\leq C_P$ for some absolute positive constant $C_P$. Assume that the true group measure $\Lambda_*$ is absolutely continuous with respect to the Haar group measure $\Lambda_0$, with density $\der\Lambda_*/\der\Lambda_0 := Q$ satisfying $0<C_1\leq Q(g)\leq C_2$ for $\Lambda_0$-almost every $g\in\G$, where $C_1,C_2$ are absolute positive constants. For any $s>0$ and $\theta,\theta'\in\R^d$, define
\begin{align*}
m (s;\theta,\theta') = \Lambda_0 \big(\big\{g\in \G: d (Pg\theta, M_{\theta'}(\G))> s \big\} \big),
\end{align*}
where $d(x,A)$ denotes the Euclidean distance from a point $x$ to the set $A$ in $\R^{\td}$. Then for any fixed $s>0$ and $\alpha\in (0,2)$, there exists $\sigma_0>0$ depending only on $s,\td,\alpha$, such that for all $\sigma<\sigma_0$, the misspecified MLE satisfies 
\begin{align*}
m\big(s;\theta_*,\hat\theta(\Lambda_0,\sigma^2)\big)< \sigma^{2-\alpha}.
\end{align*}
\end{theorem}

\begin{remark}
Although Theorem~\ref{thm:consistency_in_measure} is stated for the misspecified likelihood obtained by replacing the true group measure $\Lambda_*$ with the Haar measure $\Lambda_0$, inspection of the proof reveals that the result in fact holds more generally. Specifically, the same conclusion remains valid if $\Lambda_0$ is replaced by any reference measure $\tilde\Lambda$ on $\G$ for which the true group measure $\Lambda_*$ is absolutely continuous with respect to $\tilde\Lambda$, and the Radon--Nikodym derivative $\der \Lambda_*/\der \tilde\Lambda$ is bounded above and below by positive constants. This indicates that the recovery guarantee is not specific to the Haar measure, but rather hinges on a mild compatibility condition between the misspecified and true group measures.
\end{remark}

\begin{theorem}\label{thm:consistency_Hausdorff}
Let $\G$ be a continuous group. In addition to the conditions in Theorem~\ref{thm:consistency_in_measure}, suppose that there exist absolute positive constants $c_0$ and $\gamma\geq 1$ such that for all $g \in \G$ and all $\eta \in (0,1)$, we have 
\begin{align}\label{eq:measure_G}
\Lambda_0\big( \big\{ g' \in \G : \|g - g'\| \leq \eta \big\} \big) \geq c_0 \eta^\gamma.
\end{align}
Then for any $\gamma/(\gamma+2)<\alpha<1$, there exists $\sigma_0>0$ depending only on $\alpha,\gamma,\td$, such that for all $\sigma<\sigma_0$, the misspecified MLE satisfies
\begin{align*}
\sup_{g\in\G} d(Pg\theta_*, M_{\hat\theta(\Lambda_0,\sigma^2)}(\G)) \leq \td^{1/2} \sigma^{1-\alpha}.
\end{align*}
In other words, the one-sided Hausdorff distance from  $M_{\theta_*}(\G)$ to  $M_{\hat\theta(\Lambda_0,\sigma^2)}(\G)$ is at most $\td^{1/2} \sigma^{1-\alpha}$.
\end{theorem}

\begin{remark}
We state without proof that the measure regularity assumption \eqref{eq:measure_G} holds with $\gamma = 1$ for the projected MRA setting and with $\gamma = 3$ for the cryo-EM setting. This reflects the intrinsic dimensionality of the underlying group $\G$ in each model. In the projected MRA setting, the group acts via a one-dimensional circle of shifts, while in cryo-EM, $\G \cong \mathsf{SO}(3)$ is a three-dimensional Lie group. Intuitively, the exponent $\gamma$ in \eqref{eq:measure_G} captures the local volume growth around a group element under the Haar measure, and thus coincides with the manifold dimension of the group.
\end{remark}

\begin{remark}
The one-sided Hausdorff distance bound means that every point in the true projected orbit $M_{\theta_*}(\G)$ lies within a small neighborhood (of radius $\td^{1/2} \sigma^{1-\alpha}$) of the estimated orbit $M_{\hat\theta(\Lambda_0,\sigma^2)}(\G)$. In other words, no point in the true orbit is too far from the estimated orbit. This provides a one-sided control: it ensures that the estimated orbit approximately covers the true orbit, but not necessarily the other way around. That is, the estimator does not ``miss'' the true orbit, though it may include extra points that are not close to any point in the true orbit.

A similar one-sided guarantee holds for the bound in Theorem~\ref{thm:consistency_in_measure}. The one-sided nature of both bounds reflects the structure of our analysis: for any parameter $\theta$ whose orbit $M_\theta(\G)$ fails to adequately cover $M_{\theta_*}(\G)$—in the sense that some portion of $M_{\theta_*}(\G)$ lies outside a neighborhood of $M_\theta(\G)$—we can show that the misspecified log-likelihood $\ell(\theta; \Lambda_0, \sigma^2)$ is strictly greater than $\ell(\theta_*; \Lambda_0, \sigma^2)$. Hence, such a $\theta$ cannot be the global minimizer $\hat\theta(\Lambda_0, \sigma^2)$.
In contrast, control on the reverse one-sided Hausdorff distance---that is, ensuring that every point in the estimated (projected) orbit lies near some point in the true (projected) orbit---is more challenging and remains beyond the scope of our current analysis.
\end{remark}

\begin{remark}
Theorem~\ref{thm:consistency_in_measure} also applies when $\G$ is a discrete group, with the same statement and convergence rate. The additional condition \eqref{eq:measure_G} in Theorem~\ref{thm:consistency_Hausdorff} is tailored to the continuous group setting, where it quantifies the local mass of the Haar measure. In the discrete case, this condition naturally reduces to a uniform lower bound on the measure, i.e., $\Lambda_0(g) \geq C$ for all $g \in \G$ and some absolute constant $C > 0$. This lower bound holds automatically if $\Lambda_0$ is uniform over a finite group. Under this condition, the one-sided Hausdorff distance bound in Theorem~\ref{thm:consistency_Hausdorff} holds with the same rate $\sigma^{2-\alpha}$ as in Theorem~\ref{thm:consistency_in_measure}.

Moreover, for the discrete case, one can also establish the reverse one-sided Hausdorff bound under mild generic conditions on $P$ and $\theta_*$. The analysis proceeds along similar lines to that of Theorem~\ref{thm:consistency_in_measure}, using standard probabilistic and geometric arguments. In contrast, in the continuous group setting, the geometry of the projected orbits is substantially more intricate. This complexity motivates the use of tools from real algebraic geometry in our proof of Proposition~\ref{prop:non-self_intersection}.
\end{remark}

\subsubsection{\texorpdfstring{Vanishing Bias When $\sigma\rightarrow 0$}{Vanishing Bias When sigma goes to 0}}\label{sec:bias_sigma_0}
The bounds established in the previous section vanish as $\sigma$ decreases. Although these bounds are one-sided, by leveraging the geometric non-self-intersection property of orbits established in Proposition~\ref{prop:non-self_intersection}, we show that the projected orbit corresponding to the limit $\lim_{\sigma^2 \to 0} \hat\theta(\Lambda_0, \sigma^2)$ coincides exactly with the true projected orbit $M_{\theta_*}(\G)$. This implies exact recovery in the zero-noise limit and ensures that the full Hausdorff distance between the estimated and true projected orbits converges to zero, providing theoretical support for the use of the misspecified MLE in low-noise regimes.

To establish these results, we analyze the orbit geometry under group actions. Proposition~\ref{prop:non-self_intersection} proves that, generically, the projected orbit map $g \mapsto M_\theta(g)$ as defined in \eqref{eq:Mtheta} is injective, using dimension arguments from algebraic geometry. We further verify this property in two practical models—projected continuous MRA and cryo-EM. As a consequence, Corollary~\ref{cor:orbit_subset} rules out partial recovery scenarios, where only part of the true projected orbit is approximated in the limit, thereby addressing the limitation of one-sided Hausdorff control highlighted earlier.

\begin{proposition}[No self-intersection for generic projected orbits]\label{prop:non-self_intersection}
Let $P\in\R^{\td\times d}$ be a linear map from $\R^d$ to $\R^{\td}$. Let $\G$ be a compact and connected subgroup of the orthogonal group of dimension $d$. For $g\in \G$, let $E_{\lambda}(g)$ be the eigenspace of $g$ with eigenvalue $\lambda$. Assume that $P$ is surjective and 
\begin{align*}
\td > \dim \G + \max_{k\geq 0} \big\{k+\dim\{g\in\G\setminus \Id ~|~ \dim E_1(g)\geq k\}\big\},
\end{align*}
where, by convention,  $\dim(\emptyset) = -\infty$. Then for generic $\theta$, the map $M_\theta$ is a homeomorphism from $\G$ onto its image $M_\theta(\G)$ in the standard topology.
\end{proposition}

\begin{remark}
The dimension lower bound condition in Proposition~\ref{prop:non-self_intersection} is essential. The conclusion may fail even if we only assume $\td > \dim \G$. For example, let $L\geq 1$ and consider 
\begin{align*}
\theta = (\varphi,\phi) =(\varphi^{(1)}_1, \varphi^{(1)}_2, \ldots, \varphi^{(L)}_1, \varphi^{(L)}_2, \phi^{(1)}_1, \phi^{(1)}_2, \ldots, \phi^{(L)}_1, \phi^{(L)}_2)\in\R^d,\quad d=4L,
\end{align*}
and 
\begin{align*}
g=\Big(\bigoplus_{l=1}^L \mathcal{H}^{(l)}(\frakg_1) \Big) \oplus \Big(\bigoplus_{l=1}^L \mathcal{H}^{(l)}(\frakg_2) \Big),
\end{align*}
where $\frakg_1,\frakg_2\in[0,2\pi)$ and $\mathcal{H}^{(l)}(\frakg)$ is defined as in \eqref{eq:MRA_rotation_detail}. Define the linear map $P\in\R^{\td\times d}$ by
\begin{align*}
P\theta= \Big(\sum_{i=1}^L \varphi^{(i)}_1, \sum_{i=1}^L\varphi^{(i)}_2, \phi^{(1)}_1, \phi^{(1)}_2, \ldots, \phi^{(L)}_1, \phi^{(L)}_2\Big),\quad \td=2L+2,
\end{align*}
which takes the sum of all $\varphi^{(i)}_1, \varphi^{(i)}_2$ coordinates across $i$, and keeps the $\phi$-coordinates unchanged. One can then show $M_\theta(\G)$ intersects itself for generic $\theta$ even though $\td=2L+2 >2=\dim\G$. Indeed, considering the first two coordinates in $M_\theta(\G)$, which only depend on the $\frakg_1$ action, self-intersection occurs for generic $\varphi$. Since the actions $\frakg_1$ and $\frakg_2$ are independent, these self-intersections persist in the full map.
\end{remark}

As an immediate consequence of Proposition~\ref{prop:non-self_intersection}, we obtain the following corollary, which verifies the non-self-intersection property in the settings of projected continuous MRA and cryo-EM, respectively. An example illustrating this property in the projected MRA setting is shown in Figure~\ref{fig:curve}.

\begin{corollary}\label{cor:no_self_intersection_examples}
\leavevmode
\begin{itemize}
\item In the case of projected continuous MRA, \(M_\theta\) is a homeomorphism onto its image for generic \(\theta\) if \(L \geq 3\).
\item In the case of (projected) cryo-EM, assuming \(S_0 = S_1 = \cdots = S_L = S\geq 1\), \(M_\theta\) is a homeomorphism onto its image for generic \(\theta\) if \(L \geq 1\).
\end{itemize}
\end{corollary}

\begin{figure}[ht]
    \centering
    \includegraphics[width=0.48\linewidth]{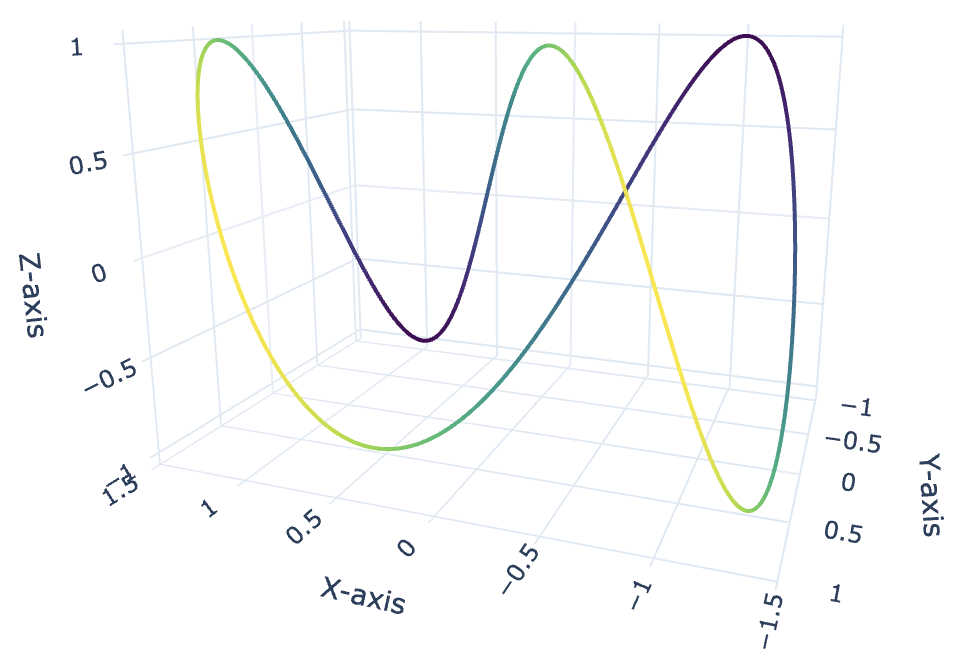}
    \includegraphics[width=0.489\linewidth]{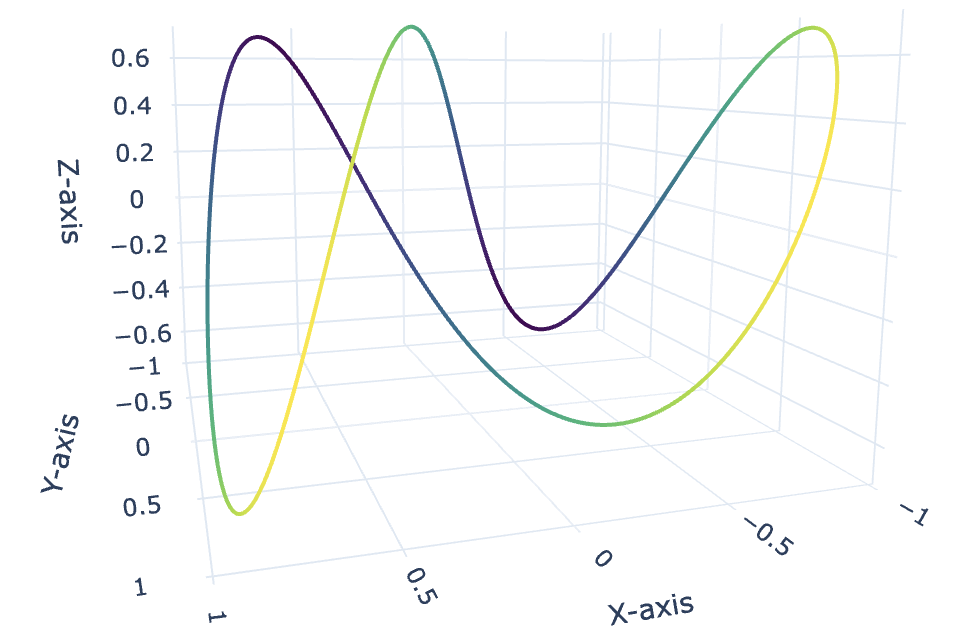}
    \caption{An illustration of $M_\theta$ for the projected  MRA with $L=3$. Though $\td =4$, since the first coordinate is fixed, $M_{\theta}$ can be visualized in $\mathbb{R}^3$. Left: $\theta=(0,1,1,1,0,0,1)$. Right: a randomly generated $\theta$. The curves are colored to illustrate that they are not self-intersecting.}
    \label{fig:curve}
\end{figure}

We further derive a general consequence of Proposition~\ref{prop:non-self_intersection}, showing that under connected group actions, any inclusion between generic orbits must in fact be equality.

\begin{corollary}\label{cor:orbit_subset}
Assume the conditions in Proposition \ref{prop:non-self_intersection} hold. Let $\theta,\theta'\in\R^d$ be generic points as described in Proposition \ref{prop:non-self_intersection}. If $M_\theta(\G)\subseteq M_{\theta'}(\G)$, then $M_\theta(\G)=M_{\theta'}(\G)$ as a subset of $\R^{\td}$. 
\end{corollary}

We conclude this section by showing that, in the zero-noise limit, the projected orbit associated with the misspecified MLE converges exactly to the true projected orbit.

\begin{theorem}[Exact projected orbit recovery in the zero-noise limit]\label{thm:orbit_recovery}
Assume the conditions in Theorem~\ref{thm:consistency_in_measure} and Proposition~\ref{prop:non-self_intersection} hold. Let $\theta_*\in\R^d$
be a generic point as described in Proposition~\ref{prop:non-self_intersection}, and suppose the limit $\bar\theta:=\lim_{\sigma^2\to 0} \hat\theta(\Lambda_0,\sigma^2)$ exists\footnote{Here, the existence of the limit $\bar\theta := \lim_{\sigma^2 \to 0} \hat\theta(\Lambda_0,\sigma^2)$ means that there exists a representative in the projected orbit $M_{\hat\theta(\Lambda_0,\sigma^2)}(\G)$ that converges (in Euclidean norm) to a point $\bar\theta$ as $\sigma^2 \to 0$.} and is also generic. Then we have
\begin{align}\label{eq:orbit_zero_noise}
M_{\bar\theta} (\G)= M_{\theta_*}(\G).
\end{align}
In particular, the Hausdorff distance between $M_{\hat\theta(\Lambda_0,\sigma^2)}(\G)$ and $M_{\theta_*}(\G)$ vanishes as $\sigma\to 0$, i.e.,
\begin{align}\label{eq:Hausdorff_consistency}
\lim_{\sigma^2 \to 0} d_H\Big(M_{\hat\theta(\Lambda_0,\sigma^2)}(\G), M_{\theta_*}(\G)\Big) = 0,
\end{align}
where $d_H(A,B)$ denotes the Hausdorff distance between sets $A,B \subseteq \R^{\td}$, defined as
\begin{align*}
d_H(A,B) := \max \Big\{ \sup_{a \in A} \inf_{b \in B} \|a - b\|,\ \sup_{b \in B} \inf_{a \in A} \|a - b\| \Big\}.
\end{align*}
\end{theorem}

\subsection{\texorpdfstring{Joint Estimation of $\theta_*$ and $\Lambda_*$}{Joint Estimation of theta* and Lambda*}}\label{sec:joint_MLE}

To address the bias introduced by the misspecified MLE and the unknown nature of the underlying data-generating measure $\Lambda_*$, we propose a joint estimation framework that explicitly parameterizes $\Lambda_*$. Motivated primarily by the cryo-EM application, we focus in this section on the setting where $\G = \SO(3)$. While our presentation emphasizes this case for clarity, the proposed methodology and analysis extend naturally to other group settings.

\subsubsection{\texorpdfstring{Parametrization of Rotational Distributions in $\mathsf{SO}(3)$}{Parametrization of Rotational Distributions in SO(3)}}\label{sec:parametrization_rotations}
Consider a possibly non-uniform rotational distribution $\lambda(\frakg)$ for $\frakg\in\mathsf{SO}(3)$ that admits a square-integrable probability density $\rho(\frakg)\in L_2(\mathsf{SO}(3),\R)$. Since the set of all the entries of real Wigner D-matrices forms an orthonormal basis for $L_2(\mathsf{SO}(3),\R)$ (see Appendix \ref{sec:real_SH}), we may expand $\der \lambda(\frakg) =\rho(\frakg)\mathrm{d} \frakg$ as 
\begin{align}\label{eq:rot_dist_expansion}
\rho(\frakg)=\sum_{p=0}^\infty \sum_{u,v=-p}^p \mathcal{B}_{puv} \mathcal{D}_{uv}^{(p)} (\frakg),
\end{align}
where $\frakg\in \mathsf{SO}(3)$, $\mathcal{D}_{uv}^{(p)} (\frakg)$ denotes the $uv$-th entry of the real Wigner D-matrix at frequency $p$, and $\mathcal{B}_{puv}\in\R$ is the corresponding coefficient. Since the integral of any density function equals one, we require
\begin{align*}
\mathcal{B}_{000}=\int\sum_{p=0}^\infty \sum_{u,v=-p}^p \mathcal{B}_{puv}\mathcal{D}_{uv}^{(p)}(\frakg)\der\frakg=\int \rho(\frakg) \der\frakg =1.
\end{align*}
One additional constraint is the non-negativity of the density, i.e.,
\begin{align*}
\sum_{p=0}^\infty \sum_{u,v=-p}^p\mathcal{B}_{puv} \mathcal{D}_{uv}^{(p)}(\frakg)\geq 0,
\end{align*}
for any $\frakg\in\mathsf{SO}(3)$. In the sequel, we restrict our attention to bandlimited densities, i.e.,
\begin{align*}
\rho(\frakg)=\sum_{p=0}^P \sum_{u,v=-p}^p \mathcal{B}_{puv} \mathcal{D}_{uv}^{(p)} (\frakg),
\end{align*}
for some band limit $P$. 

The following lemma shows that under an additional symmetry assumption—namely, \textit{in-plane uniformity}—the expansion of the rotational distribution can be further simplified. This symmetry stems from the fact that, in the context of cryo-EM, while molecules often exhibit preferred orientations in 3D space, there is generally no physical reason for them to favor particular in-plane rotations. This makes in-plane rotational invariance a natural and practically motivated modeling assumption. Consequently, we focus on the case of non-uniform rotational distributions that are invariant under in-plane rotations, as such distributions better reflect the physical constraints and experimental setup of cryo-EM imaging (see also \cite{sharon2020method}). For simplicity, we fix the image plane in cryo-EM to be perpendicular to the $z$-axis.

\begin{lemma}[In-plane uniform rotational distribution]\label{lem:inplane_uniform_density}
Assume a rotational distribution $\lambda(\frakg)$ for $\frakg\in\mathsf{SO}(3)$ admits a smooth bandlimited probability density $\rho(\frakg)$ with band limit $\bar{P}$. Further assume $\lambda(\frakg)$ is invariant to in-plane rotations and fix the image plane as perpendicular to the $z$-axis, that is,
\begin{align*}
\rho(\frakg)=\rho(\frakg z(\alpha))
\end{align*}
for all $\frakg\in \mathsf{SO}(3)$ and rotations $z(\alpha)$ of $\alpha\in\R$ radians around the $z$-axis. Then the density has the expansion
\begin{align}\label{eq:rot_dist_inplane_unif_expansion}
\rho(\frakg)=\sum_{p=0}^{\bar{P}} \sum_{u=-p}^p \mathcal{B}_{pu0} \mathcal{D}_{u0}^{(p)} (\frakg).
\end{align}
For simplicity, such density can be written as
\begin{align}\label{eq:rot_dist_inplane_unif_expansion_1}
\rho(\frakg)=\sum_{p=0}^{\bar{P}} \sum_{u=-p}^p \mathcal{B}_{pu} \mathcal{D}_{u0}^{(p)} (\frakg).
\end{align}
\end{lemma}

\begin{remark}
The main difference between our parametrization and that in \cite{sharon2020method} lies in the choice of basis functions for the expansion. In \cite{sharon2020method}, the complex Wigner D-matrices are used as the basis, which generally leads to complex-valued expansion coefficients. For a real-valued density $\rho$, this imposes additional conjugate symmetry constraints on those coefficients to ensure the density remains real. In contrast, we adopt the real Wigner D-matrices as the basis functions, making the expansion coefficients naturally real-valued without any extra constraints. Similarly, the structural properties induced by the in-plane uniformity assumption require corresponding analysis within this real-valued framework.  
\end{remark}

\subsubsection{Consistency of the Joint MLE}\label{sec:consistency_joint_MLE}
Recall (\ref{eq:likelihood_cryoEM}) presents the likelihood function for cryo-EM when $\lambda_*$ is known a priori. In what follows, we consider the joint estimation of both the structural parameter $\theta_*$ and the underlying rotational distribution $\lambda_*$.

\begin{assumption}\label{assumption:distribution}
Assume the true rotational distribution $\lambda_*(\frakg)$ on $\mathsf{SO}(3)$ is invariant under in-plane rotations (as described in Section~\ref{sec:parametrization_rotations}) and admits a square-integrable density $\rho_*$ with an expansion of the form~\eqref{eq:rot_dist_inplane_unif_expansion_1} of bandlimit~$\bar{P}$. Furthermore, assume its expansion coefficients are bounded in absolute value by some constant $b > 0$.
\end{assumption}

The boundedness assumption on the coefficients guarantees the density is also bounded for any $\frakg\in\mathsf{SO}(3)$.
When Assumption \ref{assumption:distribution} holds, from the previous section, the density can be written as $\rho_*(\frakg)=\sum_{p=0}^{\bar{P}} \sum_{u=-p}^p (\mathcal{B}_*)_{pu} \mathcal{D}_{u0}^{(p)} (\frakg)$, where $\mathcal{B}_*\in B$ is the coefficients of the distribution and $B$ is the parameter space defined as
\begin{align*}
B = \Big\{(\mathcal{B}_{pu}): \mathcal{B}_{00}=1, \sum_{p=0}^{\bar{P}} \sum_{u=-p}^p \mathcal{B}_{pu} \mathcal{D}_{u0}^{(p)} (\frakg)\geq 0 \text{ for all }\frakg \in \mathsf{SO}(3), \mathcal{B}_{pu}\in[-b,b]\text{ for all }p,u\Big\}.
\end{align*}
As a result, the joint likelihood function can be written as
\begin{align*}
p_{\theta,\mathcal{B}}(y) =  \int_{\mathsf{SO}(3)} \frac{1}{(2\pi\sigma^2)^{\td/2}}
\exp\left(-\frac{\|y-P \check{\mathcal{D}}(\frakg)\theta\|^2}{2\sigma^2}\right)\sum_{p=0}^{ \bar{P} } \sum_{u=-p}^p \mathcal{B}_{pu} \mathcal{D}_{u0}^{(p)}(\frakg)\mathrm{d}\frakg,
\end{align*}
This formulation enables joint estimation of $(\theta_*, \mathcal{B}_*)$ via maximum likelihood. 

We next establish the consistency of the resulting estimator under the cryo-EM model.

\begin{theorem}\label{thm:consistency_1}
Assume Assumption \ref{assumption:distribution} holds and $\theta_*$ is bounded, i.e., $\|\theta_*\|\leq r$ for some $r>0$. In addition, assume the model is identifiable up to the joint group orbit.
Define the joint MLE:
\begin{align*}
    (\hat \theta_n,\hat{\mathcal{B}}_n) = \argmin_{\|\theta\|\leq r,\mathcal{B}\in B} -\frac{1}{n} \sum_{i=1}^n\log p_{\theta,\mathcal{B}}(Y_i).
\end{align*}
Then,  in probability, we have $(\hat \theta_n,\hat{\mathcal{B}}_n)$ converges to some $(\theta',\mathcal{B}') \in  \{(\theta,\mathcal{B}):\|\theta\|\leq r, \mathcal{B}\in B, \orbit_{(\theta,\Lambda)}=\orbit_{(\theta_*,\Lambda_*)}$  where the density of $\Lambda$ is $\sum_{p=0}^{\bar{P}} \sum_{u=-p}^p \mathcal{B}_{pu} \mathcal{D}_{u0}^{(p)} (\frakg),\forall \frakg \in \mathsf{SO}(3)\}$.
\end{theorem}

\begin{remark}
This result lays the theoretical foundation for consistent estimation via joint maximum likelihood, but the optimization problem it defines is generally non-convex and may be computationally challenging. In practice, this motivates the use of iterative or alternating optimization schemes that alternate between updating $\theta$ and the distributional parameters $\mathcal{B}$. Moreover, while the theorem assumes a specific parametrization of the distribution on $\mathsf{SO}(3)$ using Wigner D-functions, other parametrizations—such as mixtures of Fisher/Langevin distributions—can also be employed in practical applications, especially when they offer computational advantages or better modeling flexibility.   
\end{remark}

\section{Simulations}\label{sec:simulations}

We perform a numerical analysis using continuous MRA as an example (see Example~\ref{sec:model_MRA}). Figure~\ref{fig:bias} shows the average relative recovery error (over 50 trials) as a function of the noise level for the EM algorithm, which is used to solve the MLE. The ground truth signals have length $d=7$ (i.e., $L=3$), with entries independently drawn from a uniform distribution on \([0,1]\). While our theoretical analysis mainly consider the population-level setting (i.e., $n\to \infty$), we perform simulations with a finite sample size of $n=1000$. Recall $\lambda_*$ is the true data-generating rotation distribution. Denote $\mathrm{WN}$ as the wrapped normal distribution\footnote{A wrapped normal distribution $\mathrm{WN}(\mu,\sigma^2)$ is obtained by wrapping $\mathrm{N}(\mu,\sigma^2)$ around the unit circle $[0, 2\pi]$. That is, its density at $\alpha\in[0,2\pi]$ is $\frac{1}{\sqrt{2\pi}\sigma}\sum_{k=-\infty}^{+\infty}\exp(-(\alpha-\mu + 2\pi k)^2/(2\sigma^2))$.} and $\mathrm{Unif}$ as the uniform distribution.
Data are generated under four different scenarios:
\begin{itemize}{
  \item {Blue dashed curve:} without projection, $\lambda_*=\mathrm{Unif}[0, 2\pi]$.
  \item {Red dashed curve:} without projection, $\lambda_*=\mathrm{WN}(0, 0.01)$.
  \item {Blue solid curve:} with a two-fold projection,  $\lambda_*=\mathrm{Unif}[0, 2\pi]$.
  \item {Red solid curve:} with a two-fold projection, $\lambda_*=\mathrm{WN}(0, 0.01)$.}
\end{itemize}
To further examine how the degree of deviation of the underlying true rotation distribution from the uniform distribution under projection, we include additional solid curves in green, pink, and brown, corresponding to the same two-fold projection setting but with $\lambda_*$ being $\mathrm{WN}(0, 0.1)$, $\mathrm{WN}(0, 0.5)$, and $\mathrm{WN}(0, 0.8)$, respectively.

\begin{figure}[ht]
  \centering
  \includegraphics[width=0.8\textwidth]{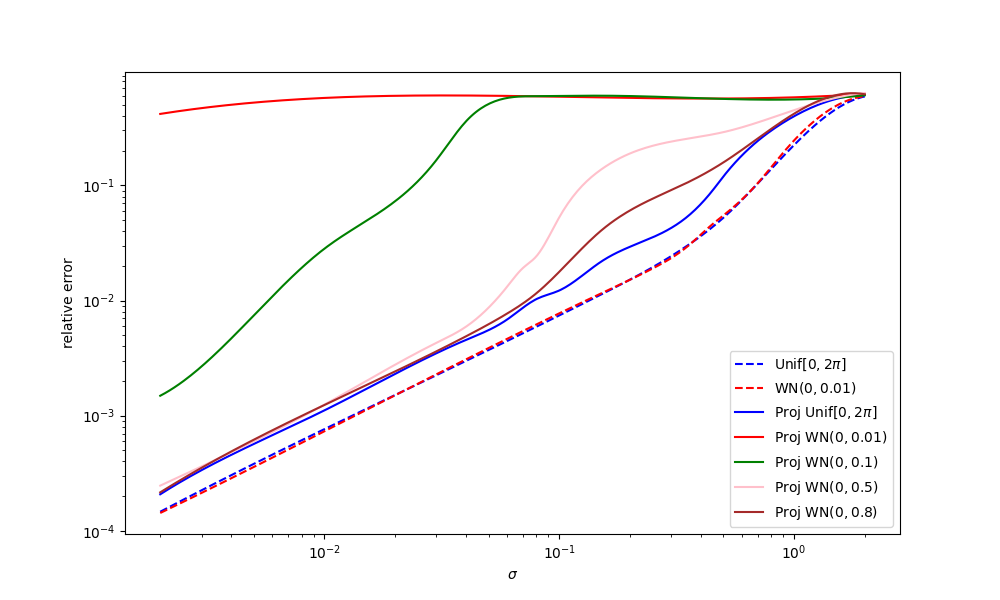}
  \caption{Relative error as a function of the noise level in estimating a signal of length $d=7$ (i.e., $L=3$) with entries independently drawn from the uniform distribution on $[0,1]$. The signal is recovered from $n=1000$ continuous MRA observations, as described in Example~\ref{sec:model_MRA}, using the EM algorithm with a uniform rotation prior in all cases. The figure includes both the setting without projection (dashed lines) and the setting with an additional two-fold projection (solid lines). The blue dashed line represents the correctly specified setting without projection, where both the data-generating rotation distribution and the EM prior are uniform on $[0, 2\pi]$. The red dashed line corresponds to a misspecified setting without projection, in which the EM algorithm assumes a uniform prior over $[0, 2\pi]$, while the true rotation distribution is a wrapped normal distribution $\mathrm{WN}(0, 0.01)$. The blue solid line shows the correctly specified setting with a two-fold projection, where both the true rotation distribution and the EM prior are uniform. The red solid line corresponds to a misspecified setting with a two-fold projection, where the EM algorithm assumes a uniform prior, but the true rotation distribution is $\mathrm{WN}(0, 0.01)$. To illustrate the effect of increasing deviation from the uniform distribution, the green, pink, and brown solid lines are included, corresponding to settings where the true rotation distributions are $\mathrm{WN}(0, 0.1)$, $\mathrm{WN}(0, 0.5)$, and $\mathrm{WN}(0, 0.8)$, respectively. 
}\label{fig:bias}
\end{figure}

The relative recovery error is defined as
$
\min_{g \in \G} {\|\hat{\theta} - g \theta_*\|}/{\|\theta_*\|},
$
where $\hat{\theta}$ is the estimated signal, $\theta_*$ is the ground truth signal. This corresponds to the normalized orbit distance between the estimate and the ground truth under the group action. In all cases, the EM algorithm uses a uniform prior over $[0,2\pi]$ for the rotation distribution, regardless of the true underlying distribution. Thus, the two blue curves correspond to settings where the assumed prior matches the true distribution (i.e., correctly specified), while the two red curves correspond to misspecified settings where the true distribution is a wrapped normal. The EM iterations are terminated when the difference between consecutive iterations dropped below $10^{-6}$.

In the correctly specified setting (the two blue curves), we observe a clear phase transition in the slope of the recovery error curve around \( \sigma = 0.5 \). Specifically, in the low-noise regime (\( \sigma \lesssim 0.5 \)), the error remains small and grows slowly with increasing noise, while in the high-noise regime (\( \sigma \gtrsim 0.5 \)), the error exhibits a steeper increase. This phenomenon is consistent with the empirical observation made in the seminal work of~\cite{sigworth1998maximum} in the context of cryo-EM, and is further supported by the theoretical results of~\cite{perry2019sample}, which established that in the MRA setting, the minimax estimation error scales as $\sigma^3$ in the high-noise regime, in contrast to the typical $\sigma$-scaling in models without hidden group structure.

For the setting without projection (the two dashed curves), the recovery error under the misspecified rotation prior (red dashed) closely matches that of the correctly specified case (blue dashed) across all noise levels. This suggests that, in the absence of projection, misspecifying the rotation prior has a negligible effect on recovery performance, which numerically verifies our results in Section \ref{sec:landscape_without_projection}.

In contrast, when projection is applied (the solid curves), the effect of misspecification becomes substantial. The red solid curve exhibits significantly higher recovery error than the blue solid curve, indicating that rotation distribution misspecification introduces a non-negligible bias in the presence of projection—even under very low-noise conditions. In this setting, the true rotation distribution is a wrapped normal distribution $\mathrm{WN}(0, 0.01)$. This distribution is sharply concentrated around $0$ and $2\pi$, representing a strong deviation from the assumed uniform prior. 
As a result, a recovery bias emerges, which numerically verifies our results in Section \ref{sec:landscape_with_projection}.

The three other solid curves (green, pink, and brown) further examine the effect of deviation from the uniform prior. As the variance of the wrapped normal distribution increases—that is, as the underlying rotation distribution $\lambda_*$ becomes closer to uniform—the relative recovery error decreases and gradually approaches that of the correctly specified case (the blue solid curve). Moreover, the region where the misspecification-induced bias is most apparent becomes narrower. This effect is most significant in the high-noise regime: at low-noise levels, the bias is reduced due to favorable signal-to-noise conditions (as predicted by our theory in Section \ref{sec:landscape_with_projection}), while when the noise level is extremely high, the relative recovery error approaches the maximum possible value. These results suggest that the severity of the bias introduced by a misspecified MLE depends jointly on the noise level and the extent of deviation between the underlying rotation distribution and the assumed uniform prior.

\section{Discussion}\label{sec:discussion}

This work reveals a fundamental dichotomy in the behavior of misspecified MLE for group orbit recovery: the effectiveness of the estimator depends critically on whether the measurement operator—such as the tomographic projection in cryo-EM—preserves or distorts the geometric invariance inherent in the data. In non-projective models, the misspecified MLE exhibits surprising resilience—it can recover the true signal exactly, even when using an incorrect uniform distribution in place of the true group measure. However, under projection, such misspecification leads to systematic bias that becomes increasingly significant at higher noise levels.

To combat this bias in cryo-EM reconstruction, we introduce a joint estimation framework that explicitly parameterizes the orientation distribution. This approach enables simultaneous recovery of the molecular structure and the underlying rotation distribution. While our current implementation focuses on a specific parameterization using real Wigner matrices, the framework’s flexibility invites exploration of alternative parameterizations; developing their efficient computational realizations remains a vital research frontier. More significantly, by accurately modeling the underlying group measure, this joint paradigm eliminates systematic biases inherent in uniform-prior methods and shows  potential for enhancing reconstruction fidelity under high-noise conditions prevalent in experimental data.

Beyond computational considerations, deeper theoretical questions arise. While our analysis shows that the bias induced by misspecification diminishes as the noise level decreases under projection, an open challenge is to quantitatively characterize how the discrepancy between the group measure assumed in the recovery method and the true data-generating one governs the resulting estimation bias. Our projected MRA simulations provide initial empirical evidence, but a rigorous theoretical framework that relates distributional divergence to bias magnitude across different noise regimes remains essential. Such a framework would offer valuable insights both theoretically and for practical applications. Moreover, the theoretical techniques developed in this work hold promise for addressing more general continuous Gaussian mixture models and latent variable frameworks, for example manifold learning, nonlinear factor models, and emerging generative models like diffusion models.

\begin{appendix}
\numberwithin{equation}{section}

\section{Additional Examples in the Generalized Orbit Recovery Model}\label{sec:additional_examples}

This appendix presents additional examples that are more technically involved, including spherical registration, unprojected cryo-EM and cryo-ET, and projected cryo-EM.

\subsection{Spherical Registration}\label{sec:model_SR_appendix}
We consider estimating a function on the unit sphere $f:\mathcal{S}^2\to \R$. We parametrize the unit sphere by the latitude $\varphi_1\in[0,\pi]$ and longitude $\varphi_2\in[0,2\pi)$, and represent the rotation of $f$ by an element $\frakg\in\mathsf{SO}(3)$ as $f_\frakg(\varphi_1,\varphi_2)= f(\frakg^{-1}\cdot (\varphi_1,\varphi_2))$. Consider we observe samples 
\begin{align}\label{eq:model_SR1}
f_\frakg(\varphi_1,\varphi_2)\der (\varphi_1,\varphi_2)+\sigma \der W(\varphi_1,\varphi_2),
\end{align}
where $\frakg\sim \lambda_*$ for some possibly non-Haar measure $\lambda_*$ on $\mathsf{SO}(3)$, $\der (\varphi_1,\varphi_2)$ denotes the surface area measure on $\mathcal{S}^2$, and $\der W(\varphi_1,\varphi_2)$ denotes a standard Gaussian process on $\mathcal{S}^2$. As detailed in Appendix \ref{sec:real_SH}, the real spherical harmonics form a complete orthonormal basis for $L_2(\mathcal{S}^2,\R)$. Assume $f\in L_2(\mathcal{S}^2, \R)$ and admits a bandlimited representation under the real spherical harmonics (see \cite[Section D.2]{fan2024maximum})
\begin{align}\label{eq:model_SR2}
f(\varphi_1,\varphi_2)=\sum_{l=0}^L\sum_{m=-l}^l \theta_m^{(l)} \mathbcal{y}_{lm}(\varphi_1,\varphi_2),
\end{align}
where $L\geq 1$ is the bandlimit and $\{\mathbcal{y}_{lm}:l\geq 0, -l\leq m\leq l\}$ are the real spherical harmonics (see Appendix \ref{sec:real_SH}). Notably, the subspace of such bandlimited functions is closed under $\mathsf{SO}(3)$ rotations over $\mathcal{S}^2$. The observation model \eqref{eq:model_SR1} is equivalent to observing all the coefficients of $f_\frakg$ in this basis of real spherical harmonics. Writing
\begin{align*}
\theta = (\theta_m^{(l)}:0\leq l\leq L~\text{and}~-l\leq m\leq l)\in\R^d,\quad d=(L+1)^2,
\end{align*}
for the vector of real spherical harmonic coefficients, the rotation $\frakg$ acting on $f$ can be represented in the space of these coefficients by the matrix multiplication $\theta \mapsto g\theta$ where $g\in \G$ is the block-diagonal matrix
\begin{align*}
g = \mathcal{D}(\frakg)=\bigoplus_{l=0}^L \mathcal{D}^{(l)}(\frakg)\in \R^{d\times d}, \quad\text{ for }\frakg\in\mathsf{SO}(3),
\end{align*}
with $\mathcal{D}^{(l)}(\frakg)$ the real Wigner D-matrix at frequency $l$ as defined in \eqref{eq:realWigner}. For the proof details, we refer the readers to Appendix \ref{sec:real_Wigner} and \cite[Lemma D.3]{fan2024maximum}. 

Thus, this observation model for the real spherical harmonic coefficients of $f$ is a special case of \eqref{eq:projectedmodel} where $P=\Id$, $\G$ is isomorphic to $\mathsf{SO}(3)$, and $\Lambda_*$ is the unique probability measure on $\G$ induced by the isomorphism between $\G=\{\mathcal{D}(\frakg):\frakg\in\mathsf{SO}(3)\}$ and $\mathsf{SO}(3)$, with respect to $\lambda_*$. The likelihood function can be written as 
\begin{align*}
    p_{\theta}(y) = 
\int_{\mathsf{SO}(3)} \frac{1}{(2\pi\sigma^2)^{d/2}}
\exp\left(-\frac{\|y- \mathcal{D}(\frakg)\theta\|^2}{2\sigma^2}\right)\der
\lambda_*(\frakg).
\end{align*}

\subsection{Unprojected Cryo-EM and Cryo-ET}\label{sec:model_cryoET_appendix}

We consider estimating a function $f:\R^3\to\R$, and the action of $\mathsf{SO}(3)$ on $\R^3$ is given by rotation about the origin. Write $f_\frakg
(x)=f(\frakg^{-1}\cdot x)$ for $x\in\R^3$ and $\frakg\in\mathsf{SO}(3)$. Consider we observe samples
\begin{align}\label{eq:model_cryoET}
f_\frakg(x)\der x+\sigma \der W(x),
\end{align}
where $\frakg\sim \lambda_*$ for some possibly non-Haar measure $\lambda_*$ on $\mathsf{SO}(3)$, and $\der W(x)$ denotes a standard Gaussian process on $\R^3$. This is a simplified, unprojected model of single-particle reconstruction in cryo-EM. In practice, a more realistic model would include steps such as contrast transfer function (CTF) correction, centering due to shifts, and tomographic projection (also see \cite{singer2018mathematics}). In the next section, we will incorporate the tomographic projection. Additionally, this model may be of independent interest in the context of cryo-ET, with no ``missing wedge problem'' (see \cite{turk2020promise, watson2024advances, zhang2019advances}).

Our model setup mainly follows \cite[Section 4.2]{fan2024maximum}, and is also similar to that in \cite[Section 4.6]{bandeira2023estimation}. We parametrize both the original signal domain and its Fourier domain by spherical coordinates $(r,\varphi_1,\varphi_2)$ with radius $r\geq 0$, latitude $\varphi_1\in[0,\pi]$, and longitude $\varphi_2\in[0,2\pi)$. Unlike the spherical registration model (Section \ref{sec:model_SR}), which represents $f:\mathcal{S}^2\to\R$ using an orthonormal basis of spherical harmonics, here we model $f:\R^3\to\R$ through a bandlimited representation for its Fourier transform $\hat{f}:\R^3\to\C$. Specifically, assume that $f\in L_2(\R^3,\R)$ and we expand $\hat{f}$ in a complete orthonormal complex basis $\{\hat{h}_{lsm}\}$, where each basis function is the product of the complex spherical harmonics $y_{lm}(\varphi_1,\varphi_2)$ and a set of radial basis functions $j_s(r)$ defined on $[0,\infty)$ satisfying completeness and orthonormality in the space of square-integrable functions with respect to the measure $r^2\der r$. By the Parseval relation, the inverse Fourier transforms $\{h_{lsm}\}$ constitute a complete orthonormal complex basis in the original signal domain of $f$. Alternatively, one can model $\hat{f}$ as being supported on a ball of finite radius $R$. In this case, the radial basis functions can be chosen as a complete orthonormal basis supported on $[0,R]$ with respect to the measure $r^2\der r$ (see e.g. \cite{kileel2024fast}). For our purpose, the specific choice of radial basis functions does not affect the formulation or results. 

Fix integer bandlimits $L\geq 1$ and $S_0,\ldots,S_L\geq 1$ and define the index set 
\begin{align}\label{eq:index_set}
\mathcal{I}=\{(l,s,m):0\leq l\leq L,~ 1\leq s\leq S_l,~-l\leq m\leq l\},~~~~~~d=|\mathcal{I}|=\sum_{l=0}^L (2l+1)S_l.
\end{align}
Assume that $f$ is \textit{$(L,S_0,\ldots,S_L)$-bandlimited with respect to the basis $\{h_{lsm}\}$}, meaning its Fourier transform admits a finite basis representation in terms of $\{\hat{h}_{lsm}\}$:
\begin{align}\label{eq:expansion_cryoET_Fourier}
\hat{f}=\sum_{(l,s,m)\in\mathcal{I}} u_{m}^{(ls)} \hat{h}_{lsm},~~~~~~u=(u_m^{(ls)}:(l,s,m)\in\mathcal{I})\in\C^d,
\end{align}
or equivalently itself admits a finite basis representation in terms of $\{h_{lsm}\}$:
\begin{align}\label{eq:expansion_cryoET}
f=\sum_{(l,s,m)\in\mathcal{I}} u_{m}^{(ls)} h_{lsm}.
\end{align}
This corresponds to modeling $\hat{f}$ up to the spherical frequency $L$, and up to the radial frequency $S_l$ for each spherical component $l=0,\ldots,L$. For real-valued $f$, setting $u=\tilde{Q}^\T \theta$, where the unitary transformation matrix $\tilde{Q}\in\C^{d\times d}$ is explicitly defined in \eqref{eq:Qtilde}, we further obtain a real basis representation for the original signal:
\begin{align}\label{eq:expansion_cryoET_real}
f=\sum_{(l,s,m)\in\mathcal{I}} \theta_{m}^{(ls)} \mathbcal{h}_{lsm},~~~~~~\theta=(\theta_m^{(ls)}:(l,s,m)\in\mathcal{I})\in\R^d,
\end{align}
for a complete orthonormal real basis $\{\mathbcal{h}_{lsm}\}$ of $L_2(\R^3,\R)$. Here we order both the complex and real coefficients appropriately according to the tuple $(l,s,m)$ to form the vectors $u$ and $\theta$, such that both the unitary transformation $\tilde{Q}$ and the following rotation representation exhibit a block-diagonal structure. Further details of the setup can be found in Appendix \ref{sec:function_basis_cryoET}.

The rotation $\frakg$ acting on $f$ can then be represented in the space of the real coefficients by the matrix multiplication $\theta\mapsto g\theta$ where $g\in \G$ is the block-diagonal matrix
\begin{align}
g = \check{\mathcal{D}}(\frakg)=\bigoplus_{l=0}^L \bigoplus_{s=1}^{S_l} \check{\mathcal{D}}^{(ls)}(\frakg)\in \R^{d\times d}, \quad\text{ for }\frakg\in\mathsf{SO}(3), \label{eq:cryoEM_rotation}
\end{align}
with $\check{\mathcal{D}}^{(ls)}(\frakg)$ the orthogonal transformation matrix for the $(l,s)$-th spherical shell component for $0\leq l\leq L$ and $1\leq s\leq S_l$. For the proof details, we refer the readers to Appendix \ref{sec:rot_representation_cryoET} and \cite[Lemma D.5]{fan2024maximum}. 

Thus, this observation model for the real coefficients of $f$ under the complete orthonormal basis $\{\mathbcal{h}_{lsm}\}$ is a special case of \eqref{eq:projectedmodel}, where $P=\Id$, $\G$ is isomorphic to $\mathsf{SO}(3)$, and $\Lambda_*$ is the unique probability measure on $\G$ induced by the isomorphism between $\G=\{\check{\mathcal{D}}(\frakg):\frakg\in\mathsf{SO}(3)\}$ and $\mathsf{SO}(3)$, with respect to $\lambda_*$. The likelihood function can be written as 
\begin{align*}
    p_{\theta}(y) = 
\int_{\mathsf{SO}(3)} \frac{1}{(2\pi\sigma^2)^{d/2}}
\exp\left(-\frac{\|y- \check{\mathcal{D}}(\frakg)\theta\|^2}{2\sigma^2}\right)\der
\lambda_*(\frakg).
\end{align*}

\subsection{(Projected) Cryo-EM}\label{sec:model_cryoEM_appendix}

We extend the model from the previous section to include the tomographic projection that occurs in the practice of cryo-EM. As before, we aim to estimate a function $f:\R^3\to \R$. In this projected model, the signal $f$ undergoes a rotation in $\mathsf{SO}(3)$ about the origin, followed by the tomographic projection. The observed samples on the projection domain $\R^2$ are given by
\begin{align}\label{eq:model_cryoEM}
(\Pi \circ f_\frakg)(x) \der x+\sigma \der W(x)
\end{align}
for $x=(x_1,x_2)\in\R^2$ and $\frakg\in\mathsf{SO}(3)$, with $\frakg\sim \lambda_*$ for some possibly non-Haar measure $\lambda_*$ on $\mathsf{SO}(3)$. The term $\der W(x)$ denotes a standard Gaussian white noise process on $\R^2$. The tomographic projection $\Pi$ is defined as
\begin{align*}
(\Pi\circ f_\frakg)(x_1,x_2) = \int_\R f_\frakg(x_1,x_2,x_3)\der x_3.
\end{align*}
We assume that $f\in L_2(\R^3,\R)$, and this implies that $\Pi\cdot f_\frakg\in L_2(\R^2,\R)$.

Our model setup mainly follows \cite[Section 4.3]{fan2024maximum}; also see \cite[Section 4.6]{bandeira2023estimation}. We again parametrize both the original signal domain and its Fourier domain by spherical coordinates $(r,\varphi_1,\varphi_2)$, where $r\geq 0$ is the radius, $\varphi_1\in[0,\pi]$ is the latitude, and $\varphi_2\in[0,2\pi)$ is the longitude. We also parametrize both the projection domain and its Fourier domain by polar coordinates $(r,\varphi_2)$, where $r\geq 0$ is the radius and $\varphi_2\in [0,2\pi)$ is the angle. We model the 2-D Fourier transform of the tomographic projection $\Pi\cdot f$ on $\R^2$ using a complete orthonormal basis $\{\hat{q}_{sm}\}$ in $L_2(\R^2,\C)$, where each basis function is the product of the complex Fourier basis function $(2\pi)^{-1/2} \exp(\i m\varphi_2)$ and a set of radial functions $\tilde{j}_s(r)$ \footnote{In Section \ref{sec:model_cryoET}, the radial functions $j_s(r)$ were orthonormal with respect to the measure $r^2\der r$. Here, we use the measure $r\der r$ instead, to ensure that the corresponding basis functions form a complete and orthonormal basis in the observation space. If $\{\tilde{j}_s: s=1,\ldots,S\}$ and $\{j_s: s=1,\ldots,S\}$ have the same linear span, then the two spaces of bandlimited functions \eqref{eq:expansion_cryoET} and \eqref{eq:expansion_cryoEM} coincide.}. The 3-D Fourier transform of the original signal $f$ is then represented using a corresponding basis $\{\hat{\tilde{h}}_{lsm}\}$ in $L_2(\R^3,\C)$, where each basis function is the product of the complex spherical harmonics $y_{lm}(\varphi_1,\varphi_2)$ and the same set of radial functions $\tilde{j}_s(r)$. These two bases are naturally connected through the Fourier slice theorem. 

Assume that $f$ is $(L,S_0,\ldots,S_L)$-bandlimited with respect to the basis $\{\tilde{h}_{lsm}\}$, leading to the representation:
\begin{align}\label{eq:expansion_cryoEM}
f = \sum_{(l,s,m)\in\mathcal{I}} u_{m}^{(ls)} \tilde{h}_{lsm} = \sum_{(l,s,m)\in\mathcal{I}} \theta_{m}^{(ls)} \tilde{\mathbcal{h}}_{lsm},
\end{align}
where
\begin{align*}
u=(u_m^{(ls)}:(l,s,m)\in\mathcal{I})\in\C^d,~~~~~~\theta=(\theta_m^{(ls)}:(l,s,m)\in\mathcal{I})\in\R^d,
\end{align*}
with the index set $\mathcal{I}$ defined in \eqref{eq:index_set}. Here, $\{\tilde{h}_{lsm}\}$ are the inverse Fourier transforms of $\{\hat{\tilde{h}}_{lsm}\}$, and $\{\tilde{\mathbcal{h}}_{lsm}\}$ are the corresponding real basis functions obtained by the same basis transformation as in \eqref{eq:complex_real_transformation}. Applying the Fourier slice theorem, its tomographic projection yields a bandlimited representation:
\begin{align*}
\Pi \circ f = \sum_{(s,m)\in\tilde{\mathcal{I}}} \tilde{u}_m^{(s)} q_{sm} = \sum_{(s,m)\in\tilde{\mathcal{I}}} \tilde{\theta}_m^{(s)} \mathbcal{q}_{sm},
\end{align*}
where
\begin{align*}
\tilde{u}=(\tilde{u}_m^{(s)}:(s,m)\in\mathcal{I})\in\C^{\td},~~~~~~\tilde{\theta}=(\tilde\theta_m^{(s)}:(s,m)\in\mathcal{I})\in\R^{\td}.
\end{align*}
Here, $\{q_{sm}\}$ are the inverse Fourier transforms of $\{\hat{q}_{sm}\}$, and $\{\mathbcal{q}_{sm}\}$ are the corresponding real basis functions on $\R^2$. The index set $\tilde{\mathcal{I}}$ is given by
\begin{align}\label{eq:index_set_projection}
\tilde{\mathcal{I}} = \{(s,m): 1\leq s\leq S,~-L\leq m\leq L\},~~~~~~\tilde{d}=|\tilde{\mathcal{I}}|= S(2L+1) 
\end{align}
for $S=\max(S_0,\ldots,S_L)$. We can then translate $\Pi$ into a linear map $P$, which maps $\theta\in\R^d$ to $\tilde\theta\in\R^{\td}$, whose explicit form is given in \eqref{eq:tomographic_proj_representation}. Thanks to the completeness and orthonormality of $\{\mathbcal{q}_{sm}\}$, the model \eqref{eq:model_cryoEM} is equivalent to observing the coefficients of $\Pi\circ f$ in this basis with i.i.d. $\mathrm{N}(0,\sigma^2)$ noise. Further details can be found in Appendix \ref{sec:function_basis_cryoEM}.

The rotation $\frakg$ acting on $f$ follows the same form as in the unprojected cryo-EM scenario. Therefore, the observation model for the real coefficients of $f$ under the complete orthonormal basis $\{\mathbcal{q}_{sm}\}$ represents a special case of \eqref{eq:projectedmodel},
where $P$ is given in \eqref{eq:tomographic_proj_representation}, extending the unprojected setup. The likelihood function can be written as 
\begin{align*}
    p_{\theta}(y) = 
\int_{\mathsf{SO}(3)} \frac{1}{(2\pi\sigma^2)^{\td/2}}
\exp\left(-\frac{\|y- P\check{\mathcal{D}}(\frakg)\theta\|^2}{2\sigma^2}\right)\der
\lambda_*(\frakg).
\end{align*}

\section{Proofs}\label{sec:proofs}
This appendix contains all proofs except that of Theorem \ref{thm:withoutProj}.

\subsection{Proofs of Theorems \ref{lem:GlobalMin_CorrectLikelihood} and \ref{thm:bias_existence}}
\begin{proof}[Proof of Theorem \ref{lem:GlobalMin_CorrectLikelihood}]
Note that 
\begin{align*}
\ell(\theta;\Lambda_*,\sigma^2)=& -\E_{\theta_*, \Lambda_*} \log p_\theta(y; \Lambda_*,\sigma^2)\\
=& -\E_{\theta_*, \Lambda_*} \log p_{\theta_*}(y; \Lambda_*,\sigma^2) + \E_{\theta_*, \Lambda_*} \log \frac{p_{\theta_*}(y; \Lambda_*,\sigma^2)}{p_{\theta}(y; \Lambda_*,\sigma^2)},
\end{align*}
where the second term is the Kullback-Leibler divergence between densities $p_{\theta_*}(y; \Lambda_*,\sigma^2)$ and $p_{\theta}(y; \Lambda_*,\sigma^2)$ and is always non-negative. Thus, a point $\theta\in\R^d$ is a global minimizer of $\ell(\theta;\Lambda_*,\sigma^2)$ if and only if $p_{\theta_*}(y; \Lambda_*,\sigma^2)=p_{\theta}(y; \Lambda_*,\sigma^2)$. By \cite[Theorem 2(2)]{nguyen2013convergence} and noting that $Pg\theta$ and $Pg\theta_*$ always lie in a compact subset of $\R^{\td}$, we conclude the claim.
\end{proof}

\begin{proof}[Proof of Theorem \ref{thm:bias_existence}]
Let $\tilde\lambda$ be any measure on $U$ such that $\eta$ is measurable and denote its corresponding pushforward measure under $\eta$ to be $\Lambda$ on $\G$. By direct calculations, the gradient of the general misspecified negative population likelihood $\ell(\theta;\Lambda_0,\Lambda,\sigma^2)$ defined in \eqref{eq:population_likelihood_general} with respect to $\theta$ is given by
\begin{align*}
&\nabla_{\theta}\ell(\theta;\Lambda_0,\Lambda,\sigma^2)\\
=&\frac{1}{\sigma^2}\cdot\E_{\theta_*,\Lambda} \frac{\int_{\G}
\exp\Big(-\frac{\|y- P g\theta\|^2}{2\sigma^2}\Big) \cdot g^\T P^\T \Big(P g\theta-y\Big)\der
\Lambda_0(g)}{\int_{\G}
\exp\Big(-\frac{\|y- P g\theta\|^2}{2\sigma^2}\Big)\der
\Lambda_0(g)}\\
=&\frac{1}{\sigma^2}\cdot \E_{\tilde{g}\sim \Lambda} \E_{y\sim \mathrm{N}(P\tilde{g}\theta_*,\sigma^2\Id_{\td})} \frac{\int_{\G}
\exp\Big(-\frac{\|y- P g\theta\|^2}{2\sigma^2}\Big) \cdot g^\T P^\T \Big(P g\theta-y\Big)\der
\Lambda_0(g)}{\int_{\G}
\exp\Big(-\frac{\|y- P g\theta\|^2}{2\sigma^2}\Big)\der
\Lambda_0(g)} .
\end{align*}
Applying the change of variable formula, 
\begin{align*}
&\sigma^2\cdot \nabla_{\theta}\ell(\theta;\Lambda_0,\Lambda,\sigma^2)|_{\theta=\theta_*} \\
=& \E_{\tilde{u}\sim \tilde\lambda}  \E_{y\sim \mathrm{N}(P\eta(\tilde{u})\theta_*,\sigma^2\Id_{\td})}   
\frac{\int_{U}
\exp\Big(-\frac{\|y- P \eta(u)\theta_*\|^2}{2\sigma^2}\Big) \cdot \eta(u)^\T P^\T \Big(P \eta(u)\theta_*-y\Big)\der
\tilde\lambda_0(u)}{\int_{U}
\exp\Big(-\frac{\|y- P \eta(u)\theta_*\|^2}{2\sigma^2}\Big)\der
\tilde\lambda_0(u)} \\
=&\E_{\tilde{u}\sim \tilde\lambda} H(\tilde{u};\theta_*,\sigma^2).
\end{align*}
Note that $H(\tilde{u};0,\sigma^2)=0$ for any $\tilde{u}\in U$ and $\sigma^2$. Moreover, by assumption, $H$ is real analytic function in $(v,\theta,\tau^2)$ over the open and connected domain $\tilde{U}\times \R^d\times \R_+\subset \R^{m+d+1}$. Since there exists at least one point $(u_0,\theta_0,\tau_0^2)\in \tilde{U}\times \R^d\times \R_+$ such that $H(u_0;\theta_0,\tau_0^2)\not=0$ (by the previous argument, we know that $\theta_0\not=0$), the identity theorem for real analytic functions implies that the set of points where $H=0$ is a Lebesgue measure zero subset of this domain. In other words, for generic $\tilde{u}\in\tilde{U}$, $\theta_*\in\R^d\setminus\{0\}$, and $\sigma^2\in\R_+$, we have $H(\tilde{u};\theta_*,\sigma^2)\not=0$. Consequently, since $H$ is real analytic (and hence continuous) in $\tilde{u}$ over $\tilde{U}$ for fixed $\theta_*$ and $\sigma^2$, we can construct a measure $\tilde\lambda$ on $\tilde{U}$ (and extend it to $U$) such that $\E_{\tilde{u}\sim \tilde\lambda} H(\tilde{u};\theta_*,\sigma^2)\not=0$ for given generic $\theta_*\in\R^d\setminus\{0\}$ and $\sigma^2$. Finally, setting $\Lambda_*$ as the pushforward measure of $\tilde\lambda$ under $\eta$, we obtain $\nabla_{\theta}\ell(\theta;\Lambda_0,\sigma^2)|_{\theta=\theta_*}\not=0$. This establishes the desired result.
\end{proof}

\subsection{Proofs of Results in Section \ref{sec:quantification_bias}}
\begin{proof}[Proof of Theorem \ref{thm:consistency_in_measure}]
Write $\ell(\theta;\Lambda_0,\sigma^2)$ as $\ell(\theta;\sigma^2)$ for simplicity. Denote $\phi(y) = \frac{1}{(2\pi)^{\td/2}}\exp\left(-{\|y\|^2}/{2}\right)$ for any $ y\in\R^{\td}$. Then $\ell(\theta;\sigma^2)$ can be expressed as
\begin{align*}
    \ell(\theta;\sigma^2) &= -\E_{\theta_*,\Lambda_*}\log\int_{\G} \frac{1}{\sigma^{\td}}  \phi\left( \frac{y-Pg\theta}{\sigma}\right) \der \Lambda_0(g)\\
    &=-\int \left(\int_{\G} \frac{1}{\sigma^{\td}}  \phi\left( \frac{y-Pg\theta_*}{\sigma}\right) \der \Lambda_*(g)\right)  \log \left(\int_{\G} \frac{1}{\sigma^{\td}}  \phi\left( \frac{y-Pg\theta}{\sigma}\right) \der \Lambda_0(g)\right) \der y.
\end{align*}

First, we give an upper bound for $ \ell(\theta_*;\sigma^2)$:
\begin{align}\label{eq:upper
_bound_loglikelihood}
\ell(\theta_*;\sigma^2) &= \int \left(\int_{\G} \frac{1}{\sigma^{\td}}  \phi\left( \frac{y-Pg\theta_*}{\sigma}\right) \der \Lambda_*(g)\right)  \log \frac{\int_{\G} \frac{1}{\sigma^{\td}}  \phi\left( \frac{y-Pg\theta_*}{\sigma}\right) \der \Lambda_*(g)}{\int_{\G} \frac{1}{\sigma^{\td}}  \phi\left( \frac{y-Pg\theta_*}{\sigma}\right) \der \Lambda_0(g)} \der y \notag\\
&\quad - \int \left(\int_{\G} \frac{1}{\sigma^{\td}}  \phi\left( \frac{y-Pg\theta_*}{\sigma}\right) \der \Lambda_*(g)\right)  \log \frac{\int_{\G} \frac{1}{\sigma^{\td}}  \phi\left( \frac{y-Pg\theta_*}{\sigma}\right) \der \Lambda_*(g)}{\phi(y)} \der y \notag\\
&\quad - \int \left(\int_{\G} \frac{1}{\sigma^{\td}}  \phi\left( \frac{y-Pg\theta_*}{\sigma}\right) \der \Lambda_*(g)\right)  \log \phi(y) \der y \notag\\
&\leq \int \left(\int_{\G} \frac{1}{\sigma^{\td}}  \phi\left( \frac{y-Pg\theta_*}{\sigma}\right) \der \Lambda_*(g)\right) (\log C_2) \der y - 0 \notag\\
&\quad - \int \left(\int_{\G} \frac{1}{\sigma^{\td}}  \phi\left( \frac{y-Pg\theta_*}{\sigma}\right) \der \Lambda_*(g)\right)  \left(-\frac{\td}{2}\log(2\pi) -\frac{\|y\|^2}{2}\right) \der y \notag\\
&\leq \log C_2 -  \left(-\frac{\td}{2}\log(2\pi) -\left(\sup_{g\in\G}\|Pg\theta_*\|^2+\sigma^2\right)\right) \notag\\
&\leq \log C_2 +\frac{\td}{2}\log(2\pi)+C_P^2\cdot\|\theta_*\|^2+\sigma^2.
\end{align}
Here, the first inequality follows from the fact that $Q(g) \leq C_2$ for $\Lambda_0$-almost every $g$, with $C_2 \geq 1$, and from the non-negativity of the Kullback–Leibler divergence. The second inequality follows from the following argument: let $\mathrm{g} \sim \Lambda_*$ and $\varepsilon \sim  \mathrm{N}(0, \sigma^2 \Id_{\tilde{d}})$ be two independent random quantities. Then we have $\int \left( \int_{\G} \frac{1}{\sigma^{\tilde{d}}} \, \phi\left( \frac{y - P g \theta_*}{\sigma} \right) \, \mathrm{d}\Lambda_*(g) \right) \|y\|^2 \, \mathrm{d}y = \mathbb{E} \left\| P \mathrm{g} \theta_* + \varepsilon \right\|^2 \leq 2 \mathbb{E} \left( \|P \mathrm{g} \theta_*\|^2 + \|\varepsilon\|^2 \right) \leq 2 \left( \sup_{g \in \G} \|P g \theta_*\|^2 + \sigma^2 \right)$. The last inequality uses the bound $\|P g \theta_*\| \leq C_P\cdot\|\theta_*\|$ for all $g \in \G$.

Second, consider any $\theta\in\R^d$ such that $m(s; \theta_*, \theta) \geq \epsilon$ for some fixed $\epsilon>0$. Then we have 
\begin{align*}
\Lambda_* \big(\big\{g\in \G: d (Pg\theta_*, M_\theta(\G))> s \big\}\big)\geq C_1\epsilon.
\end{align*}
We give a lower bound for $\ell(\theta;\sigma)$ by decomposing $\R^{\td}$ into the set $A=\{y\in\R^{\td}: d(y,M_\theta(\G))\leq s/2\}$ and its complement. For any $y\notin A$, we have
\begin{align}\label{eq:lower
_bound_loglikelihood_0}
    \log \left(\int_{\G} \frac{1}{\sigma^{\td}}  \phi\left( \frac{y-Pg\theta}{\sigma}\right) \der \Lambda_0(g)\right) &\leq \log \left(\int_\G \frac{1}{\sigma^{\td}} \frac{1}{(2\pi)^{\td /2}} \exp\left(-\frac{s^2}{8\sigma^2}\right) \der \Lambda_0(g) \right) \notag \\
    &\leq  \log\left(\frac{1}{\sigma^{\td}} \frac{1}{(2\pi)^{\td /2}} \exp\left(-\frac{s^2}{8\sigma^2}\right)\right)\notag \\
    &= -\frac{\td}{2}\log(2\pi) - \td \log \sigma -\frac{s^2}{8\sigma^2},\\
    &\leq  -\frac{\td}{2}\log(2\pi)  -\frac{s^2}{16\sigma^2}\notag\\
    &<0,\notag
\end{align}
where the second to last inequality holds when $\sigma \leq {s^2}/({16\tilde d})$, as $\frac{s^2}{16\sigma^2}\geq \frac{\td}{\sigma} \geq \td \log \frac{1}{\sigma}$. Recall that $\mathrm{g}\sim \Lambda_*$ and $\varepsilon\sim \mathrm{N}(0,\sigma^2 \Id_{\td})$ are two independent random quantities. We have
\begin{align*}
\int_{A^c} \left(\int_{\G} \frac{1}{\sigma^{\td}}  \phi\left( \frac{y-Pg\theta_*}{\sigma}\right) \der \Lambda_*(g)\right)  \der y &= \P \left(P\mathrm{g}\theta_*+\varepsilon\notin A \right)\\
&\geq \P\left(d(P\mathrm{g}\theta_*, M_{\theta}(\G))>s\right) \P\left(\|\varepsilon\|\leq s/2\right)\\
&\geq C_1\epsilon\cdot \left(1-\exp\left(-\frac{s^2}{10\sigma^2}\right)\right),
\end{align*}
where the last inequality holds when $\sigma^2\leq {s^2}/({2\td})$ by Lemma \ref{lem:chi_square_tail}. Hence, when $\sigma^2\leq \min\{( {s^2}/({16\tilde d}))^2,{s^2}/({2\td}), s^2/10\}$, we have
\begin{align*}
&-\int_{A^c} \left(\int_{\G} \frac{1}{\sigma^{\td}}  \phi\left( \frac{y-Pg\theta_*}{\sigma}\right) \der \Lambda_*(g)\right)  \log \left(\int_{\G} \frac{1}{\sigma^{\td}}  \phi\left( \frac{y-Pg\theta}{\sigma}\right) \der \Lambda_0(g)\right) \der y\\
&\geq  -\int_{A^c} \left(\int_{\G} \frac{1}{\sigma^{\td}}  \phi\left( \frac{y-Pg\theta_*}{\sigma}\right) \der \Lambda_*(g)\right) \left(-\frac{\td}{2}\log(2\pi) -\frac{s^2}{16\sigma^2}\right) \der y\\
&  = \left(\frac{\td}{2}\log(2\pi) +\frac{s^2}{16\sigma^2}\right)   \int_{A^c} \left(\int_{\G} \frac{1}{\sigma^{\td}}  \phi\left( \frac{y-Pg\theta_*}{\sigma}\right) \der \Lambda_*(g)\right) \der y\\
&\geq C_1\epsilon\cdot \left(1-\exp\left(-\frac{s^2}{10\sigma^2}\right)\right)\cdot \left(\frac{\td}{2}\log(2\pi) +\frac{s^2}{16\sigma^2}\right) \\
&\geq \frac{C_1\epsilon}{2}\cdot \left(\frac{\td}{2}\log(2\pi) +\frac{s^2}{16\sigma^2}\right).
\end{align*}
In addition, since $\sup_{y}\phi(y)\leq \frac{1}{(2\pi)^{\td/2}}$, we have
\begin{align*}
&-\int_{A} \left(\int_{\G} \frac{1}{\sigma^{\td}}  \phi\left( \frac{y-Pg\theta_*}{\sigma}\right) \der \Lambda_*(g)\right)  \log \left(\int_{\G} \frac{1}{\sigma^{\td}}  \phi\left( \frac{y-Pg\theta}{\sigma}\right) \der \Lambda_0(g)\right) \der y\\    
&\geq -\int_{A} \left(\int_{\G} \frac{1}{\sigma^{\td}}  \phi\left( \frac{y-Pg\theta_*}{\sigma}\right) \der \Lambda_*(g)\right)   \log \left( \frac{1}{(2\pi)^{\td/2}}\frac{1}{\sigma^{\td}} \right) \der y\\
& =  -  \frac{\td}{2}  \log \left(\frac{1}{2\pi \sigma^2}\right)\int_{A} \left(\int_{\G} \frac{1}{\sigma^{\td}}  \phi\left( \frac{y-Pg\theta_*}{\sigma}\right) \der \Lambda_*(g)\right)   \der y\\
&\geq  -  \frac{\td}{2}  \log \left(\frac{1}{2\pi \sigma^2}\right),
\end{align*}
where the last inequality holds when $\sigma^2\leq 1/(2\pi)$. Combining all, when 
\begin{align*}
\sigma^2\leq \min\{( {s^2}/({16\tilde d}))^2,{s^2}/({2\td}),s^2/10,1/(2\pi)\},
\end{align*}
we have
\begin{align}\label{eq:lower
_bound_loglikelihood}
&\ell(\theta;\sigma^2)\notag\\ 
&= -\int_{A^c} \left(\int_{\G} \frac{1}{\sigma^{\td}}  \phi\left( \frac{y-Pg\theta_*}{\sigma}\right) \der \Lambda_*(g)\right)  \log \left(\int_{\G} \frac{1}{\sigma^{\td}}  \phi\left( \frac{y-Pg\theta}{\sigma}\right) \der \Lambda_0(g)\right) \der y \notag\\
&\quad -\int_{A} \left(\int_{\G} \frac{1}{\sigma^{\td}}  \phi\left( \frac{y-Pg\theta_*}{\sigma}\right) \der \Lambda_*(g)\right)  \log \left(\int_{\G} \frac{1}{\sigma^{\td}}  \phi\left( \frac{y-Pg\theta}{\sigma}\right) \der \Lambda_0(g)\right) \der y \notag\\
&\geq \frac{C_1\epsilon}{2} \cdot\left(\frac{\td}{2}\log(2\pi) +\frac{s^2}{16\sigma^2}\right) - \frac{\td}{2}  \log \left(\frac{1}{2\pi \sigma^2}\right).
\end{align}
Hence, by settting $\epsilon=\sigma^{2-\alpha}$, there exists some threshold $\sigma_0^2$ that only depends on $s,\td,\alpha$ such that for any $\sigma^2$ less than it, the above display is greater than the upper bound \eqref{eq:upper
_bound_loglikelihood} derived above for $\ell(\theta_*;\sigma^2)$. Consequently, under such conditions, the MLE $\hat{\theta}(\Lambda_0, \sigma^2)$, being the minimizer of $\ell(\theta;\sigma^2)$ by definition, must satisfy $m\big(s; \theta_*, \hat{\theta}(\Lambda_0, \sigma^2)\big) < \sigma^{2-\alpha}$. The proof is complete.
\end{proof}

\begin{proof}[Proof of Theorem \ref{thm:consistency_Hausdorff}]
Consider any $\theta\in\R^d$ such that $\sup_{g\in\G} d(Pg\theta_*, M_\theta(\G)) > \td^{1/2} \sigma^{1-\alpha}$ with $ \td^{1/2} \sigma^{1-\alpha}<1$. By compactness of $M_\theta(\G)$, there exists $g_0\in\G$ such that $d(Pg_0\theta_*, M_\theta(\G)) > \td^{1/2} \sigma^{1-\alpha}$. We then have
\begin{align*}
m(\td^{1/2} \sigma^{1-\alpha}/2;\theta_*,\theta)\geq& \Lambda_0\Big(\big\{g\in\G: d(Pg\theta_*, Pg_0\theta_*)\leq \td^{1/2} \sigma^{1-\alpha}/2\big\}\Big)\\
\geq & \Lambda_0\Big(\big\{g\in\G: \|g-g_0\|\cdot C_P\|\theta_*\|\leq \td^{1/2} \sigma^{1-\alpha}/2\big\}\Big)\\
\geq & c_0 \Big(\frac{\td^{1/2}}{2C_P\|\theta_*\|}\cdot \sigma^{1-\alpha}\Big)^\gamma.
\end{align*}
Here, the first inequality follows from the triangle inequality and the last inequality follows from the assumption \eqref{eq:measure_G}. 

As in the proof of Theorem~\ref{thm:consistency_in_measure}, setting $s^2 = \tilde{d} \sigma^{2 - 2\alpha}$ with $\alpha\in(0,1)$, the bound in \eqref{eq:lower
_bound_loglikelihood_0} becomes
\begin{align*}
-\frac{\tilde{d}}{2}\log(2\pi) - \tilde{d} \log \sigma -\frac{s^2}{8\sigma^2}
&= -\frac{\tilde{d}}{2}\log(2\pi) - \tilde{d} \log \sigma - \frac{\tilde{d}}{8} \sigma^{-2\alpha} \\
&\leq -\frac{\tilde{d}}{2}\log(2\pi) - \frac{\tilde{d}}{16} \sigma^{-2\alpha} \\
&< 0,
\end{align*}
provided that $16\log (1/\sigma)\leq \sigma^{-2\alpha}$. Furthermore, if $\sigma$ also satisfies
\begin{align*}
\sigma^2\leq \min\{\sigma^{2-2\alpha}/2, \td\sigma^{2-2\alpha}/10,1/(2\pi)\},
\end{align*}
then the bound in \eqref{eq:lower
_bound_loglikelihood} is modified to
\begin{align*}
\ell(\theta; \sigma^2)
\geq \left( \frac{\tilde{d}}{2}\log(2\pi) + \frac{\tilde{d}}{16} \sigma^{-2\alpha} \right) \cdot \frac{C_1 c_0}{2} \left( \frac{\tilde{d}^{1/2}}{2C_P \|\theta_*\|} \cdot \sigma^{1 - \alpha} \right)^\gamma
- \frac{\tilde{d}}{2} \log \left( \frac{1}{2\pi \sigma^2} \right).
\end{align*}
Comparing with \eqref{eq:upper
_bound_loglikelihood}, we observe that as long as $ \gamma(1 - \alpha) - 2\alpha < 0$, i.e. $\alpha>\gamma/(\gamma+2)$, there exists a threshold $\sigma_0 > 0$, depending only on $\alpha, \gamma, \tilde{d}$, such that for all $\sigma < \sigma_0 $, we have $\ell(\theta;\sigma^2)>\ell(\theta_*,\sigma^2)$, which implies the desired claim.

\end{proof}

\subsection{Proofs of Results in Section \ref{sec:bias_sigma_0}}
\begin{proof}[Proof of Proposition \ref{prop:non-self_intersection}]
It suffices to show that $M_\theta$ is injective for generic $\theta$. Given this, since $M_{\theta}$ is continuous, $\G$ is compact, and $\R^{\td}$ is Hausdorff, we can apply \cite[Thm 26.6]{munkrestopology} to conclude $T_\theta$ is a homeomorphism for generic $\theta$. 

To this end, it suffices to prove that the set of triples $(g_1, g_2,\theta)\in\G\times \G\times \R^d$ such that $g_1\not= g_2$ but $Pg_1\theta= Pg_2\theta$ has dimension strictly less than $d$. Indeed, the set of $\theta$ for which $M_\theta$ fails to be injective is the projection of this set of triples to $\R^d$, and projection does not raise the dimension of a real semi-algebraic set.

Let us now fix $g_1,g_2$ with $g_1\not=g_2$ and compute the dimension of $\{\theta\in\R^d~|~Pg_1\theta=Pg_2\theta\}$. Notice that
\begin{align*}
Pg_1\theta= Pg_2\theta
\end{align*}
if and only if
\begin{align*}
\theta\in \mathrm{ker} (P(g_1-g_2)).
\end{align*}
So the dimension of $\{\theta\in\R^d~|~Pg_1\theta=Pg_2\theta\}$ is 
\begin{align*}
\dim \mathrm{ker} (P(g_1-g_2))=d-\rank (P(g_1-g_2))= d-\rank (P(\Id-g_2g_1^{-1})).
\end{align*}
Since $P$ has rank $\td\leq d$, we further have
\begin{align}
\dim \mathrm{ker} (P(g_1-g_2)) \leq& d-\td+\dim \mathrm{ker} (\Id-g_2g_1^{-1})\notag\\
=& d-\td +\dim E_1(g_2g_1^{-1}).\label{eq:dim_upper}
\end{align}

Now recall that for any polynomial map $f:X\to Y$ between real semi-algebraic sets $X$ and $Y$, we have
\begin{align*}
    \dim(X)=\max_{m\geq 0}\big\{m+\dim\{y\in Y ~|~ \dim f^{-1} (y)\geq m\} \big\}.
\end{align*}
Applying this to the map from $\{(g_1,g_2,\theta)\in\G\times\G\times\R^d ~|~g_1\not=g_2,Pg_1\theta=Pg_2\theta\}$ to $\{(g_1,g_2)\in\G~|~g_1\not=g_2\}$, we see that
\begin{align*}
&\dim \{(g_1,g_2,\theta)\in\G\times\G\times\R^d ~|~g_1\not=g_2,Pg_1\theta=Pg_2\theta\} \\
=&\max_{m\geq 0}\big\{m+\dim\{(g_1,g_2)\in\G\times \G ~\text{with}~g_1\not=g_2 ~|~ \dim \{\theta\in\R^d~|~Pg_1\theta=Pg_2\theta\}\geq m\} \big\}\\
\leq& \max_{m\geq 0}\big\{m+\dim\{(g_1,g_2)\in\G\times \G ~\text{with}~g_1\not=g_2 ~|~ \dim E_1(g_2g_1^{-1})\geq m-d+\td\} \big\},
\end{align*}
where the last inequality follows from \eqref{eq:dim_upper}. Note that $g_1\not=g_2$ if and only if $g_2g_1^{-1}\not=\Id$ and each $g$ arises as $g_2g_1^{-1}$ for a $\dim \G$-dimensional set of pairs $(g_1,g_2)$. We have the dimension of the set of pairs $(g_1,g_2)$ with $g_1\not=g_2$ where the dimension of the set $E_1(g_2g_1^{-1})$ is larger than $m$ is equal to $\dim \G+ \dim\{g\in\G\setminus \Id ~|~ \dim E_1(g)\geq m\}$. It follows that
\begin{align*}
&\dim \{(g_1,g_2,\theta)\in\G\times\G\times\R^d ~|~g_1\not=g_2,Pg_1\theta=Pg_2\theta\}\\
\leq&\max_{m\geq 0}\big\{m+\dim\G+ \dim\{g\in\G\setminus \Id ~|~ \dim E_1(g)\geq m-d+\td\} \big\}\\
=&\max_{k\geq 0} \big\{k+d-\td+\dim\G+ \dim\{g\in\G\setminus \Id ~|~ \dim E_1(g)\geq k\} \big\}\\
=&d+\max_{k\geq 0} \big\{k+ \dim\{g\in\G\setminus \Id ~|~ \dim E_1(g)\geq k\} \big\}+\dim\G-\td.
\end{align*}
By assumption on $\td$, this is strictly less than $d$, as desired.
\end{proof}

\begin{proof}[Proof of Corollary \ref{cor:no_self_intersection_examples}]
In the case of projected continuous MRA, it is straightforward to verify that $P$ is surjective. Moreover, we have $\dim\G=1$, and the group element $g$ is given by \eqref{eq:MRA_rotation}. Thus, for any non-identity element $g\in\G$, $\dim E_1(g)=1$, which implies 
\begin{align*}
\max_{k \geq 0} \big\{ k + \dim \{ g \in \G \setminus \Id \mid \dim E_1(g) \geq k \} \big\} = 2.
\end{align*}
To apply Proposition \ref{prop:non-self_intersection}, the condition reduces to \( \td = L+1 > 3 \), that is, \( L \geq 3 \).

For cryo-EM, recall that the tomographic projection corresponds to the linear map \(P\) defined in \eqref{eq:tomographic_proj_representation}. To prove that \(P\) is surjective, it suffices to check that the corresponding map \(P^{\C}\) between complex coefficients, given by \eqref{eq:tomographic_proj_representation_complex}, is surjective. By \eqref{eq:coef_projection_relation}, it is enough to show that for any \((s,m) \in \tilde{\mathcal{I}}\), there exists some non-zero \(c_{lm}\) with \(l = |m|, \ldots, L\). Noting that \( P_{mm}(0) = (2m-1)!! \neq 0 \) for all \( m \geq 0 \), and that \( c_{mm} \) is a non-zero constant multiple of \( P_{mm}(0) \), we conclude that \( c_{mm} \) is non-zero, and similarly for \( c_{m,-m} \). This guarantees the surjectivity of the map \(P\).

Moreover, note that $\dim\G=3$ and the group element $g$ is given by \eqref{eq:cryoEM_rotation}. Thus, for any non-identity element $g\in\G$, $\dim E_1(g)=S_0$, which implies
\begin{align*}
\max_{k \geq 0} \big\{ k + \dim \{ g \in \G \setminus \Id \mid \dim E_1(g) \geq k \} \big\} = 2S_0.
\end{align*}
To apply Proposition \ref{prop:non-self_intersection}, the condition becomes $\td=S(2L+1)>2S=2S_0$, which holds if $L\geq 1$. 

\end{proof}

\begin{proof}[Proof of Corollary \ref{cor:orbit_subset}]
By Proposition~\ref{prop:non-self_intersection}, both $M_\theta(\G)$ and $M_{\theta'}(\G)$ are smooth manifolds of the same dimension. Since $\G$ is connected, both orbits are connected. The inclusion $M_\theta(\G) \subseteq M_{\theta'}(\G)$ then implies equality.
\end{proof}

\begin{proof}[Proof of Theorem \ref{thm:orbit_recovery}]
To prove \eqref{eq:orbit_zero_noise}, it suffices by Corollary~\ref{cor:orbit_subset} to show that $M_{\theta_*}(\G) \subseteq M_{\bar\theta}(\G)$. By continuity of the map $M_\theta$, it is enough to verify that for any $s > 0$,
\begin{align}\label{eq:measure_zero_noise}
\Lambda_0\left( \left\{ g \in \G : d(Pg\theta_*, M_{\bar\theta}(\G)) > s \right\} \right) = 0,
\end{align}
where $\Lambda_0$ denotes the Haar measure on $\G$. To establish this, note that for any $g \in \G$ and $\theta', \theta'' \in \mathbb{R}^d$, we have
\begin{align*}
\|Pg\theta' - Pg\theta''\| \leq C_P \cdot \|\theta' - \theta''\|.
\end{align*}
In particular, it follows that for any $s > 0$, there exists $\delta > 0$ such that for all $\sigma^2 \leq \delta$,
\begin{align}\label{eq:Hausdorff_consistency_zero_noise}
d_H(M_{\hat\theta(\Lambda_0,\sigma^2)}(\G), M_{\bar\theta}(\G)) \leq s/2.
\end{align}
Therefore, for all such $\sigma^2$,
\begin{align*}
\big\{ g \in \G : d(Pg\theta_*, M_{\hat\theta(\Lambda_0,\sigma^2)}(\G)) \leq s/2 \big\}
\subseteq
\big\{ g \in \G : d(Pg\theta_*, M_{\bar\theta}(\G)) \leq s \big\},
\end{align*}
which further gives
\begin{align*}
\Lambda_0\Big( \big\{ g \in \G : d(Pg\theta_*, M_{\bar\theta}(\G)) > s \big\} \Big)
\leq
\Lambda_0\Big( \big\{ g \in \G : d(Pg\theta_*, M_{\hat\theta(\Lambda_0,\sigma^2)}(\G)) > s/2 \big\} \Big).
\end{align*}
Since Theorem \ref{thm:consistency_in_measure} guarantees the right-hand side vanishes as $\sigma^2\to 0$ for any fixed $s$, the right-hand side must be zero, thereby proving \eqref{eq:measure_zero_noise}. The Hausdorff consistency result \eqref{eq:Hausdorff_consistency} then follows from \eqref{eq:Hausdorff_consistency_zero_noise} by substituting $M_{\bar\theta}(\G)$ with $M_{\theta_*}$.

\end{proof}

\subsection{Proofs of Results in Section \ref{sec:joint_MLE}}

\begin{proof}[Proof of Lemma \ref{lem:inplane_uniform_density}]
By \eqref{eq:realWigner}, for any $p\geq 1$, $-p\leq u,v<0$, there exists some matrix $\mathcal{Q}_{puv}\in\C^{4\times 4}$ such that
\begin{align*}
(\mathcal{D}_{uv}^{(p)}(\frakg), \mathcal{D}_{u,-v}^{(p)}(\frakg), \mathcal{D}_{-u,v}^{(p)}(\frakg), \mathcal{D}_{-u,-v}^{(p)}(\frakg))^\T = \mathcal{Q}_{puv}\cdot(D_{uv}^{(p)}(\frakg), D_{u,-v}^{(p)}(\frakg), D_{-u,v}^{(p)}(\frakg), D_{-u,-v}^{(p)}(\frakg))^\T,
\end{align*}
for any $\frakg\in\mathsf{SO}(3)$. Here $D_{uv}^{(p)}$ is the $uv$-th entry of the complex Wigner D-matrix at frequency $p$.
The explicit form of $\mathcal{Q}_{puv}$ can be given by
\begin{align*}
\mathcal{Q}_{puv}=\frac{1}{2}\cdot\left(
    \begin{array}{cccc}
        1 & -(-1)^u & -(-1)^v & (-1)^{u+v}\\
        -\i & -(-1)^u \i & (-1)^v \i & (-1)^{u+v} \i\\
        \i & -(-1)^u \i & (-1)^v \i & -(-1)^{u+v} \i\\
        1 & (-1)^u  & (-1)^v & (-1)^{u+v}
    \end{array}\right),
\end{align*}
with its inverse
\begin{align*}
(\mathcal{Q}_{puv})^{-1}=\frac{1}{2}\cdot\left(
    \begin{array}{cccc}
        1 & \i & -\i & 1\\
        -(-1)^u & (-1)^u \i & (-1)^u \i & (-1)^u \\
        -(-1)^v & -(-1)^v \i & -(-1)^v \i & (-1)^v \\
        (-1)^{u+v} & -(-1)^{u+v} \i & (-1)^{u+v} \i & (-1)^{u+v}
    \end{array}\right).
\end{align*}
Note that $\mathcal{Q}_{puv}$ is unitary. Then for any $\frakg\in\mathsf{SO}(3)$ and rotation $z(\alpha)$ around the z-axis,
\begin{align*}
&(\mathcal{D}_{uv}^{(p)}, \mathcal{D}_{u,-v}^{(p)}, \mathcal{D}_{-u,v}^{(p)}, \mathcal{D}_{-u,-v}^{(p)})^\T(\frakg z(\alpha))\\
&=\mathcal{Q}_{puv}\cdot(D_{uv}^{(p)}, D_{u,-v}^{(p)}, D_{-u,v}^{(p)}, D_{-u,-v}^{(p)})^\T(\frakg z(\alpha))\\
&=\mathcal{Q}_{puv}\cdot\diag(e^{-i v\alpha}, e^{i v\alpha}, e^{-i v\alpha}, e^{i v\alpha})\cdot(D_{uv}^{(p)}, D_{u,-v}^{(p)}, D_{-u,v}^{(p)}, D_{-u,-v}^{(p)})^\T(\frakg)\\
&=\mathcal{Q}_{puv}\cdot\diag(e^{-i v\alpha}, e^{i v\alpha}, e^{-i v\alpha}, e^{i v\alpha})\cdot (\mathcal{Q}_{puv})^{-1}\cdot (\mathcal{D}_{uv}^{(p)}, \mathcal{D}_{u,-v}^{(p)}, \mathcal{D}_{-u,v}^{(p)}, \mathcal{D}_{-u,-v}^{(p)})^\T(\frakg).
\end{align*}
Hence, the in-plane uniformity of $\rho(\frakg)$ gives
\begin{align}\label{eq:constraint_inplane_B}
&(\mathcal{B}_{puv}, \mathcal{B}_{pu,-v}, \mathcal{B}_{p,-uv}, \mathcal{B}_{p,-u,-v})\\
&=(\mathcal{B}_{puv}, \mathcal{B}_{pu,-v}, \mathcal{B}_{p,-uv}, \mathcal{B}_{p,-u,-v})\cdot\mathcal{Q}_{puv}\cdot\diag(e^{-i v\alpha}, e^{i v\alpha}, e^{-i v\alpha}, e^{i v\alpha})\cdot (\mathcal{Q}_{puv})^{-1} \notag
\end{align}
for any $\alpha\in\R$. By direct calculations, 
\begin{align}\label{eq:inplane_calc}
\mathcal{Q}_{puv}\!\cdot\!\diag(e^{-i v\alpha}, e^{i v\alpha}, e^{-i v\alpha}, e^{i v\alpha})\!\cdot\! (\mathcal{Q}_{puv})^{-1}=\left(
    \begin{array}{cccc}
        \cos(v\alpha) & \sin(v\alpha) & 0 & 0\\
        -\sin(v\alpha) & \cos(v\alpha) & 0 & 0 \\
        0 & 0 & \cos(v\alpha) & \sin(v\alpha) \\
        0 & 0 & -\sin(v\alpha) & \cos(v\alpha)
    \end{array}\right).
\end{align}
We can perform similar calculations for the case when $u=0$ and $-p\leq v<0$. In conclusion, combining \eqref{eq:constraint_inplane_B} and \eqref{eq:inplane_calc} implies $\mathcal{B}_{puv}=0$ for any $v\not=0$ and $-p\leq u,v\leq p$, which further gives the expansion \eqref{eq:rot_dist_inplane_unif_expansion}.
\end{proof}

\begin{proof}[Proof of Theorem \ref{thm:consistency_1}]
We use Wald's consistency proof (see \cite[Theorem 5.14]{van2000asymptotic}). First, note that $B$ is compact. This is because for any $\frakg \in \mathsf{SO}(3)$, the constraint $\sum_{p=0}^{\bar{P}} \sum_{u=-p}^p \mathcal{B}_{pu} \mathcal{D}_{u0}^{(p)} (\frakg)\geq 0 $ is linear. Hence, the set $\{(\mathcal{B}_{pu}): \sum_{p=0}^{\bar{P}} \sum_{u=-p}^p \mathcal{B}_{pu} \mathcal{D}_{u0}^{(p)} (\frakg)\geq 0 \}$ is compact. Then
\begin{align*}
    B = \left\{(\mathcal{B}_{pu}): \mathcal{B}_{00}=1, \mathcal{B}_{pu}\in[-b,b]\text{ for all }p,u\right\}\cap \left(\cap_{\frakg\in\mathsf{SO}(3)} \left\{(\mathcal{B}_{pu}): \sum_{p=0}^{\bar{P}} \sum_{u=-p}^p \mathcal{B}_{pu} \mathcal{D}_{u0}^{(p)} (\frakg)\geq 0 \right\}\right),
\end{align*}
which is an  intersection of infinite closed sets and consequently is also compact. As a result, the parameter space $\Theta \times B$ is compact, where $\Theta :=\{x\in\R^d:\|x\|\leq r\}$.

Denote $m_{\theta,\mathcal{B}}(y) = \log p_{\theta,\mathcal{B}}(y)$. Note that for any $y$, the map $(\theta,\mathcal{B})\mapsto m_{\theta,\mathcal{B}}(y)$ is continuous. Hence, it is also upper-semicontinuous. Now we need to show $\E_{\theta_*,\mathcal{B}_*} \sup_{(\theta,\mathcal{B})\in\Theta\times B}m_{\theta,\mathcal{B}}(Y)<\infty$. Note that for any $y\in\R^{\td}$ and any $(\theta,\mathcal{B})\in\Theta\times B$, we have
\begin{align*}
m_{\theta,\mathcal{B}}(y) &= \log \left( \int_{\mathsf{SO}(3)}\frac{1}{(2\pi\sigma^2)^{\td/2}}
\exp\left(-\frac{\|y-P \check{\mathcal{D}}(\frakg)\theta\|^2}{2\sigma^2}\right)\sum_{p=0}^{ \bar{P} } \sum_{u=-p}^p \mathcal{B}_{pu} \mathcal{D}_{u0}^{(p)}(\frakg)\mathrm{d}\frakg\right)\\
&\leq \log \left( \int_{\mathsf{SO}(3)}\frac{1}{(2\pi\sigma^2)^{\td/2}}
\sum_{p=0}^{ \bar{P} } \sum_{u=-p}^p \mathcal{B}_{pu} \mathcal{D}_{u0}^{(p)}(\frakg)\mathrm{d}\frakg\right)\\
&\leq \log\left(\frac{1}{(2\pi\sigma^2)^{\td/2}}\right).
\end{align*}
Hence,
\begin{align*}
\E_{\theta_*,\mathcal{B}_*} \sup_{(\theta,\mathcal{B})\in\Theta\times B}m_{\theta,\mathcal{B}}(Y) \leq \E_{\theta_*,\mathcal{B}_*} \sup_{(\theta,\mathcal{B})\in\Theta\times B} \log\left(\frac{1}{(2\pi\sigma^2)^{\td/2}}\right)  = \log\left(\frac{1}{(2\pi\sigma^2)^{\td/2}}\right) <\infty.
\end{align*}
Since the model is assumed to be identifiable up to the joint orbit,  minimizers of the population negative log-likelihood $ \E_{\theta_*,\Lambda_*} - \log p_{\theta,\mathcal{B}}(y)$ form the set  $\{(\theta,\mathcal{B}):\|\theta\|\leq r, \mathcal{B}\in B, \orbit_{(\theta,\Lambda)}=\orbit_{(\theta_*,\Lambda_*)}$  where the density of $\Lambda$ is $\sum_{p=0}^{\bar{P}} \sum_{u=-p}^p \mathcal{B}_{pu} \mathcal{D}_{u0}^{(p)} (\frakg),\forall \frakg \in \mathsf{SO}(3)\}$. Then by Theorem 5.14 of \cite{van2000asymptotic}, in probability,
we have $(\hat \theta_n,\hat{\mathcal{B}}_n)$ converges to some point in this set.
\end{proof}

\section{Calculus of Spherical Harmonics}\label{sec:calculus_SH}

We establish the notations for certain special functions associated with the action of $\mathsf{SO}(3)$ and provide key identities among them, which will be used in the proofs and may also hold independent interest for other research endeavors. 

\subsection{Complex Spherical Harmonics and (Complex) Wigner D-Matrices}
We follow the conventions established in \cite{rose1995elementary}, noting that some variations exist in the literature.

\subsubsection{Complex Spherical Harmonics}\label{sec:complex_SH}
Let $P_{lm}(x)$ denote the associated Legendre polynomials (without
Cordon-Shortley phase)
\begin{equation}\label{eq:legendre}
P_{lm}(x)=\frac{1}{2^l l!}(1-x^2)^{m/2} \frac{\der^{l+m}}{\der x^{l+m}}(x^2-1)^l 
\quad \text{ for } m=-l,-l+1,\ldots,l-1,l.
\end{equation}
Let $\mathcal{S}^2 \subset \R^3$ be the unit sphere, parametrized by the 
latitude $\varphi_1 \in [0,\pi]$ and longitude $\varphi_2 \in [0,2\pi)$.
The \emph{complex spherical harmonics} basis on $\mathcal{S}^2$ is given by (see \cite[Eq.\
(III.20)]{rose1995elementary})
\begin{equation}\label{eq:complexharmonics}
y_{lm}(\varphi_1,\varphi_2)
=(-1)^m\sqrt{\frac{2l+1}{4\pi} \cdot \frac{(l-m)!}{(l+m)!}} \cdot
P_{lm}(\cos \varphi_1)e^{\i m\varphi_2}
\text{ for } l \geq 0 \text{ and } m=-l,\ldots,l.
\end{equation}
The index $l$ is the frequency, and there are $2l+1$ basis functions for each frequency $l$. These functions are orthonormal in $L_2(\mathcal{S}^2,\C)$ with respect
to the surface area measure $\sin \varphi_1\,\der \varphi_1\,\der \varphi_2$,
and satisfy the conjugation symmetry (see \cite[Eq.\
(III.23)]{rose1995elementary})
\begin{equation}\label{eq:conjugationsymmetry}
\overline{y_{lm}(\varphi_1,\varphi_2)}=(-1)^m y_{l,-m}(\varphi_1,\varphi_2).
\end{equation}
Notably, the complex spherical harmonics form a complete orthonormal basis for $L_2(\mathcal{S}^2,\C)$. The convention of complex spherical harmonics here has a $(-1)^m$ factor difference against those in \cite[Eq.\ 5.2(1)]{varshalovich1988quantum} and \cite[Eq.\ 4.36]{chirikjian2016harmonic}.

\subsubsection{Complex Wigner D-Matrices}\label{sec:complex_Wigner}
Let $f \in L_2(\mathcal{S}^2,\C)$.
Then $f$ may be decomposed in the complex spherical harmonics basis
(\ref{eq:complexharmonics}) as
\[f=\sum_{l=0}^\infty \sum_{m=-l}^l u_m^{(l)}y_{lm}.\]
Writing $u^{(l)}=(u_m^{(l)}:-l \leq m \leq l)\in\C^{2l+1}$,
the rotation $f \mapsto
f_\frakg$ given by $f_\frakg(\varphi_1,\varphi_2)=f(\frakg^{-1} \cdot (\varphi_1,\varphi_2))$
for $\frakg\in \mathsf{SO}(3)$ is described by the map of spherical harmonic coefficients
(see \cite[Eq.\ (4.28a)]{rose1995elementary})
\[u^{(l)}\mapsto D^{(l)}(\frakg) u^{(l)} \text{ for each } l=0,1,2,\ldots,\]
where $D^{(l)}(\frakg)\in\C^{(2l+1)\times (2l+1)}$ is the \emph{complex Wigner
D-matrix} at frequency $l$ corresponding to $\frakg$. Equivalently, writing $y_l=(y_{lm}:-l\leq m\leq l)\in\C^{2l+1}$, 
\begin{align}\label{eq:complexWigner}
y_l(\frakg^{-1} \cdot (\varphi_1,\varphi_2))= (D^{(l)}(\frakg))^\T y_l(\varphi_1,\varphi_2) \text{ for each } l=0,1,2,\ldots.
\end{align}
Notably, the complex Wigner D-matrices are unitary. They may be further expressed explicitly using Euler angles. For any $\frakg\in\mathsf{SO}(3)$ parameterized by Euler angles $(\alpha,\beta,\gamma)$ with $\alpha,\gamma\in[0,2\pi)$ and $\beta\in[0,\pi]$ (characterized by the z-y-z convention and the right-handed frame), the $mn$-th entry of the complex Wigner D-matrix at frequency $l$ has the form (see \cite[Eq.\ (4.12) \& (4.13)]{rose1995elementary})
\[
D^{(l)}_{mn}(\frakg(\alpha,\beta,\gamma))=e^{-im\alpha} d^{(l)}_{mn}(\cos\beta) e^{-in\gamma} \text{ for each }-l\leq m,n\leq l,
\]
where
\begin{align*}
d^{(l)}_{mn}(\cos\beta)=&[(l+m)!(l-m)!(l+n)!(l-n)!]^{1/2}\\
&\cdot\sum_k (-1)^k \frac{(\cos(\beta/2)
)^{2l-2k-m+n} (\sin(\beta/2))^{2k+m-n}}{k!(l-m-k)!(l+n-k)!(m-n+k)!},
\end{align*}
where $d^{(l)}(\cos\beta)\in\R^{(2l+1)\times(2l+1)}$ is the Wigner small d-matrix at frequency $l$ and the sum with respect to $k$ runs over all integer values for which the factorial arguments are non-negative. The convention here has a $(-1)^{m-n}$ factor difference against that in \cite[Eq.\ 4.3(5)]{varshalovich1988quantum}. 

The complex Wigner D-matrices forms an irreducible representation of the group $\mathsf{SO}(3)$. Moreover, by Peter-Weyl Theorem, the set of all the entries form a complete orthogonal basis for $L_2(\mathsf{SO}(3),\C)$.

\subsubsection{Properties}\label{sec:complex_properties}
Let the following expectations be taken with respect to the Haar probability measure on $\mathsf{SO}(3)$ (see \cite[Section 16]{rose1995elementary} and \cite[Appendix B.2]{bandeira2023estimation}). We have the following properties for complex Wigner D-matrices.
\begin{enumerate}
\item[(1)] Group homomorphism: for any $l\geq 0$ and $\frakg_1,\frakg_2\in \mathsf{SO}(3)$,
\[D^{(l)}(\frakg_1 \frakg_2)=D^{(l)}(\frakg_1)D^{(l)}( \frakg_2).\]
It follows immediately that $D^{(l)}(\frakg)$ is unitary for any $l\geq 0$ and $\frakg\in\mathsf{SO}(3)$.

\item[(2)] Mean identity:
\begin{equation}\label{eq:ED}
D^{(0)}(\frakg)=1, \qquad \E_{\frakg}[D^{(l)}(\frakg)]=0 \text{ for all } l \geq 1.
\end{equation}

\item[(3)] Orthogonality: for any $l,l'\geq 0$ and $-l\leq q,m\leq l$
and $-l'\leq q',m'\leq l'$,
\begin{equation}\label{eq:WignerDorthog}
\E_\frakg\left[D^{(l)}_{qm}(\frakg)D^{(l')}_{q'm'}(\frakg)\right]
=\frac{(-1)^{m+q}}{2l+1}\1\{l=l',q=-q',m=-m'\}.
\end{equation}

\item[(4)] Product of two complex Wigner D-matrices: for any $l,l'\geq 0$ and $-l\leq q,m\leq l$ and $-l'\leq q',m'\leq l'$ and $\frakg\in\mathsf{SO}(3)$,
\begin{align}\label{eq:Wignerproduct}
D^{(l)}_{qm}(\frakg) D^{(l')}_{q'm'}(\frakg)=\sum_{l''=|l-l'|}^{l+l'} C_{q,q',q+q'}^{l,l',l''} C_{m,m',m+m'}^{l,l',l''} D_{q+q',m+m'}^{(l'')}(\frakg),
\end{align}
where $C_{m,m',m''}^{l,l',l''}$ is a complex Clebsch-Gordan coefficient. Its explicit form can be found in \cite[Eq.\ (2.41)]{bohm2013quantum} and \cite[Appendix A.2]{bandeira2023estimation}. Of note, all complex Clebsch-Gordan coefficients are real.

\item[(5)] Third order identity: for any $l,l',l''\geq 0$ and
$-l\leq q,m\leq l$ and $-l'\leq q',m'\leq l'$ and $-l''\leq q'',m''\leq l''$,
\begin{align}\label{eq:WignerDtriple}
&\hspace{0.1in}\E_\frakg\left[D_{qm}^{(l)}(\frakg)D_{q'm'}^{(l')}(\frakg)D_{q''m''}^{(l'')}(\frakg)\right]\nonumber\\
&=\1\{q+q'=-q''\} \cdot \1\{m+m'=-m''\} \cdot \1\{|l-l'| \leq l'' \leq l+l'\}
\nonumber\\
&\hspace{0.3in}\cdot \frac{(-1)^{m''+q''}}
{2l''+1} C_{q,q',-q''}^{l,l',l''} C_{m,m',-m''}^{l,l',l''}.
\end{align}

\end{enumerate}

\subsection{Real Spherical Harmonics and Real Wigner D-Matrices}

\subsubsection{Real Spherical Harmonics}\label{sec:real_SH}
A \emph{real} basis of \emph{spherical harmonics} $\mathbcal{y}_{lm}:\mathcal{S}^2\to \R$ can be defined in terms of the complex analogues
\begin{align}\label{eq:realharmonics}
\mathbcal{y}_{lm}=\begin{cases}
        \frac{\i}{\sqrt{2}}\cdot (y_{lm}-(-1)^m y_{l,-m}) &  \text{ if } m < 0, \\
        y_{l0} & \text{ if } m = 0, \\
        \frac{1}{\sqrt{2}}\cdot (y_{l,-m}+(-1)^m y_{lm}) &  \text{ if } m > 0. \\
    \end{cases}  
\end{align}
Similarly, writing $\mathbcal{y}_l=(\mathbcal{y}_{lm}:-l\leq m\leq l)\in\R^{2l+1}$, we have the relationship in the matrix form $\mathbcal{y}_l= Q^l y_l$ where $Q^l\in\C^{(2l+1)\times (2l+1)}$ is the change of basis matrix from complex to real spherical harmonics basis at frequency $l$. For any $l>0$, $Q^l$ is a unitary matrix having non-zero entries only on the main diagonal and anti-diagonal, and its non-zero elements are explicitly given by 
\begin{equation}\label{eq:Q_values}
    Q^l_{mm} = \begin{cases}
        \i/\sqrt{2} &  \text{ if } m < 0, \\
        1 & \text{ if } m = 0, \\
        (-1)^m/\sqrt{2} &  \text{ if } m > 0, \\
    \end{cases}  
    \quad
        Q^l_{-m,m} = \begin{cases}
        1/\sqrt{2} &  \text{ if } m < 0, \\
        1  & \text{ if } m = 0, \\
        -(-1)^m \i/\sqrt{2} &  \text{ if } m > 0. \\
    \end{cases}  
\end{equation}
The real spherical harmonics form a complete orthonormal basis for $L_2(\mathcal{S}^2,\R)$. 


\subsubsection{Real Wigner D-Matrices}
The \emph{real Wigner D-matrix}\label{sec:real_Wigner}
$\mathcal{D}^{(l)}(\frakg)\in\R^{(2l+1)\times (2l+1)}$ at frequency $l$ corresponding to $\frakg\in\mathsf{SO}(3)$ can be defined through the rotation transformation applied to real spherical harmonics similar to \eqref{eq:complexWigner}
\begin{align*}
\mathbcal{y}_l(\frakg^{-1} \cdot (\varphi_1,\varphi_2))= (\mathcal{D}^{(l)}(\frakg))^\T \mathbcal{y}_l(\varphi_1,\varphi_2) \text{ for each } l=0,1,2,\ldots.
\end{align*}
Equivalently, the real Wigner D-matrices can be written in terms of the complex analogues
\begin{align}\label{eq:realWigner}
\mathcal{D}^{(l)}(\frakg)=\overline{Q^l} D^l(\frakg) (Q^l)^{\T} \text{ for each } l=0,1,2,\ldots.
\end{align}
Notably, the real Wigner D-matrices are orthogonal. Similarly, by Peter-Weyl Theorem, the set of all the entries of real Wigner D-matrices form a complete orthogonal basis for $L_2(\mathsf{SO}(3),\R)$.


\subsubsection{Properties}\label{sec:real_properties}
Let the following expectations be taken with respect to the Haar probability measure on $\mathsf{SO}(3)$. We have the following properties for real Wigner D-matrices.

\begin{enumerate}
\item[(1)] Group homomorphism: for any $l\geq 0$ and $\frakg_1,\frakg_2\in \mathsf{SO}(3)$,
\[\mathcal{D}^{(l)}(\frakg_1 \frakg_2)=\mathcal{D}^{(l)}(\frakg_1)\mathcal{D}^{(l)}( \frakg_2).\]
It follows immediately that $\mathcal{D}^{(l)}(\frakg)$ is orthogonal for any $l\geq 0$ and $\frakg\in\mathsf{SO}(3)$.

\item[(2)] Mean identity:
\begin{equation}\label{eq:realED}
\mathcal{D}^{(0)}(\frakg)=1, \qquad \E_{\frakg}[\mathcal{D}^{(l)}(\frakg)]=0 \text{ for all } l \geq 1.
\end{equation}

\item[(3)] Orthogonality: for any $l,l'\geq 0$ and $-l\leq q,m\leq l$
and $-l'\leq q',m'\leq l'$,
\begin{equation}\label{eq:realWignerDorthog}
\E_\frakg\left[\mathcal{D}^{(l)}_{qm}(\frakg)\mathcal{D}^{(l')}_{q'm'}(\frakg)\right]
=\frac{1}{2l+1}\1\{l=l',q=q',m=m'\}.
\end{equation}

\item[(4)] Product of two real Wigner D-matrices: for any $l,l'\geq 0$ and $-l\leq q,m\leq l$
and $-l'\leq q',m'\leq l'$ and $\frakg\in\mathsf{SO}(3)$,
\begin{align}\label{eq:real Wignerproduct}
\mathcal{D}^{(l)}_{qm}(\frakg)\mathcal{D}^{(l')}_{q'm'}(\frakg)=\sum_{l''=|l-l'|}^{l+l'} 
\sum_{q''\in A(q,q')} \sum_{m''\in A(m,m')}
\mathcal{C}_{q,q',q''}^{l,l',l''} \overline{\mathcal{C}_{m,m',m''}^{l,l',l''}} \mathcal{D}_{q'' m''}^{l''}(\frakg),
\end{align}
where $\mathcal{C}_{q,q',q''}^{l,l',l''}$ is a real Clebsch-Gordan coefficient defined as
\begin{align*}
\mathcal{C}_{q,q',q''}^{l,l',l''}= \sum_{p=-l}^{l} \sum_{p'=-l'}^{l'} \sum_{p''=-l''}^{l''} \overline{Q^{l}_{qp}} \overline{Q^{l'}_{q'p'}} Q^{l''}_{q''p''} C_{p,p',p''}^{l,l',l''},
\end{align*}
and
\begin{align*}
A(q,q')=&\{q+q', q-q', q'-q, -q-q'\}, \\
A(m,m')=&\{m+m', m-m', m'-m, -m-m'\}.
\end{align*}
Notably, although the real Clebsch-Gordan coefficients are complex in general, they are either real or pure imaginary. In particular, for $q''\in A(q,q')$, there exist some real constant $C$ depending only on $l,l',l'',q,q',q''$ such that
\begin{align*}
    \mathcal{C}_{q,q',q''}^{l,l',l''}= C\cdot (1+(-1)^{l+l'-l''}) ~{\rm or}~ C\cdot(1-(-1)^{l+l'-l''})\i.
\end{align*}

\item[(5)] Third order identity: we can have a similar property as that of the complex analogues. We omit it here for simplicity.
\end{enumerate}

\section{Function Bases and Rotation Representation in Cryo-ET and Cryo-EM}\label{sec:function_basis}

This appendix provides the necessary background on function bases and rotation representations used in Appendix~\ref{sec:additional_examples}.

\subsection{Function Basis and Rotation Representation in Unprojected Cryo-EM and Cryo-ET}\label{sec:basis_rotation_cryoET}

\subsubsection{Function Basis}\label{sec:function_basis_cryoET}
For $f\in L_2(\R^3,\C)$, denote its Fourier transform
\begin{align}\label{eq:3D_FT}
\hat{f}(k_1,k_2,k_3)=\int_{\R^3} e^{-2\pi\i (k_1x_1+k_2x_2+k_3x_3)} f(x_1,x_2,x_3)\der x_1 \der x_2 \der x_3.
\end{align}
We reparametrize both Cartesian coordinates $x=(x_1,x_2,x_3)\in\R^3$ in the original function domain and $k=(k_1,k_2,k_3)\in\R^3$ in the Fourier domain by spherical coordinates $(r,\varphi_1,\varphi_2)$ with radius $r\geq 0$, latitude $\varphi_1\in[0,\pi]$, and longitude $\varphi_2\in[0,2\pi)$. With a slight abuse of notation, we write $f(r,\varphi_1,\varphi_2)$ and $\hat{f}(r,\varphi_1,\varphi_2)$ for this parametrization. 

We model the signal of interest as a function in $L_2(\R^3,\R)$ in the example of unprojected cryo-EM and cryo-ET. We describe the function bases for both $L_2(\R^3,\C)$ and $L_2(\R^3,\R)$ used in Section \ref{sec:model_cryoET}, following the basis choice in \cite{fan2024maximum}. Let $\{\hat{h}_{lsm}\}$ denote a set of complex functions, where each function is the product of the complex spherical harmonics $y_{lm}(\varphi_1,\varphi_2)$ (see Appendix \ref{sec:complex_SH}) and radial functions $j_s(r)$:
\begin{align*}
\hat{h}_{lsm}(r,\varphi_1,\varphi_2) = j_s(r) y_{lm}(\varphi_1,\varphi_2)~~~\text{for}~s\geq 1,~l\geq 0,~m=-l,-l+1,\ldots,l.
\end{align*}
Here, $\{j_s:s\geq 1\}$ may be any complete orthonormal system of radial basis functions $j_s:[0,\infty)\to \R$ for square-integrable functions with respect to the measure $r^2\der r$. That is,
\begin{align}\label{eq:orth_cryoET}
\int_0^{\infty} j_s(r) j_{s'}(r) r^2\der r = \1\{s=s'\}.
\end{align}
Using the spherical change-of-coordinates $\der x_1 \der x_2 \der x_3 = r^2 \sin\varphi_1 \der r\der \varphi_1 \der \varphi_2$ and the completeness of complex spherical harmonics for $L_2(\mathcal{S}^2,\C)$, these functions $\{\hat{h}_{lsm}\}$ form a complete orthonormal basis for $L_2(\R^3,\mathbb{C})$. Then so are their inverse Fourier transform $\{h_{lsm}\}$, by the Parseval relation. Consequently, any $f\in L_2(\R^3,\C)$ can be expressed as a series expansion in terms of $\{h_{lsm}\}$:
\begin{align*}
f=\sum_{s=1}^\infty \sum_{l=0}^{\infty}\sum_{m=-l}^l u_{m}^{(ls)} h_{lsm},
\end{align*}
where $u_{m}^{(ls)}$ are the expansion coefficients. Similarly, by linearity of the Fourier transform, its Fourier transform can be expanded in terms of $\{\hat{h}_{lsm}\}$: 
\begin{align*}
\hat{f}=\sum_{s=1}^\infty \sum_{l=0}^{\infty}\sum_{m=-l}^l u_{m}^{(ls)} \hat{h}_{lsm}.
\end{align*}

Note that a function $f\in L_2(\R^3,\C)$ is real-valued if and only if its Fourier transform satisfies
\begin{align*}
\hat{f}(r,\varphi_1,\varphi_2) = \overline{\hat{f}(r,\pi-\varphi_1,\pi+\varphi_2)}.
\end{align*}
Applying the symmetry property of the complex spherical harmonics $y_{lm}(\pi-\varphi_1,\pi+\varphi_2)=(-1)^{l+m} \overline{y_{l,-m}(\varphi_1,\varphi_2)}$ (see \cite[Appendix D.3.1]{fan2024maximum}), it may be shown that this condition is equivalent to the sign symmetry of the expansion coefficients 
\begin{align}\label{eq:real_symm_coef}
u_m^{(ls)}=(-1)^{l+m} \overline{u_{-m}^{(ls)}}
\end{align}
in the basis representations above. Thus, we could define a real basis $\{\mathbcal{h}_{lsm}\}$
\begin{align}\label{eq:complex_real_transformation}
\mathbcal{h}_{lsm}=\begin{cases}
        \frac{\i}{\sqrt{2}}\cdot (h_{lsm}-(-1)^{l+m} h_{l,s,-m}) &  \text{ if } m < 0, \\
        \i^l\cdot h_{ls0} & \text{ if } m = 0, \\
        \frac{1}{\sqrt{2}}\cdot (h_{l,s,-m}+(-1)^{l+m} h_{lsm}) &  \text{ if } m > 0. \\
    \end{cases}  
\end{align}
It can be shown $\mathbcal{h}_{lsm}$ satisfies \eqref{eq:real_symm_coef} for its coefficients $u_m^{(ls)}$ in the basis $\{h_{lsm}\}$, and hence is real-valued. Thus, $\{\mathbcal{h}_{lsm}\}$ form a complete orthonormal basis for $L_2(\R^3,\R)$. Consequently, any $f\in L_2(\R^3,\R)$ can be expressed in the following series expansions:
\begin{align*}
f=\sum_{s=1}^\infty \sum_{l=0}^{\infty}\sum_{m=-l}^l u_{m}^{(ls)} h_{lsm} = \sum_{s=1}^\infty \sum_{l=0}^{\infty}\sum_{m=-l}^l \theta_{m}^{(ls)} \mathbcal{h}_{lsm},
\end{align*}
where $\theta_m^{(ls)}$ are real free coefficients and $u_m^{(ls)}$ are complex coefficients satisfying the sign symmetry \eqref{eq:real_symm_coef}. For $f$ having a bandlimited representation in these bases, we can use the same index set $\mathcal{I}$ to express it in terms of both the complex and real bases, leading to
\begin{align*}
f=\sum_{(l,s,m)\in\mathcal{I}} u_{m}^{(ls)} h_{lsm} = \sum_{(l,s,m)\in\mathcal{I}} \theta_{m}^{(ls)} \mathbcal{h}_{lsm},
\end{align*}
where the index set $\mathcal{I}$ is defined in \eqref{eq:index_set}, with the bandlimit $L$ of the spherical frequency, and the bandlimit $S_l$ of the radial frequency for each spherical component $l=0,\ldots, L$. This corresponds to the $(L,S_0,\ldots,S_L)$-bandlimited real-valued function with respect to the basis $\{h_{lsm}\}$ as defined in Section \ref{sec:model_cryoET}, which serves as the primary assumption of our observational model. Furthermore, the vectors of coefficients $u^{(ls)}=(u_m^{(ls)}:-l\leq m\leq l\}\in\C^{2l+1}$ and $\theta^{(ls)}=(\theta_m^{(ls)}:-l\leq m\leq l\}\in\R^{2l+1}$ are related by a unitary transform $u^{(ls)}=(\check{Q}^{(ls)})^\T \theta^{(ls)}$ defined as 
\begin{align*}
u_m^{(ls)}=\begin{cases}
        \frac{1}{\sqrt{2}}\cdot (\theta_{|m|}^{(ls)}+\i \theta_{-|m|}^{(ls)}) &  \text{ if } m < 0, \\
        \i^l\cdot \theta_0^{(ls)} & \text{ if } m = 0, \\
        \frac{(-1)^{l+m}}{\sqrt{2}}\cdot (\theta_{|m|}^{(ls)}-\i \theta_{-|m|}^{(ls)}) &  \text{ if } m > 0. \\
    \end{cases}  
\end{align*}
If we order these coefficients according to the tuple $(l,s,m)$ (also see Section \ref{sec:model_cryoET}), we have the unitary transformation $u=\check{Q}^{\T}\theta$ with
\begin{align}\label{eq:Qtilde}
\check{Q}=\bigoplus_{l=0}^L \bigoplus_{s=1}^{S_l} \check{Q}^{(ls)}.
\end{align}

\subsubsection{Rotation Representation}\label{sec:rot_representation_cryoET}
Notably, due to the commutativity of the rotation operator and the Fourier transform, applying a rotation $\frakg\in\SO(3)$ to $f$ corresponds to the same rotation being applied to its Fourier transform $\hat{f}$. In other words,  
\begin{align*}
f_\frakg(x)=f(\frakg^{-1}\cdot x) ~~~\text{is equivalent to}~~~\hat{f}_\frakg(k)=\hat{f}(\frakg^{-1}\cdot k).
\end{align*}
For a $(L,S_0,\ldots,S_L)$-bandlimited function $f$, its Fourier transform has the basis expansion
\begin{align*}
\hat{f}(r,\varphi_1,\varphi_2)=&\sum_{(l,s,m)\in\mathcal{I}} u_{m}^{(ls)} \hat{h}_{lsm}\\
=&\sum_{l=0}^L\sum_{s=1}^{S_l}  j_s(r) \hat{f}_{ls}(\varphi_1,\varphi_2)
\end{align*}
with the $(l,s)$-th spherical shell component $\hat{f}_{ls}$ defined as
\begin{align*}
\hat{f}_{ls}(\varphi_1,\varphi_2)=\sum_{m=-l}^l  u_{m}^{(ls)} y_{lm}(\varphi_1,\varphi_2).
\end{align*}
Then the rotation action $\frakg$ acts separately on each $\hat{f}_{ls}$, and the vector of coefficients $u^{(ls)}$ transforms according to the map $u^{(ls)} \mapsto D^{(l)}(\frakg) 
\cdot u^{(ls)}$ for each $0\leq l\leq L$ and $1\leq s\leq S_l$ (see Appendix \ref{sec:complex_SH}). Applying the unitary relation $u^{(ls)}=(\check{Q}^{(ls)})^\T \theta^{(ls)}$, this rotation induces the transformation $\theta^{(ls)}\mapsto \check{\mathcal{D}}^{(ls)}(\frakg) \cdot \theta^{(ls)}$ on the vector of real coefficients where $\check{\mathcal{D}}^{(ls)}(\frakg)$ is an orthogonal matrix defined as 
\begin{align}
\check{\mathcal{D}}^{(ls)}(\frakg)=\overline{\check{Q}^{(ls)}} D^l(\frakg ) (\check{Q}^{(ls)})^\T.
\end{align}
It follows that the rotation representation for the whole vector $\theta$ of the real coefficients has the block-diagonal structure:
\begin{align*}
\check{\mathcal{D}}(\frakg)=\bigoplus_{l=0}^L \bigoplus_{s=1}^{S_l} \check{\mathcal{D}}^{(ls)}(\frakg).
\end{align*}

\subsection{Function Basis and Rotation Representation in (Projected) Cryo-EM}\label{sec:basis_rotation_cryoEM}

\subsubsection{Function Basis}\label{sec:function_basis_cryoEM}

For $f\in L_2(\R^3,\C)$, the tomographic projection operator $\Pi$ acts on $f$, yielding $\Pi\circ f \in L_2(\R^2,\C)$. Let $\hat{f}$ be its 3-D Fourier transform as defined in \eqref{eq:3D_FT} and denote the 2-D Fourier transform of its tomographic projection
\begin{align*}
\widehat{\Pi\circ f}(k_1,k_2)=\int_{\R^2} e^{-2\pi\i (k_1x_1+k_2x_2)} (\Pi\circ f)(x_1,x_2)\der x_1 \der x_2.
\end{align*}
We reparametrize the 3-D space using spherical coordinates as before. Similarly, we reparametrize both Cartesian coordinates $x=(x_1,x_2)\in\R^2$ in the projection domain and $k=(k_1,k_2)\in\R^2$ in the Fourier domain by polar coordinates $(r,\varphi_2)$ with radius $r\geq 0$ and angle $\varphi_2\in[0,2\pi)$. With a slight abuse of notation, we denote the reparametrized functions as $(\Pi \circ f)(r,\varphi_2)$ and $\widehat{\Pi \circ f}(r,\varphi_2)$. 

We continue to model the signal of interest as a function in $L_2(\R^3,\R)$ and model its tomographic projection as a function in $L_2(\R^2,\R)$. The basis functions we use for $L_2(\R^3,\C)$ and $L_2(\R^3,\R)$ have radial components that are slightly different from those described in Appendix \ref{sec:function_basis_cryoET}. Corresponding bases for $L_2(\R^2,\C)$ and $L_2(\R^2,\R)$ are introduced, naturally connected to the 3-D bases through the Fourier slice theorem (also see \cite{fan2024maximum}). Let $\{\hat{\tilde{h}}_{lsm}\}$ denote a set of complex functions in $L_2(\R^3,\C)$, where each function is the product of the complex spherical harmonics $y_{lm}(\varphi_1,\varphi_2)$ and the radial functions $\tilde{j}_s(r)$:
\begin{align*}
\hat{\tilde{h}}_{lsm}(r,\varphi_1,\varphi_2) = \tilde{j}_s(r) y_{lm}(\varphi_1,\varphi_2)~~~\text{for}~s\geq 1,~l\geq 0,~m=-l,-l+1,\ldots,l.
\end{align*}
Let $\{\hat{q}_{sm}\}$ denote a set of complex functions in $L_2(\R^2,\C)$, where each function is the product of the Fourier basis function and the same radial functions $\tilde{j}_s(r)$:
\begin{align*}
\hat{q}_{sm}(r,\varphi_2) = \tilde{j}_s(r) \cdot (2\pi)^{-1/2} \exp(\i m \varphi_2)~~~\text{for}~s\geq 1,~m\in\mathbb{Z}.
\end{align*}
Here, compared to the orthogonality relation in \eqref{eq:orth_cryoET}, $\{\tilde{j}_s:s\geq 1\}$ are chosen as any complete orthonormal system of radial basis functions $\tilde{j}_s:[0,\infty)\to \R$, defined for square-integrable functions with respect to the measure $r\der r$. Specifically,
\begin{align}\label{eq:orth_cryoEM}
\int_0^{\infty} \tilde{j}_s(r) \tilde{j}_{s'}(r) r\der r = \1\{s=s'\}.
\end{align}
The completeness and orthogonality ensures that $\{\hat{q}_{sm}\}$ form a complete orthonormal basis in $L_2(\R^2,\C)$. We write $\{q_{sm}\}$ for the 2-D inverse Fourier transform of $\{\hat{q}_{sm}\}$, and $\{\tilde{h}_{lsm}\}$ for the 3-D inverse Fourier transform of $\{\hat{\tilde{h}}_{lsm}\}$. Then $\{q_{sm}\}$ form a complete orthonormal basis in $L_2(\R^2,\R)$ by the Parseval relation.

For any function $f\in L_2(\R^3,\C)$, the Fourier slice theorem provides the connection in Cartesian coordinates:
\begin{align*}
\widehat{\Pi \circ f} (k_1,k_2) = \hat{f}(k_1,k_2, 0).
\end{align*}
In spherical coordinates $(r,\varphi_1,\varphi_2)$ for $\R^3$ and polar coordinates $(r,\varphi_2)$ for $\R^2$, this equivalently translates to
\begin{align}\label{eq:Fourier_slice_spherical}
\widehat{\Pi \circ f}(r,\varphi_2)=\hat{f}(r,\frac{\pi}{2},\varphi_2).
\end{align}
This restriction $\varphi_1=\pi/2$ applied to each complex spherical harmonic $y_{lm}$ yields
\begin{align*}
y_{lm}(\frac{\pi}{2},\varphi_2)=c_{lm}\cdot (2\pi)^{-1/2} \exp(\i m \varphi_2),
\end{align*}
where $c_{lm}$ is explicitly given by
\begin{align*}
c_{lm}=(-1)^m \sqrt{\frac{(2l+1)}{2}\frac{(l-m)!}{(l+m)!}} \cdot P_{lm}(0)
\end{align*}
with $P_{lm}(x)$ being the associated Legendre polynomials defined in \eqref{eq:legendre}. Setting $f=\tilde{h}_{lsm}$ in \eqref{eq:Fourier_slice_spherical}, we obtain
\begin{align*}
\widehat{\Pi \circ \tilde{h}_{lsm}}(r,\varphi_2)=&\hat{\tilde{h}}_{lsm}(r,\frac{\pi}{2},\varphi_2)= \tilde{j}_s(r) \cdot y_{lm}(\frac{\pi}{2},\varphi_2)\\
=& c_{lm} \cdot \tilde{j}_s(r) \cdot (2\pi)^{-1/2} \exp(\i m \varphi_2)\\
=& c_{lm}\cdot \hat{q}_{sm}(r,\varphi_2).
\end{align*}
Taking inverse Fourier transforms, we get
\begin{align}\label{eq:basis_projection_relation}
\Pi \circ \tilde{h}_{lsm}= c_{lm} \cdot q_{sm}.
\end{align}

For a $(L,S_0,\ldots,S_L)$-bandlimited function $f$ with respect to the basis $\{\tilde{h}_{lsm}\}$, expressed as
\begin{align*}
f=\sum_{(l,s,m)\in\mathcal{I}} u_{m}^{(ls)} \tilde{h}_{lsm},
\end{align*}
where the index set is defined in \eqref{eq:index_set}, it follows from \eqref{eq:basis_projection_relation} that its tomographic projection is also bandlimited in the basis $\{q_{sm}\}$, namely
\begin{align*}
\Pi \circ f = \sum_{(s,m)\in \tilde{\mathcal{I}}} \tilde{u}_m^{(s)} q_{sm}, 
\end{align*}
where the index set is defined in \eqref{eq:index_set_projection}. The coefficients $\tilde{u}_m^{(s)}$ are determined by the relation
\begin{align}\label{eq:coef_projection_relation}
\tilde{u}_m^{(s)}= \sum_{\substack{l=|m| \\ S_l\geq s}}^L c_{lm}\cdot u_{m}^{(ls)}.
\end{align}
To work with a real-valued $f$, applying the same transformation rule as in \eqref{eq:complex_real_transformation}, we define a real basis $\{\tilde{\mathbcal{h}}_{lsm}\}$ from $\{\tilde{h}_{lsm}\}$, such that 
\begin{align}
f=\sum_{(l,s,m)\in\mathcal{I}} u_{m}^{(ls)} \tilde{h}_{lsm}=\sum_{(l,s,m)\in\mathcal{I}} \theta_{m}^{(ls)} \tilde{\mathbcal{h}}_{lsm}.
\end{align}
Here, the coefficients $u=(u_m^{(ls)}:(l,s,m)\in\mathcal{I})\in \C^d$ for the complex basis and $\theta=(\theta_m^{(ls)}:(l,s,m)\in\mathcal{I})\in\R^d$ for the real basis are related via the same unitary transform $u=\check{Q}^\T \theta$ as defined in \eqref{eq:Qtilde}. Both $u$ and $\theta$ are ordered according to the tuple $(l,s,m)$, as described in Appendix \ref{sec:function_basis_cryoET}. For the projected function space, we similarly define a real basis $\{\mathbcal{q}_{sm}\}$ from $\{q_{sm}\}$ by 
\begin{align*}
\mathbcal{q}_{sm}=\begin{cases}
        \frac{\i}{\sqrt{2}}\cdot (q_{sm}-(-1)^m q_{s,-m}) &  \text{ if } m < 0, \\
        q_{s0} & \text{ if } m = 0, \\
        \frac{1}{\sqrt{2}}\cdot (q_{s,-m}+(-1)^m q_{sm}) &  \text{ if } m > 0, \\
    \end{cases}  
\end{align*}
such that the real-valued $\Pi \circ f$ can be expressed as
\begin{align*}
\Pi \circ f = \sum_{(s,m)\in \tilde{\mathcal{I}}} \tilde{u}_m^{(s)} q_{sm}=\sum_{(s,m)\in \tilde{\mathcal{I}}} \tilde\theta_m^{(s)} \mathbcal{q}_{sm}.
\end{align*}
Here, the coefficients $\tilde{u}=(\tilde{u}_m^{(s)}:(s,m)\in\mathcal{I})\in\C^{\td}$ for the complex basis and $\tilde{\theta}=(\tilde\theta_m^{(s)}:(s,m)\in\mathcal{I})\in\R^{\td}$ for the real basis, both ordered according to the tuple $(s,m)$, are related by a unitary transform $\tilde{u}=\tilde{Q}^\T \tilde\theta$ defined as
\begin{align*}
\tilde{u}_m^{(s)}=\begin{cases}
        \frac{1}{\sqrt{2}}\cdot (\tilde\theta_{|m|}^{(s)}+\i \tilde\theta_{-|m|}^{(s)}) &  \text{ if } m < 0, \\
         \tilde\theta_0^{(s)} & \text{ if } m = 0, \\
        \frac{(-1)^{m}}{\sqrt{2}}\cdot (\theta_{|m|}^{(s)}-\i \theta_{-|m|}^{(s)}) &  \text{ if } m > 0. \\
    \end{cases}  
\end{align*}
For the real-valued $\Pi \circ f$, its complex coefficients $\tilde{u}_m^{(s)}$ satisfy a sign symmetry analogous to that in the 3-D case \eqref{eq:real_symm_coef}:
\begin{align*}
\tilde{u}_m^{(s)}= (-1)^m \overline{\tilde{u}_{-m}^{(s)}}.
\end{align*}
Following the ordering of the coefficients described above, the relation \eqref{eq:coef_projection_relation} translates into a linear map $P^{\C}:\C^d\to \C^{\td}$ that acts on the complex coefficients, defined as
\begin{align}\label{eq:tomographic_proj_representation_complex}
\tilde{u}=P^{\C}u,~~~~~~P^{\C}_{(s',m'),(l,s,m)}=\1\{s=s'\}\cdot\1\{m=m'\}\cdot c_{lm}.
\end{align}
Combining with the unitary transforms $u=\check{Q}^\T \theta$ and $\tilde{u}=\tilde{Q}^\T \tilde\theta$, the tomographic projection $\Pi$ corresponds to a linear map $P:\R^d\to\R^{\td}$ that acts on the real coefficients, defined as
\begin{align}\label{eq:tomographic_proj_representation}
\tilde\theta=P \theta,~~~~~~P = \overline{\tilde{Q}} P^{\C} \check{Q}^{\T}.
\end{align}

\subsubsection{Rotation Representation}\label{sec:rot_representation_cryoEM}
The only difference in modeling $f$ compared to Appendix \ref{sec:function_basis_cryoET} lies in the choice of the radial basis functions. The rotation $\frakg\in\SO(3)$ applied to $f$ does not affect the radial component, so the rotation representation remains exactly the same as that in Appendix \ref{sec:rot_representation_cryoET}.

\section{Auxiliary Lemmas}

\begin{lemma}\label{lem:chi_square_tail}
Let $Z\sim \mathrm{N}(0,\Id_d)$ be a standard Gaussian random vector in $\R^d$. Then for any $t\geq d$, we have
\begin{align*}
    \P(\|Z\|^2\geq 5t)\leq \exp(-t).
\end{align*}
\end{lemma}
\begin{proof}
Note that $\|Z\|^2$ follows the chi-squared distribution with $d$ degrees of freedom. We have for any $t\geq d$, 
\begin{align*}
\P(\|Z\|^2\geq 5t) \leq \P(\|Z\|^2 \geq d+2\sqrt{dt} +2t )\leq \exp(-t),
\end{align*}
where the last inequality is by Lemma 1 of \cite{laurent2000adaptive}.
\end{proof}

\end{appendix}

\paragraph{Acknowledgments.} The authors would like to thank Dan Edidin, Will Sawin, and Ziquan Zhuang for helpful discussions on this work.

\paragraph{Funding.} Amit Singer and Sheng Xu  were supported in part by AFOSR FA9550-23-1-0249, the Simons Foundation Math+X Investigator Award, NSF DMS 2009753, and NIH/NIGMS R01GM136780-01. Anderson Ye Zhang was supported in part by NSF DMS 2112988.

\bibliographystyle{plain} 
\bibliography{main}

\begin{thebibliography}{10}

\bibitem{abbe2018multireference}
Emmanuel Abbe, Tamir Bendory, William Leeb, Jo{\~a}o~M Pereira, Nir Sharon, and
  Amit Singer.
\newblock Multireference alignment is easier with an aperiodic translation
  distribution.
\newblock {\em IEEE Transactions on Information Theory}, 65(6):3565--3584,
  2018.

\bibitem{abbe2018estimation}
Emmanuel Abbe, Joao~M Pereira, and Amit Singer.
\newblock Estimation in the group action channel.
\newblock In {\em 2018 IEEE International Symposium on Information Theory
  (ISIT)}, pages 561--565. IEEE, 2018.

\bibitem{amini2013pseudo}
Arash~A Amini, Aiyou Chen, Peter~J Bickel, and Elizaveta Levina.
\newblock Pseudo-likelihood methods for community detection in large sparse
  networks.
\newblock {\em The Annals of Statistics}, 41(4):2097–2122, 2013.

\bibitem{baldwin2020non}
Philip~R Baldwin and Dmitry Lyumkis.
\newblock Non-uniformity of projection distributions attenuates resolution in
  cryo-em.
\newblock {\em Progress in biophysics and molecular biology}, 150:160--183,
  2020.

\bibitem{bandeira2020optimal}
Afonso Bandeira, Jonathan Niles-Weed, and Philippe Rigollet.
\newblock Optimal rates of estimation for multi-reference alignment.
\newblock {\em Mathematical Statistics and Learning}, 2(1):25--75, 2020.

\bibitem{bandeira2023estimation}
Afonso~S Bandeira, Ben Blum-Smith, Joe Kileel, Jonathan Niles-Weed, Amelia
  Perry, and Alexander~S Wein.
\newblock Estimation under group actions: recovering orbits from invariants.
\newblock {\em Applied and Computational Harmonic Analysis}, 2023.

\bibitem{bandeira2020non}
Afonso~S Bandeira, Yutong Chen, Roy~R Lederman, and Amit Singer.
\newblock Non-unique games over compact groups and orientation estimation in
  cryo-em.
\newblock {\em Inverse Problems}, 36(6):064002, 2020.

\bibitem{bendory2020single}
Tamir Bendory, Alberto Bartesaghi, and Amit Singer.
\newblock Single-particle cryo-electron microscopy: Mathematical theory,
  computational challenges, and opportunities.
\newblock {\em IEEE signal processing magazine}, 37(2):58--76, 2020.

\bibitem{bendory2019multi}
Tamir Bendory, Nicolas Boumal, William Leeb, Eitan Levin, and Amit Singer.
\newblock Multi-target detection with application to cryo-electron microscopy.
\newblock {\em Inverse Problems}, 35(10):104003, 2019.

\bibitem{bendory2024sample}
Tamir Bendory and Dan Edidin.
\newblock The sample complexity of sparse multireference alignment and
  single-particle cryo-electron microscopy.
\newblock {\em SIAM Journal on Mathematics of Data Science}, 6(2):254--282,
  2024.

\bibitem{bendory2022dihedral}
Tamir Bendory, Dan Edidin, William Leeb, and Nir Sharon.
\newblock Dihedral multi-reference alignment.
\newblock {\em IEEE Transactions on Information Theory}, 68(5):3489--3499,
  2022.

\bibitem{bendory2023autocorrelation}
Tamir Bendory, Yuehaw Khoo, Joe Kileel, Oscar Mickelin, and Amit Singer.
\newblock Autocorrelation analysis for cryo-em with sparsity constraints:
  improved sample complexity and projection-based algorithms.
\newblock {\em Proceedings of the National Academy of Sciences},
  120(18):e2216507120, 2023.

\bibitem{bohm2013quantum}
Arno B{\"o}hm.
\newblock {\em Quantum mechanics: foundations and applications}.
\newblock Springer Science \& Business Media, 2013.

\bibitem{boyle2013angular}
Michael Boyle.
\newblock Angular velocity of gravitational radiation from precessing binaries
  and the corotating frame.
\newblock {\em Physical Review D—Particles, Fields, Gravitation, and
  Cosmology}, 87(10):104006, 2013.

\bibitem{brunel2019learning}
Victor-Emmanuel Brunel.
\newblock Learning rates for gaussian mixtures under group action.
\newblock In {\em Conference on Learning Theory}, pages 471--491. PMLR, 2019.

\bibitem{carragher2019current}
Bridget Carragher, Yifan Cheng, Adam Frost, Robert~M Glaeser, Gabriel~C Lander,
  Eva Nogales, and H-W Wang.
\newblock Current outcomes when optimizing ‘standard’ sample preparation
  for single-particle cryo-em.
\newblock {\em Journal of microscopy}, 276(1):39--45, 2019.

\bibitem{chirikjian2016harmonic}
Gregory~S Chirikjian and Alexander~B Kyatkin.
\newblock {\em Harmonic analysis for engineers and applied scientists: updated
  and expanded edition}.
\newblock Courier Dover Publications, 2016.

\bibitem{dubochet1988cryo}
Jacques Dubochet, Marc Adrian, Jiin-Ju Chang, Jean-Claude Homo, Jean Lepault,
  Alasdair~W McDowall, and Patrick Schultz.
\newblock Cryo-electron microscopy of vitrified specimens.
\newblock {\em Quarterly reviews of biophysics}, 21(2):129--228, 1988.

\bibitem{dwivedi2018theoretical}
Raaz Dwivedi, Nhat Ho, Koulik Khamaru, Martin~J Wainwright, and Michael~I
  Jordan.
\newblock Theoretical guarantees for the em algorithm when applied to
  mis-specified gaussian mixture models.
\newblock In {\em Proceedings of the 32nd international conference on neural
  information processing systems}, pages 9704--9712, 2018.

\bibitem{dwivedi2020singularity}
Raaz Dwivedi, Nhat Ho, Koulik Khamaru, Martin~J Wainwright, Michael~I Jordan,
  and Bin Yu.
\newblock Singularity, misspecification, and the convergence rate of em.
\newblock {\em The Annals of Statistics}, 48(6):3161--3182, 2020.

\bibitem{fan2024maximum}
Zhou Fan, Roy~R Lederman, Yi~Sun, Tianhao Wang, and Sheng Xu.
\newblock Maximum likelihood for high-noise group orbit estimation and
  single-particle cryo-em.
\newblock {\em The Annals of Statistics}, 52(1):52--77, 2024.

\bibitem{fan2023likelihood}
Zhou Fan, Yi~Sun, Tianhao Wang, and Yihong Wu.
\newblock Likelihood landscape and maximum likelihood estimation for the
  discrete orbit recovery model.
\newblock {\em Communications on Pure and Applied Mathematics},
  76(6):1208--1302, 2023.

\bibitem{frank2006three}
Joachim Frank.
\newblock {\em Three-dimensional electron microscopy of macromolecular
  assemblies: visualization of biological molecules in their native state}.
\newblock Oxford university press, 2006.

\bibitem{glaeser2017opinion}
Robert~M Glaeser and Bong-Gyoon Han.
\newblock Opinion: hazards faced by macromolecules when confined to thin
  aqueous films.
\newblock {\em Biophysics reports}, 3:1--7, 2017.

\bibitem{harauz1986exact}
George Harauz and Marin van Heel.
\newblock Exact filters for general geometry three dimensional reconstruction.
\newblock {\em Optik.}, 73(4):146--156, 1986.

\bibitem{henderson1990model}
Richard Henderson, Joyce~M Baldwin, Thomas~A Ceska, Friedrich Zemlin, Erich
  Beckmann, and Kenneth~H Downing.
\newblock Model for the structure of bacteriorhodopsin based on high-resolution
  electron cryo-microscopy.
\newblock {\em Journal of molecular biology}, 213(4):899--929, 1990.

\bibitem{kam1980reconstruction}
Zvi Kam.
\newblock The reconstruction of structure from electron micrographs of randomly
  oriented particles.
\newblock {\em Journal of Theoretical Biology}, 82(1):15--39, 1980.

\bibitem{katsevich2023likelihood}
Anya Katsevich and Afonso~S Bandeira.
\newblock Likelihood maximization and moment matching in low snr gaussian
  mixture models.
\newblock {\em Communications on Pure and Applied Mathematics}, 76(4):788--842,
  2023.

\bibitem{kileel2024fast}
Joe Kileel, Nicholas~F Marshall, Oscar Mickelin, and Amit Singer.
\newblock Fast expansion into harmonics on the ball.
\newblock {\em arXiv preprint arXiv:2406.05922}, 2024.

\bibitem{kleijn2012bernstein}
Bas~JK Kleijn and Aad~W Van~der Vaart.
\newblock The bernstein-von-mises theorem under misspecification.
\newblock {\em Electronic Journal of Statistics}, 6:354--381, 2012.

\bibitem{laurent2000adaptive}
Beatrice Laurent and Pascal Massart.
\newblock Adaptive estimation of a quadratic functional by model selection.
\newblock {\em Annals of statistics}, pages 1302--1338, 2000.

\bibitem{liu2013deformed}
Ying Liu, Xing Meng, and Zheng Liu.
\newblock Deformed grids for single-particle cryo-electron microscopy of
  specimens exhibiting a preferred orientation.
\newblock {\em Journal of structural biology}, 182(3):255--258, 2013.

\bibitem{liu2025overcoming}
Yun-Tao Liu, Hongcheng Fan, Jason~J Hu, and Z~Hong Zhou.
\newblock Overcoming the preferred-orientation problem in cryo-em with
  self-supervised deep learning.
\newblock {\em Nature methods}, 22(1):113--123, 2025.

\bibitem{lyumkis2019challenges}
Dmitry Lyumkis.
\newblock Challenges and opportunities in cryo-em single-particle analysis.
\newblock {\em Journal of Biological Chemistry}, 294(13):5181--5197, 2019.

\bibitem{mityagin2015zero}
Boris Mityagin.
\newblock The zero set of a real analytic function.
\newblock {\em arXiv preprint arXiv:1512.07276}, 2015.

\bibitem{munkrestopology}
James Munkres.
\newblock {\em Topology}.
\newblock Prentice Hall, Inc., 2000.

\bibitem{nguyen2013convergence}
XuanLong Nguyen.
\newblock Convergence of latent mixing measures in finite and infinite mixture
  models.
\newblock {\em The Annals of Statistics}, 41(1):370--400, 2013.

\bibitem{noble2018routine}
Alex~J Noble, Venkata~P Dandey, Hui Wei, Julia Brasch, Jillian Chase,
  Priyamvada Acharya, Yong~Zi Tan, Zhening Zhang, Laura~Y Kim, Giovanna Scapin,
  et~al.
\newblock Routine single particle cryoem sample and grid characterization by
  tomography.
\newblock {\em Elife}, 7:e34257, 2018.

\bibitem{perry2019sample}
Amelia Perry, Jonathan Weed, Afonso~S Bandeira, Philippe Rigollet, and Amit
  Singer.
\newblock The sample complexity of multireference alignment.
\newblock {\em SIAM Journal on Mathematics of Data Science}, 1(3):497--517,
  2019.

\bibitem{pumir2021generalized}
Thomas Pumir, Amit Singer, and Nicolas Boumal.
\newblock The generalized orthogonal procrustes problem in the high noise
  regime.
\newblock {\em Information and Inference: A Journal of the IMA},
  10(3):921--954, 2021.

\bibitem{punjani2017cryosparc}
Ali Punjani, John~L Rubinstein, David~J Fleet, and Marcus~A Brubaker.
\newblock cryosparc: algorithms for rapid unsupervised cryo-em structure
  determination.
\newblock {\em Nature methods}, 14(3):290--296, 2017.

\bibitem{romanov2021multi}
Elad Romanov, Tamir Bendory, and Or~Ordentlich.
\newblock Multi-reference alignment in high dimensions: Sample complexity and
  phase transition.
\newblock {\em SIAM Journal on Mathematics of Data Science}, 3(2):494--523,
  2021.

\bibitem{rose1995elementary}
Morris~Edgar Rose.
\newblock {\em Elementary theory of angular momentum}.
\newblock Courier Corporation, 1995.

\bibitem{rosenthal2003optimal}
Peter~B Rosenthal and Richard Henderson.
\newblock Optimal determination of particle orientation, absolute hand, and
  contrast loss in single-particle electron cryomicroscopy.
\newblock {\em Journal of molecular biology}, 333(4):721--745, 2003.

\bibitem{saxton1982correlation}
WO~Saxton and W\_ Baumeister.
\newblock The correlation averaging of a regularly arranged bacterial cell
  envelope protein.
\newblock {\em Journal of microscopy}, 127(2):127--138, 1982.

\bibitem{scheres2012relion}
Sjors~HW Scheres.
\newblock Relion: implementation of a bayesian approach to cryo-em structure
  determination.
\newblock {\em Journal of structural biology}, 180(3):519--530, 2012.

\bibitem{scheres2012prevention}
Sjors~HW Scheres and Shaoxia Chen.
\newblock Prevention of overfitting in cryo-em structure determination.
\newblock {\em Nature methods}, 9(9):853--854, 2012.

\bibitem{sharon2020method}
Nir Sharon, Joe Kileel, Yuehaw Khoo, Boris Landa, and Amit Singer.
\newblock Method of moments for 3d single particle ab initio modeling with
  non-uniform distribution of viewing angles.
\newblock {\em Inverse Problems}, 36(4):044003, 2020.

\bibitem{sigworth1998maximum}
Fred~J Sigworth.
\newblock A maximum-likelihood approach to single-particle image refinement.
\newblock {\em Journal of structural biology}, 122(3):328--339, 1998.

\bibitem{singer2018mathematics}
Amit Singer.
\newblock Mathematics for cryo-electron microscopy.
\newblock In {\em Proceedings of the International Congress of Mathematicians:
  Rio de Janeiro 2018}, pages 3995--4014. World Scientific, 2018.

\bibitem{tan2017addressing}
Yong~Zi Tan, Philip~R Baldwin, Joseph~H Davis, James~R Williamson, Clinton~S
  Potter, Bridget Carragher, and Dmitry Lyumkis.
\newblock Addressing preferred specimen orientation in single-particle cryo-em
  through tilting.
\newblock {\em Nature methods}, 14(8):793--796, 2017.

\bibitem{taylor2008retrospective}
Kenneth~A Taylor and Robert~M Glaeser.
\newblock Retrospective on the early development of cryoelectron microscopy of
  macromolecules and a prospective on opportunities for the future.
\newblock {\em Journal of structural biology}, 163(3):214--223, 2008.

\bibitem{turk2020promise}
Martin Turk and Wolfgang Baumeister.
\newblock The promise and the challenges of cryo-electron tomography.
\newblock {\em FEBS letters}, 594(20):3243--3261, 2020.

\bibitem{van2000asymptotic}
Aad~W Van~der Vaart.
\newblock {\em Asymptotic statistics}, volume~3.
\newblock Cambridge university press, 2000.

\bibitem{varshalovich1988quantum}
Dmitrii~Aleksandrovich Varshalovich, Anatolij~Nikolaevic Moskalev, and
  Valerii~Kel'manovich Khersonskii.
\newblock {\em Quantum theory of angular momentum}.
\newblock World Scientific, 1988.

\bibitem{wang2019variational}
Yixin Wang and David Blei.
\newblock Variational bayes under model misspecification.
\newblock {\em Advances in Neural Information Processing Systems}, 32, 2019.

\bibitem{watson2024advances}
Abigail~JI Watson and Alberto Bartesaghi.
\newblock Advances in cryo-et data processing: meeting the demands of visual
  proteomics.
\newblock {\em Current Opinion in Structural Biology}, 87:102861, 2024.

\bibitem{white1996estimation}
Halbert White.
\newblock {\em Estimation, inference and specification analysis}.
\newblock Number~22. Cambridge university press, 1996.

\bibitem{zhang2019advances}
Peijun Zhang.
\newblock Advances in cryo-electron tomography and subtomogram averaging and
  classification.
\newblock {\em Current opinion in structural biology}, 58:249--258, 2019.

\end{thebibliography}

\end{document}